\documentclass[a4paper,10pt,leqno]{amsart}
\title{$L^2$-Betti numbers in prime characteristic and a conjecture of Wise}
        \author{Avramidi, Grigori}
        \address{Max Planck Institute for Mathematics\\
Vivatsgasse 7\\53111 Bonn, Germany}
              \email{gavramidi@mpim-bonn..mpg..de}
              \urladdr{ttps://sites.google.com/site/gavramidi/}
        \author{L\"uck, Wolfgang}
        \address{Mathematicians Institut der Universit\"at Bonn\\
                Endenicher Allee 60\\
                53115 Bonn, Germany}
         \email{wolfgang.lueck@him.uni-bonn.de}
          \urladdr{http://www.him.uni-bonn.de/lueck}
         \date{April, 2026}
          \makeatletter
     \@namedef{subjclassname@2020}{%
  \textup{2020} Mathematics Subject Classification}
\makeatother

\usepackage{comment}
\usepackage{hyperref}
\usepackage{color}
\usepackage{pdfsync}
\usepackage{calc}
\usepackage{enumerate,amssymb}
\usepackage{graphicx}
\usepackage{datetime2}
\usepackage[arrow,curve,matrix,tips,2cell]{xy}

\SelectTips{eu}{10}\UseTips\UseAllTwocells
\DeclareMathAlphabet{\matheurm}{U}{eur}{m}{n}
\newcommand{\Groups}{\matheurm{Groups}}

\newcommand{\Groupsinj}{\matheurm{Groups}_{\operatorname{inj}}}

\newcommand{\wLinnellinj}{\matheurm{w.}\;\matheurm{Linnell}_{\operatorname{inj}}}
  \newcommand{\Fields}{\matheurm{Fields}}
  \newcommand{\FieldsQC}{\matheurm{Fields}_{\IQ,\IC}}
  \newcommand{\DivisionRings}{\matheurm{Division}\; \matheurm{Rings}}  
  \newcommand{\Rings}{\matheurm{Rings}}

\DeclareMathOperator{\cok}{cok}
\DeclareMathOperator{\colim}{colim}

\DeclareMathOperator{\cone}{cone}

\DeclareMathOperator{\F}{F}

\DeclareMathOperator{\GL}{GL}

\DeclareMathOperator{\id}{id}
\DeclareMathOperator{\im}{im}

\DeclareMathOperator{\M}{M}

\DeclareMathOperator{\pr}{pr}
\DeclareMathOperator{\res}{res}
\DeclareMathOperator{\rk}{rk}

\DeclareMathOperator{\virt}{virt}

  \newcommand{\IC}{\mathbb{C}}

  \newcommand{\IF}{\mathbb{F}}

  \newcommand{\IQ}{\mathbb{Q}}
  \newcommand{\IR}{\mathbb{R}}

  \newcommand{\IZ}{\mathbb{Z}}

  \newcommand{\calb}{\mathcal{B}}
  \newcommand{\calc}{\mathcal{C}}
  
  \newcommand{\cald}{\mathcal{D}}

  \newcommand{\cale}{\mathcal{E}}
  \newcommand{\calf}{\mathcal{F}}
   
  \newcommand{\calg}{\mathcal{G}}

  \newcommand{\caln}{\mathcal{N}}

  \newcommand{\calp}{\mathcal{P}}
  
  \newcommand{\calRF}{\mathcal{RF}}
  \newcommand{\calRC}{\mathcal{RC}}

    \newcommand{\calu}{\mathcal{U}}

     \newcounter{commentcounter}

     \theoremstyle{plain} \newtheorem{theorem}{Theorem}[section]
      \newtheorem{lemma}[theorem]{Lemma}
     
     \newtheorem{proposition}[theorem]{Proposition}
     \newtheorem{conjecture}[theorem]{Conjecture}
     
      \newtheorem*{theorem*}{Theorem}
     \newtheorem*{theoremA*}{Theorem A} \newtheorem*{theoremB*}{Theorem B}

     \theoremstyle{definition} \newtheorem{definition}[theorem]{Definition}
      \newtheorem{example}[theorem]{Example}
     \newtheorem{question}[theorem]{Question} 
      \newtheorem{remark}[theorem]{Remark}

     \newtheorem*{definition*}{Definition}

     \theoremstyle{remark}

     \makeatletter\let\c@equation=\c@theorem\makeatother

     \newcommand{\version}[1] %marks the date of last editing and compilation
     {\begin{center} last edited on #1\\
         last compiled on \today \ at \DTMcurrenttime.\\
         name of tex-file: \jobname
       \end{center}}

     \newcommand{\ALI}{\textup{ALI}}
     \newcommand{\calALI}{\mathcal{ALI}}
     \newcommand{\LI}{\textup{LI}}
     \newcommand{\RAAG}{\textup{RAAG}}
      \newcommand{\RALI}{\textup{RALI}}
      \newcommand{\RTFN}{\textup{RTFN}}
      \newcommand{\calRALI}{\mathcal{RALI}}
     \newcommand{\RALIRF}{\textup{RALIRF}}

\begin{document}

     \begin{abstract}
       We systematically study $L^2$-Betti numbers in zero and prime
       characteristic and apply them to a conjecture of Wise stating that all towers of
       a finite $2$-complex are non-positive if and only if the second $L^2$-Betti number
       vanishes.
     \end{abstract}

     \maketitle

%%%%%%%%%%%%%%%%%%%%%%%%%%%%%%%%%%%%%%%%%%%%%%%%%%%%%%%%%%%%%%%%%%%%
%%%%%%%%%%%%%%%%%%%%%%%%%% Introduction %%%%%%%%%%%%%%%%%%%%%%%%%%%%%%%%
%%%%%%%%%%%%%%%%%%%%%%%%%%%%%%%%%%%%%%%%%%%%%%%%%%%%%%%%%%%%%%%%%%%%

  \typeout{------------------- Introduction -----------------}
  \section{Introduction}\label{sec:introduction}

\subsection{$L^2$-Betti numbers in prime characteristic}
The $L^2$-Betti numbers of a $CW$-complex $X$ can be defined as von Neumann dimensions of
the homology of $X$ with coefficients in the group von Neumann algebra
$\mathcal N(\pi_1(X))$. In the special case when the fundamental group $\pi_1(X)$ is
residually finite, they can be computed as limits of normalized $\mathbb Q$-Betti numbers
of finite covers, while in the special case when $\pi_1(X)$ is locally indicable, they can
be obtained as skew field dimensions of homology with coefficients in a certain skew field
$\mathcal D_{\mathbb Q[\pi_1(X)]}$ containing the group ring.

The last two approaches admit mod $p$ analogues. This paper follows the skew field
approach, primarily in the context of \RALI-groups, i.e., residually (amenable and locally
indicable) groups.  Note that not every \RALI-group is residually finite and the class of
\RALI-groups is rather large, see for instance Remark~\ref{rem:examples_of_RALI-groups}.
We show that basic and important properties carry over to the mod $p$ case, e.g.,
multiplicativity, the K\"unneth and Euler-Poincar\'e formulas, Poincar\'e duality,
vanishing for mapping tori, and formulas for fibrations and $3$-manifolds.  Further topics
concern monotonicity, approximation results and the fact that $L^2$-Betti numbers in
characteristic zero are less or equal to the ones in characteristic $p$ and that they
agree for almost all primes $p$, selfmaps of aspherical closed manifolds of degree
different from $-1,0,1$ and not divisible by $p$, and comparison to $\mathbb F_p$-homology
growth.

\subsection{$L^2$-Betti numbers and towers of $2$-complexes}
Daniel Wise suggested geometric characterizations for the vanishing of the top $L^2$-Betti
number of a $2$-complex in terms of non-positivity of Euler characteristics of
subcomplexes and more generally of complexes mapped in via tower maps.  A \emph{tower map}
$X\rightarrow Y$ between a pair of $2$-complexes is a map that can be factored as a finite
composition of regular covering maps and embeddings (of complexes). For a family of groups
$\mathcal F$, call $X\rightarrow Y$ an \emph{$\mathcal F$-tower map} if it has a
factorization in which each covering group lies in the family $\mathcal F$. A $2$-complex
$Y$ has \emph{non positive towers} (or \emph{non positive $\mathcal F$-towers}) if for
every tower (or $\mathcal F$-tower) map $X\rightarrow Y$ from a finite, connected
$2$-complex $X$, either $X$ is contractible or $\chi(X)\leq 0$. Here is one version of
Wise's conjecture.

\begin{conjecture}
Suppose that $Y$ is a $2$-complex. The following are equivalent. 
\begin{enumerate}
\item $b^{(2)}_2(\widetilde Y;\mathcal N(\pi_1(Y)))=0$.
\item $Y$ has non-positive towers.
\end{enumerate}
\end{conjecture}

There is some evidence for this conjecture when one restricts to particular types of
$2$-complexes. It is known to be true for one-relator complexes, for spines of aspherical
$3$-manifolds with non-empty boundary, and for quotients of some two-dimensional Davis
complexes (see~\ref{wiseconjectureexamples}).  However, the conjecture appears to be
difficult in general because it would have strong topological consequences: if a
$2$-complex has non-positive $\mathbb Z$-towers, i.e.,
an $\calf$-tower for $\calf$ consisting of the trivial and the infinite cyclic group,
then it is aspherical and has locally
indicable fundamental group. So, since non-positivity is inherited by subcomplexes, Wise's
conjecture implies that subcomplexes of contractible $2$-complexes are aspherical (a
conjecture of Whitehead), have locally indicable fundamental groups (a conjecture of
Howie), and that finite, acyclic subcomplexes are contractible (a conjecture of
Kervaire-Laudenbach).
One can work modulo these conjectures by assuming a priori that the groups involved are
locally indicable. Assuming more, namely that the fundamental groups are \RALI, we get the
following as a special case of Theorem~\ref{the:vanishing_of_the_top_L2-Betti_numbers_and_RALI-towers_d_is_2}.

\begin{theorem}
  Suppose that $X$ and $Y$ are finite, connected $2$-complexes with \RALI\
  fundamental groups. Let $X\rightarrow Y$ be a $\mathbb
  Z$-tower, or more generally a \RALI-tower. Suppose that $b_2^{(2)}(\widetilde Y;\mathcal N(\pi_1(Y)))=0$.

  Then  either  $\chi(X)\leq 0$, or $X$ is contractible.
\end{theorem}
We also obtain a mod $p$ variation of this result (see
Proposition~\ref{pro:vanishing_of_the_top_L2-Betti_numbers_beta_upper(inf,p_n)(Y;F)_and_calRF_p-towers})
\begin{proposition}
  Suppose $X$ and $Y$ are finite, connected $2$-complexes with residually $p$-finite
  fundamental groups. Let $X\rightarrow Y$ be a $\mathbb Z$-tower,  or more generally a $\calRF_p$-tower.
  Suppose that
\[
\inf_{Y'\rightarrow Y}{b_2(Y';\mathbb F_p)\over|Y'\rightarrow Y|}=0,
\]
where the infimum is taken over all regular finite covers $Y'\rightarrow Y$ whose covering
group is a $p$-group.

Then either $\chi(X)\leq 0$, or $X$ is contractible.
\end{proposition}

%-----------------------------------------------------------------------------

\subsection{Acknowledgments}\label{subsec:Acknowledgements}

The paper is funded by the Deutsche Forschungsgemeinschaft (DFG, German Research
Foundation) under Germany's Excellence Strategy \--- GZ 2047/1, Projekt-ID 390685813,
Hausdorff Center for Mathematics at Bonn. The authors
thank Andrei Jaikin-Zapirain for explaining his work to us and also Sam Fisher and Pablo Sanchez Peralta for
their useful comments.

The paper is organized as follows:
\tableofcontents

%%%%%%%%%%%%%%%%%%%%%%%%%%%%%%%%%%%%%%%%%%%%%%%%%%%%%%%%%%%%%%%%%%%%%
%%%%%%%%%%%%%%%%%%%%%%%%%%%%%% Section 2 %%%%%%%%%%%%%%%%%%%%%%%%%%%%%%%%
%%%%%%%%%%%%%%%%%%%%%%%%%%%%%%%%%%%%%%%%%%%%%%%%%%%%%%%%%%%%%%%%%%%%%

\typeout{---------- Section 2:  $L^2$-Betti numbers over group von Neumann algebras  ---------------}

\section{$L^2$-Betti numbers defined over the group von Neumann algebra}%
\label{sec:L2-Betti_numbers_defined_over_the_group_von_Neumann_algebra}

In this section we briefly recall the notion of $L^2$-Betti numbers in the classical
setting defined using the group von Neumann algebra of a group $G$ which was initiated by
Atiyah~\cite{Atiyah(1976)}.

Let $G$ be a group. Then one can consider the complex group ring $\IC G$ and complete it
to the group von Neumann algebra $\caln(G)$ which can be identified with the algebra
$\calb(L^2(G))^G$ of bounded $G$-equivariant operators $L^2(G) \to L^2(G)$. In the sequel
we will consider $\caln(G)$ just as a ring and forget the topology. There is an dimension
function $\dim_{\caln(G)}$ which assigns to every $\caln(G)$-module $M$ an element in
$\IR^{\ge 0} \amalg \{\infty\}$ and has a lot of pleasant features such as Additivity. For
a finitely generated $\caln(G)$-module $M$ one has $\dim_{\caln(G)}(M) < \infty$ and for a
finitely generated projective $\caln(G)$-module $\dim_{\caln(G)}(M)$ is given by the
classical Murray-von Neumann dimension.  More information about $\dim_{\caln(G)}$ can be
found in~\cite{Lueck(1998a)},~\cite{Lueck(1998b)}, and~\cite[Chapter~6]{Lueck(2002)}.

Given an (arbitrary)  $G$-$CW$-complex $X$, one can define its $n$-th $L^2$-Betti number
\begin{equation}
  b_n^{(2)}(X;\caln(G)) = \dim_{\caln(G)}\bigl(H_n(\caln(G) \otimes_{\IZ G} C_*(X))\bigr)
  \quad \in \IR^{\ge 0} \amalg \{\infty\}
\label{b_n_upper_(2)((X;caln(G))}
\end{equation}
where $C_*(X)$ is the cellular $\IZ G$-chain complex. If $X$ is of finite type, or
equivalently, if $X_m/G$ is compact for every $m$-skeleton $X_m$, then we have
$b_n^{(2)}(X;\caln(G)) < \infty$.

Now suppose that $G$ is locally indicable and $F$ is a field satisfying
$\IQ \subseteq F \subseteq \IC$.  Then there is a specific division ring
$\cald(FG \subseteq \calu(G))$ containing $FG$ as subring. It is given by the division
closure of $FG$ in the algebra $\calu(G)$ of affiliated operators, which in turn is the
Ore localization of $\caln(G)$ with respect to the multiplicative subset of $\caln(G)$
consisting on non-zero-divisors. Then

\begin{equation}
  b_n^{(2)}(X;\caln(G))
  = \dim_{\cald(FG \subseteq \calu(G))}\bigl(H_n(\cald(FG\subseteq \calu(G))  \otimes_{\IZ G} C_*(X))\bigr).
\label{b_n:_upper_(2)((X;caln(G))_in_terms_of_cald(FG_subseteq_calu(G))}
\end{equation}

Note that $\dim_{\cald(FG\subseteq \calu(G))}\bigl(H_n(\cald(FG\subseteq \calu(G)) \otimes_{\IZ G} C_*(X))\bigr)$
takes value in $\IZ^{\ge 0} \amalg \{\infty\}$.  So it is not surprizing that the proof
of~\eqref{b_n:_upper_(2)((X;caln(G))_in_terms_of_cald(FG_subseteq_calu(G))} is based on
the proof of the strong Atiyah Conjecture for locally indicable groups by
Jaikin-Zapirain-L\'{o}pez-\'{A}lvarez~\cite[Theorem~1.1]{Jaikin-Zapirain+Lopez-Alvarez(2020)}.
Namely, they show for a locally indicable group and any finitely presented $FG$-module $M$
that $\dim_{\caln(G)}(\caln(G) \otimes_{\IZ G} M)$ is an integer. Having the strong Atiyah
Conjecture at one's disposal, the constructions of $\cald(FG \subseteq \calu(G))$ and
$\calu(G)$ and the proof
of~\eqref{b_n:_upper_(2)((X;caln(G))_in_terms_of_cald(FG_subseteq_calu(G))} can be found
in~\cite[Chapters~8~and~10]{Lueck(2002)}, at least for $F = \IC$.  Here ideas from
Linnell~\cite{Linnell(1991)} are essential.

%%%%%%%%%%%%%%%%%%%%%%%%%%%%%%%%%%%%%%%%%%%%%%%%%%%%%%%%%%%%%%%%%%%%%
%%%%%%%%%%%%%%%%%%%%%%%%%%%%%% Section 3 %%%%%%%%%%%%%%%%%%%%%%%%%%%%%%%%
%%%%%%%%%%%%%%%%%%%%%%%%%%%%%%%%%%%%%%%%%%%%%%%%%%%%%%%%%%%%%%%%%%%%%

\typeout{---------- Section 2:  $L^2$-Betti numbers in arbitrary characteristic  ---------------}

\section{$L^2$-Betti numbers in arbitrary characteristic}%
\label{sec:L2-Betti_numbers_in_arbitrary_characteristic}

There is a program by Jaikin-Zapirain~\cite{Jaikin-Zapirain(2019positive),
  Jaikin-Zapirain(2021), Jaikin-Zapirain-Linton(2025)} to define the $\IF_p$-version of
the construction and facts presented in
Subsection~\ref{sec:L2-Betti_numbers_defined_over_the_group_von_Neumann_algebra} for all
locally indicable groups $G$. For torsionfree amenable groups this has already been done
by Linnell-L\"uck-Sauer~\cite{Linnell-Lueck-Sauer(2011)}.

% -----------------------------------------------------------------------------

\subsection{Basics on division rings}\label{subsec:basics_on_division_rings}

Given a ring $R$, a \emph{$R$-division ring} $\iota \colon R \to \cald
$ is a ring homomorphism with $R$ as source and some division ring (= skew-field)
$\cald$ as target. We call two $R$-division rings $\iota \colon R \to \cald$ and $\iota'
\colon R \to \cald''$ isomorphic if and only if there is an isomorphism of division rings $f \colon \cald
\xrightarrow{\cong} \cald''$ satisfying $f \circ \iota =
\iota'$.  Note that there is at most one $R$-isomorphism $f$ from $\iota$ to
$\iota'$ if $\iota$ and $\iota'$ are \emph{epic}, i.e., the images of $\iota$ and
$\iota'$ generate $\cald$ and $\cald'$ as division rings.

A division $R$-ring $\iota \colon R \to \cald$ is called \emph{universal} if it is epic
and for every division $R$-ring $\cale$ we have for the associated Sylvester matrix rank
functions $\rk_{\cald} \ge \rk_{\cale}$.  Recall that $\rk_{\cale}$ assigns the natural
number $\dim_{\cale}(\cale \otimes_{R} M)$ to any finitely presented $R$-module $M$.  For
two universal division $R$-rings $\iota \colon R \to \cald$ and
$\iota' \colon R \to \cald'$ there is precisely one isomorphism of $R$-division rings
between them, see Cohn~\cite[Theorem 4.4.1]{Cohn(1995)}.

The universal division $R$-ring $\cald$ is called \emph{universal division ring
  of fractions of $R$},  if $rk_{\cald}$ is faithful, or, equivalently, $\iota \colon R \to \cald$ is injective.

In the sequel let $G$ be a group and $F$ be a field. Denote by $FG$ the associated group
ring or, more generally, by $F \ast G$ some crossed product ring.

Let $\iota \colon F \ast G \to \cald$ be an injective $F\ast G$-division ring. For a subgroup
$H \subseteq G$, we denote the division closure of $\iota(F \ast H)$ in $\cald$ by
$\cald_H$.  We call the injective $F\ast G$ division ring $\iota$ \emph{Linnell} if
$\cald = \cald_G$ holds and the multiplication map
$\mu_H \colon \cald_H \otimes_{FH} FG \to \cald$ sending $x \otimes y$ to
$x \cdot \iota(y)$ is injective for all subgroups $H \subseteq G$. In this situation, we
say that $\cald$ is a \emph{Linnell $F \ast G$-division ring}.  Note that $\iota$ is epic and injective.
If $\iota$ is Linnell, then the restriction
$\iota|_{F \ast H} \colon F \ast H \to \cald_H$ is a Linnell $FH$-division ring for every subgroup
$H \subseteq G$

There is also the following weaker notion.  We call an injective  $F\ast G$-division ring
$\iota \colon F \ast G \to \cald$ 
\emph{Hughes-free}, if $\cald = \cald_{G}$ holds and, for all pairs $(N,H)$ of subgroups
$N \subseteq H \subseteq G$ of $G$ for which $N$ is normal in $H$ and $H/N$ is infinite
cyclic, the multiplication map $\mu_{(N,H)} \colon \cald_N \otimes_{FN} FH \to \cald$
sending $x \otimes y$ to $x \cdot \iota(y)$ is injective.  In this situation, we say that
$\iota$ is a \emph{Hughes-free $F\ast G$-division ring}. Note  that such $\iota$ is injective and epic.

If $G$ is a locally indicable group, then, for two Hughes free division rings
$\iota \colon F \ast G \to \cald_{F\ast G}$ and
$\iota' \colon F \ast G \to \cald'_{F\ast G}$, there is precisely one isomorphism of
division rings $f \colon \cald_{F\ast G} \xrightarrow{\cong} \cald_{F\ast G}'$ satisfying
$f \circ \iota = \iota'$, see Hughes~\cite{Hughes(1970)}.
Gr\"ater~\cite[Corollary~8.3]{Graeter(2020)} has shown for every locally indicable group
$G$ and every field $F$ that every Hughes free $FG$-division ring
$\iota \colon F \ast G \to \cald_{F\ast G}$ is a Linnell $FG$-division ring.

% -----------------------------------------------------------------------------

\subsection{(Weak) Linnell and (weak) Hughes groups}\label{subsec:(Weak)_Linnell_and_(weak)_Hughes_groups}

\begin{definition}[(Weak) Linnell group]\label{def:(weak)_Linnell_group}
  A \emph{Linnell group} is a group for which a Linnell division ring
  $\iota \colon F \ast G \to \cald_{F\ast G}$ exists for every field $F$ and   crossed product $F \ast G$
  and is unique up to   $F \ast G$-isomorphism for every crossed product ring $F \ast G$ over any field $F$.

  A \emph{weak Linnell group} is a group for which a Linnell free division ring
  $\iota \colon FG \to \cald_{FG}$ exists for every field $F$ and is unique up  to $FG$-isomorphism for the
  group ring $FG$ over any field $F$.
\end{definition}

The following conjecture is taken from~\cite[Conjecture~1]{Jaikin-Zapirain-Linton(2025)}.

\begin{conjecture}[Torsionfree groups are Linnell groups]\label{con:Torsionfree_groups_are_Linnell_groups}
  Every torsionfree group is a  Linnell-group.
\end{conjecture}

 Note that the  existence of an embedding of $FG$ into a skewfield for some field $F$ of characteristic
 zero implies that $G$ is torsionfree. So any Linnell group is automatically torsionfree.

 There is one disadvantage with the notion of a
  Linnell group, namely, we do not know whether every subgroup of a (weak) Linnell group is a
  (weak) Linnell group again. (The uniqueness  part appearing in Definition~\ref{def:(weak)_Linnell_group}
  is the problem.) This would not be a problem if we would know  that
  Conjecture~\ref{con:Torsionfree_groups_are_Linnell_groups} is true.
  
  Therefore we will work with the following notion.

\begin{definition}[(Weak) Hughes group]\label{def:(weak)_Hughes_group}
  A \emph{Hughes group} is a locally indicable group for which a Hughes free division ring
  $\iota \colon F \ast G \to \cald_{F\ast G}$ exists for every field $F$ and   crossed product $F \ast G$.

  A \emph{weak Hughes group} is a locally indicable group for which a Hughes  free division ring
  $\iota \colon FG \to \cald_{FG}$ exists for every field $F$.
\end{definition}

Recall that for a Hughes group $G$ and two Hughes free division rings
$\iota \colon F \ast G \to \cald_{F\ast G}$ and
$\iota' \colon F \ast G \to \cald'_{F\ast G}$, there is precisely one isomorphism of
division rings $f \colon \cald_{F\ast G} \xrightarrow{\cong} \cald_{F\ast G}'$ satisfying
$f \circ \iota = \iota'$, see Hughes~\cite{Hughes(1970)}. 

\begin{remark}[Functoriality of $\cald_{FG}$]\label{rem:functoriality_of_cald_FG}
  Let $\Groups$ be the category of groups. Let $\Groupsinj$ be the category of groups with
  injective group homomorphisms as morphisms.  Denote by $\wLinnellinj$ be category
  of weak Linnell groups  whose morphism are injective group homomorphism.
  Let $\DivisionRings$ be the
  category of division rings.  Denote by $\Fields$ the full subcategory of
  $\DivisionRings$ whose objects are fields.  Let $\FieldsQC$ be the full subcategory of
  $\Fields$ whose objects $F$ are fields satisfying $\IQ \subseteq F \subseteq \IC$.  Then we
  get covariants functors
  \begin{eqnarray*}
    \Groups \times \Fields
    & \to &
   \Rings, \quad (G,F) \mapsto FG;
    \\
    \Groupsinj \times \FieldsQC
    & \to &
    \Rings, \quad (G,F) \mapsto \cald(FG \subseteq \calu(G));
    \\
    \wLinnellinj \times \Fields
    & \to &
   \DivisionRings, \quad (G,F) \mapsto \cald_{FG}.
  \end{eqnarray*}
  A group homomorphism $f \colon H \to G$ induces a homomorphism of $F$-algebras
  $FH \to FG$.  An injective group homomorphism $i \colon H \to G$ induces an injective
  homomorphism of complex algebras $\caln(i) \colon \caln(H) \to \caln(G)$,
  see~\cite[Definition~1.23 on page~29]{Lueck(2002)}. This implies that we obtain the
  functor $(G,F) \mapsto \cald(FG \subseteq \calu(G))$.

  The functor $(G,F) \mapsto \cald_{FG}$ can be constructed after one has chosen for every
  Linnell group $G$ and field $F$ a model $\iota \colon FG \to \cald_{FG}$ of a Linnell
  $FG$-division ring using the fact that for two such models there is precisely one
  isomorphism of $FG$-division algebras between them.  In particular we obtain for every
  inclusion of groups $i \colon H \to G$ and every field $F$ satisfying
  $\IQ \subseteq F \subseteq \IC$ a commutative squares of rings
  \begin{equation}
    \xymatrix{FH \ar[d]_{Fi} \ar[r]
      &
      \cald(FH \subseteq \calu(H)) \ar[d]^{\cald(i)}
      \\
      FG \ar[r] 
      &
      \cald(FG \subseteq \calu(G))
    }
    \label{square_for_cald(FH_subseteq_calu(H)_and_cald(FG_subseteq_cald(FG_substeq_calu(G)}.
  \end{equation}
  For every inclusion of Linnell groups $H \to G$ and every field $F$ we obtain a
  commutative squares of rings
  \begin{equation}
    \xymatrix{FH \ar[d]_{Fi} \ar[r]
      &
      \cald_{FH} \ar[d]^{\cald_{i}}
      \\
      FG \ar[r] 
      &
      \cald_{FG}
    }
    \label{square_for_cald_FH_and_cald_FG}
  \end{equation}
  where the left arrows are the canoncial embeddings.
\end{remark}

Note that every (weak) Hughes group is a (weak) Linnell group and that a (weak) Linnell group
is a (weak) Hughes group if and only if it is locally indicable.  We will later prove that every
\RALI-group, i.e., a residually (amenable and locally indicable) group is a weak Lewin
group and in particular a weak Linnell group and a weak Hughes group, see
Theorem~\ref{the:Residually_(locally_indicable_and_amenable)_groups}~%
\ref{the:Residually_(locally_indicable_and_amenable)_groups:Levin}.
Moreover, there is the conjecture  that any locally indicable group is a
Lewin group and in particular a  Linnell group and a Hughes group, see
Conjecture~\ref{con:Locally_indicibale_groups_are_Lewin_groups}.

\begin{lemma}\label{lem:subgroups_of_(weak)_Hughes-groups_are_(weak)-Hughes_groups}\
  \begin{enumerate}
\item\label{lem:subgroups_of_(weak)_Hughes-groups_are_(weak)-Hughes_groups:subgroups}
  If $G$ is a Hughes  group or a weak Hughes group respectively and $H$ a subgroup of $G$,
  then $H$ is a Hughes group or a weak Hughes group respectively;

\item\label{lem:subgroups_of_(weak)_Hughes-groups_are_(weak)-Hughes_groups_subgroups:colimits}
  Let $G$ be the directed system of subgroups $\{G_i \mid i \in I\}$, directed by inclusion.
  If $G_i$ is a Hughes  group or a weak Hughes group respectively  for every $i \in I$,
  then $G$ a Hughes group or a weak Hughes  group respectively;

\item\label{lem:subgroups_of_(weak)_Hughes-groups_are_(weak)-Hughes_groups_subgroups:fin.gen.}
  A group $G$ is a Hughes  group or a weak Hughes group respectively if and only if each of its finitely generated
  subgroups is a Hughes  group or a weak Hughes group respectively.

\end{enumerate}
\end{lemma}
\begin{proof}~\ref{lem:subgroups_of_(weak)_Hughes-groups_are_(weak)-Hughes_groups:subgroups}
  This follows from the following facts for a locally indicable group $G$,
  subgroup $H \subseteq G$, and every field $F$, which we have already mentioned
  above. Namely, every Hughes free $F\ast G$-division ring is a Linnell $F\ast G$-division
  ring and vice versa, for two Hughes free $F\ast G$-division ring there is precisely one
  isomorphism of $F \ast G$-division rings between them, for a Linnell $F\ast G$ division
  ring $\iota \colon F\ast G \to \cald$ its restriction $F\ast H \to \cald_H$ is a Linnell
  $FH$-division ring, and $H$ is locally indicable.
  \\[1mm]~\ref{lem:subgroups_of_(weak)_Hughes-groups_are_(weak)-Hughes_groups_subgroups:colimits}
  Since each group $G_i$ is locally indicable $G$ is locally indicable. 
  The desired $FG$-division rings is constructed as the colimit of the Hughes free division rings $F[G_i] \to \cald_{F[G_i]}$
  taking Remark~\ref{rem:functoriality_of_cald_FG} into account.
  \\[1mm]~\ref{lem:subgroups_of_(weak)_Hughes-groups_are_(weak)-Hughes_groups_subgroups:fin.gen.}
This follows from assertion~\ref{lem:subgroups_of_(weak)_Hughes-groups_are_(weak)-Hughes_groups:subgroups}
and~\ref{lem:subgroups_of_(weak)_Hughes-groups_are_(weak)-Hughes_groups_subgroups:colimits}
\end{proof}

% -----------------------------------------------------------------------------

\subsection{The definition of $L^2$-Betti numbers over fields}%
\label{subsec:The_definition_of_L2-Betti_numbers_over_fields}

\begin{definition}[$L^2$-Betti numbers over fields]\label{def:L2_Betti_numbers_over_field}
  Consider any  weak Hughes  group $G$ and any field $F$. Let $\iota \colon FG \to \cald_{FG}$
  be the Hughes free $FG$-division ring. Then we define the
\emph{$n$th $L^2$-Betti number over $F$} of any  $G$-$CW$-complex $X$
\[
  b_n^{(2)}(X;\cald_{FG}) = \dim_{\cald_{FG}}\bigl(H_n(\cald_{FG}  \otimes_{\IZ G} C_*(X))\bigr)
  \quad \in \IZ^{\ge 0} \amalg \{\infty\}.
 \]
\end{definition}

Note that the defintion of $b_n^{(2)}(X;\cald_{FG})$ depends only on $G$ and $F$,
since $\iota \colon FG \to \cald_{FG}$ is unique up to unique isomorphism of $FG$-division rings.

Before we study the main properties of $L^2$-Betti numbers over fields,
we explain how it fits with the classical notion of the $L^2$-Betti number
$b_n^{(2)}(X;\caln(G))$.

\begin{lemma}\label{lem:cald_(FG)_is_cald(FG_subseteq_calu(G))}
 Let $G$ be a weak Hughes group. Then:

 \begin{enumerate}

\item\label{lem:cald_(FG)_is_cald(FG_subseteq_calu(G)):isomorphism}
  Let $F$ be a field satisfying  $\IQ \subseteq F \subseteq \IC$.
  The $FG$-division ring  given by the division closure $\cald(FG \subseteq \calu)$
  of $FG$ in $\calu(G)$ and the $FG$-division ring $\cald_{FG}$
  appearing in Definition~\ref{def:weak_Lewin_group} 
 are isomorphic;

\item\label{lem:cald_(FG)_is_cald(FG_subseteq_calu(G)):equality_of_L2-Betti_numbers} Let
  $F$ be a field satisfying $\IQ \subseteq F \subseteq \IC$. Let $X$ be a
  $G$-$CW$-complex. Then the numbers $b_n^{(2)}(X;\caln(G))$
  of~\eqref{b_n_upper_(2)((X;caln(G))} and $b_n^{(2)}(X;\cald_{FG})$ of
  Definition~\ref{def:L2_Betti_numbers_over_field} agree;

  \item\label{lem:cald_(FG)_is_cald(FG_subseteq_calu(G)):dependence_on_F}
  Let $F$ be any field. Then we get
\[
  b_n(X;\cald_{FG})
      = \begin{cases}
       b_n(X;\cald_{\IQ G}) =  b_n^{(2)}(X;\caln(G)) & \text{if} \; F \; \text{has characteristic zero};
        \\
       b_n(X;\cald_{\IF_p G})   & \text{if} \; F \; \text{has prime characteristic}\; p.
      \end{cases}    
 \]
\end{enumerate}
\end{lemma}
\begin{proof}~\ref{lem:cald_(FG)_is_cald(FG_subseteq_calu(G)):isomorphism}
  This follows from~\cite[8.2]{Jaikin-Zapirain+Lopez-Alvarez(2020)}.
  \\[1mm]~\ref{lem:cald_(FG)_is_cald(FG_subseteq_calu(G)):equality_of_L2-Betti_numbers}
  This follows from~\ref{b_n:_upper_(2)((X;caln(G))_in_terms_of_cald(FG_subseteq_calu(G))}  and
  from assertion~\ref{lem:cald_(FG)_is_cald(FG_subseteq_calu(G)):isomorphism}.
  \\[1mm]~\ref{lem:cald_(FG)_is_cald(FG_subseteq_calu(G)):dependence_on_F}
  This follows from assertion~\ref{lem:cald_(FG)_is_cald(FG_subseteq_calu(G)):equality_of_L2-Betti_numbers}
  and the commutative square~\eqref{square_for_cald_FH_and_cald_FG}.
\end{proof}

Note  that everything in this section does makes sense and carries over to all torsionfree groups if
Conjecture~\ref{con:Torsionfree_groups_are_Linnell_groups} is true.

%-----------------------------------------------------------------------------

\subsection{Restriction and induction}\label{subsec:Restriction_and_induction}

\begin{theorem}[Restriction to subgroups of finite index]\label{the:restriction}
  Let $H \subseteq G$ be a subgroup of finite index $[G:H]$ and let $X$ be any  $G$-$CW$-complex.
  Let $\res_G^H X$ be  the  $H$-$CW$-complex obtained from the $G$-$CW$-complex $X$
  by restricting the $G$-action to   an $H$-action.

  \begin{enumerate}

\item\label{the:restriction:special_F}
 We get for every $n \in \IZ^{\ge 0}$
  \[
 b_n^{(2)}(\res_G^H X;\caln(H)) =  [G:H] \cdot b_n^{(2)}(X;\caln(G));
\]
\item\label{the:restriction:general_F}
  Assume  that $G$ is a weak Hughes group. Let $F$ be any field. 
 Then we get for every $n \in \IZ^{\ge 0}$
\[
b_n^{(2)}(\res_G^H X;\cald_{FH}) = [G:H] \cdot b_n^{(2)}(X;\cald_{FG})
\]
using the convention $[G:H] \cdot \infty = \infty$.
\end{enumerate}
\end{theorem}
\begin{proof}~\ref{the:restriction:special_F} This follows from~\cite[Theorem~6.54~(6) on
  page~265]{Lueck(2002)}.  \\[1mm]~\ref{the:restriction:general_F} Since $[G:H]$ is
  finite, there exists a normal subgroup $K \subseteq G$ with $K \subseteq H$ and
  $[G:K] < \infty$.  Hence we can assume without loss of generality that $H \subseteq G$
  is normal.   By the Linnell property the multiplication map
  $\mu_H \colon \cald_H \otimes_{FH} FG \to \cald$ is injective. The conjugation action of $G$ on $G$ extends
  to  a $G$-action through automorphisms of rings $FG \to FG$ and hence to $\cald_{FG} \to \cald_{FG}$,
  since for two Hughes free $FG$-division rings there is precisely one isomorphism of $FG$-division rings between them.
  Obviously this $G$-conjugation of $\cald_{FG}$ induces a $G$-conjugation action on $\cald_H$.
  Hence   $\mu_H$ induces an embedding of the  crossed product ring
  $\cald_{FH} \ast G/H$ into $\cald_{FG}$. As $\cald_{FG}$ is a division ring and $G/H$ is finite,
  $\cald_{FH} \ast G/H$ is Artinian. Since  $\cald_{FH} \ast G/H$ is an integral domain and Artinian, it
  is a division ring. Since  the image of $\iota$ is division closed in $\cald_{FG}$ and contained in  $\cald_{FH} \ast G/H$
  we have  $\cald_{FH} \ast G/H = \cald_{FG}$. This implies  $\dim_{\cald_{FH}}(\res^{\cald_{FH}}_{\cald_{FG}} \cald_{FG}) = [G:H]$.
  Hence we obtain  for every $\cald_{FG}$-module $M$
  \[
  \dim_{\cald_{FH}}(\res^{\cald_{FH}}_{\cald_{FG}} M) = [G:H] \cdot \dim_{\cald_{FG}}(M).
  \]
  Note  that the multiplication map $\mu_H$ induces an isomorphism of $\cald_{FH}$-$FG$-bimodules
  $\cald_{FH} \otimes_{FH} FG \xrightarrow{FG} \cald_{FG}$.
 Hence we get an isomorphism of $\cald_{FH}$-chain complexes
 \begin{multline*}
   \cald_{FH} \otimes_{FH} C_*(\res^H_G X)  \xrightarrow{\cong}   \cald_{FH} \otimes_{FH} \res_{FG}^{FH} C_*(X)
   \\
   \xrightarrow{\cong}   \cald_{FH} \otimes_{FH} FG  \otimes_{FG}  C_*(X)
   \xrightarrow{\cong}   \res^{\cald_{FH}}_{\cald_{FG}} (\cald_{FG} \otimes_{FG} C_*(X)).
 \end{multline*}
 It induces an isomorphism of $\cald_{FH}$-modules
 \[
   H_n(\cald_{FH} \otimes_{FH} C_*(\res^H_G X))
   \xrightarrow{\cong}  \res^{\cald_{FH}}_{\cald_{FG}} H_n(\cald_{FG} \otimes_{FG} C_*(X)).
 \]
 
\end{proof}

\begin{remark}[Virtually weak  Hughes  groups]%
\label{rem:Virtually_weak_Hughes_groups}
Suppose that $G$ is virtual  weak Hughes  group in the sense  that
there exists a subgroup $H \subseteq G$ of finite index $[G:H]$ such that
$H$ is weak Hughes  group.
Let $X$ be an arbitrary $G$-$CW$-complex. Choose any
subgroup $H$ of $G$ of finite index which
is a weak Hughes  group. Then we can define for any field $F$
\[
  b^{(2)}_n(X;\cald_{FG}) = \frac{b_n(\res_G^H X;\cald_{FH})}{[G:H]} \in \IQ^{\ge} \amalg \{\infty\}.
\]
Let $K$ be another subgroup  of $G$ of finite index. Then we get from
Theorem~\ref{the:restriction}~\ref{the:restriction:general_F}
\begin{multline*}
  \frac{b^{(2)}_n(\res_G^H X;\cald_{FH})}{[G:H]}
= \frac{b^{(2)}_n(\res_G^{H \cap K} X;\cald_{F[H \cap K]})}{[G:H] \cdot [H : (H\cap K)]}
\\
= \frac{b^{(2)}_n(\res_G^{H \cap K} X;\cald_{F[H \cap K]})}{[G:K] \cdot [K : (H\cap K)]}
= \frac{b^{(2)}_n(\res_G^K X;\cald_{FK})}{[G:K]}.
\end{multline*}
This shows the independence of the definition of
$ b^{(2)}_n(X;\cald_{FG})$ from the choice of $H$.
If $G$ is a weak Hughes  group, then this new definition agrees with the old one by
Theorem~\ref{the:restriction}~\ref{the:restriction:general_F}. 
\end{remark}

\begin{theorem}[Induction]\label{the:induction}
  Let $H$ be a subgroup of the group  $G$ and let $X$ be a  $H$-$CW$-complex.

  \begin{enumerate}
    \item\label{the:induction:N(G)} 
We get for every $n \in \IZ^{\ge 0}$
\[
  b_n^{(2)}(G \times_HX;\caln(G)) = b_n^{(2)}(X;\caln(H));
\]
  \item\label{the:induction:general_F}
    Suppose  that $G$ is a weak Hughes  group. Then  $H$ is a 
    weak Hughes  group,
    $G \times_H X$ is a  $G$-$CW$-complex,
    and we get for every field $F$ and for $n \in \IZ^{\ge 0}$
  \[
  b_n^{(2)}(G \times_HX;\cald_{FG}) = b^{(2)}_n(X;\cald_{FH}).
\]
\end{enumerate}
\end{theorem}
\smallskip
\begin{proof}~\ref{the:induction:N(G)} This follows from~\cite[Theorem~6.54~(7) on  page~265]{Lueck(2002)}.
  \\[1mm]~\ref{the:induction:general_F} 
  The subgroup $H$ is a
  weak Hughes group by Lemma~\ref{lem:subgroups_of_(weak)_Hughes-groups_are_(weak)-Hughes_groups}.
  If $C_*(X;F)$ is the cellular $FH$-chain complex of the
  $H$-$CW$-complex $X$, then $FG \otimes_{FH} C_*(X;F)$ is $\IZ G$-isomorphic the cellular
  $FG$-chain complex $C_*(G \times_H X;F)$ of the $G$-$CW$-complex $G \times_HX$.  We
  conclude from the commutative square~\eqref{square_for_cald_FH_and_cald_FG} and the fact
  that $\cald(G)$ is flat as a $\cald(H)$-module that the $\cald_{FG}$-modules
  $H_n(\cald_{FG} \otimes_{\IZ G} C_*(G \times_H X;F))$
  and $\cald_{FG} \otimes_{\cald_{FH}} H_n(\cald_{FH} \otimes_{FH} C_*(X;F))$ are
 $\cald_{FG}$-isomorphic. 
\end{proof}

%-----------------------------------------------------------------------------

\subsection{The $0$th $L^2$-Betti number}\label{subsec:The_0th_L2-Betti_number}

\begin{theorem}[$0$th $L^2$-Betti number]\label{the:zeroth_L2-Betti_number}
  Let $G$ be a group and let $X$ be a connected $G$-$CW$-complex.
  \begin{enumerate}

     \item\label{the:zeroth_L2-Betti_number:N(G)}
  We get $b_0^{(2)}(X;\caln(G)) = |G|^{-1}$ if $G$ is finite, and $b_0^{(2)}(X;\caln(G)) = 0$ otherwise;

\item\label{the:zeroth_L2-Betti_number:general_F}
  Suppose that  $G$ is  weak Hughes  group.
  Then  we get $b^{(2)}_0(X;\cald_{FG}) = 0$ for every  field $F$.

  \end{enumerate}
\end{theorem}
\begin{proof}~\ref{the:zeroth_L2-Betti_number:N(G)}  See~\cite[Theorem~6.5~(8) on page~266]{Lueck(2002)}.
  \\[1mm]~\ref{the:zeroth_L2-Betti_number:general_F}
  There is an obvious  $\cald_{FG}$-epimorphism  $f  \colon \cald_{FG} \to \cald_{FG} \otimes_{\IZ G} \IZ$,
    where $G$ acts trivially on $\IZ$. Choose $g \in G$ which is not the unit $e$.
    Then the non-trivial element $g-e \in FG \subseteq \cald_{FG}$ lies in the kernel of $f$.
    As $\cald_{FG}$ is a field,  this implies $ \cald_{FG} \otimes_{FG} \IZ = 0$ and hence
    \begin{multline*}
      b_0^{(2)}(X;\cald_{FG}) = \dim_{\cald_{FG}}\bigl(H_0(\cald_{FG} \otimes_{FG} C_*(X))\bigr) \\
      = \dim_{\cald_{FG}}(\cald_{FG} \otimes_{\IZ G} H_0(X;\IZ))  = \dim_{\cald_{FG}}(\cald_{FG} \otimes_{\IZ G} \IZ) = 0.
    \end{multline*}
  \end{proof}

  %-----------------------------------------------------------------------------

  \subsection{Künneth formula}\label{subsec:Kuenneth_formula}

    \begin{theorem}[K\"unneth formula]\label{the:Kuenneth_formula}
      Let $G$ and $H$ be groups.  Let $X$ be a $G$-$CW$-complex and $Y$ be a
      $H$-$CW$-complex. In the sequel we use  the convention
      that $0 \cdot \infty = 0$, $r \cdot \infty = \infty$ for
      $r \in \IR_{>0} \amalg \{\infty\}$, and $r + \infty = \infty$ for $r \in\IR_{\ge 0} \amalg \{\infty\}$ hold.

      \begin{enumerate}
      \item\label{the:Kuenneth_formula:N(G)}
        Then $X \times Y$ is $(G \times H)$-$CW$-complex  and we get:
        \[
        b_n^{(2)}(X \times Y;\caln(G \times H)) = \sum_{p + q  = n} b_p^{(2)}(X;\caln(G)) \cdot b_q^{(2)}(Y;\caln(H));
      \]

    \item\label{the:Kuenneth_formula:general_F} Suppose that $G \times  H$ is a  weak Hughes
      groups. Then $G$ and $H$ are  weak Hughes group and we get for every field $F$
      \[
        b_n^{(2)}(X \times Y;\cald_{F[G \times H]}) = \sum_{p + q  = n} b_p^{(2)}(X;\cald_{FG}) \cdot b_q^{(2)}(Y;\cald_{FH}).
      \]
      \end{enumerate}
    \end{theorem}
    \begin{proof}~\eqref{the:Kuenneth_formula:N(G)} See~\cite[Theorem~6.54~(5) on
      page~266]{Lueck(2002)}.
      \\[1mm]~\eqref{the:Kuenneth_formula:general_F} Let
      $\varphi_{F [G \times H]} \colon F[G \times H] \to \cald_{F[G \times H]}$ be the
      Hughes free division ring. Recall that $\varphi_{F[G \times H]}$ is injective, epic, 
      and unique up to isomorphisms  of $F[G \times H]$-division rings. We conclude from
      Lemma~\ref{lem:subgroups_of_(weak)_Hughes-groups_are_(weak)-Hughes_groups}~%
\ref{lem:subgroups_of_(weak)_Hughes-groups_are_(weak)-Hughes_groups:subgroups}
      that $G$ and $H$ are weak Hughes groups.  Since $G \times H$ is also a weak Linnell
      group, the inclusion
      \[
        \varphi_{FG} \colon FG \to \cald_{FG} := \cald(\varphi_{F[G \times H]}(FG) \subseteq \cald_{F[G \times H]})
        \]
      of $FG$ into the division closure of $\varphi_{F[G \times H]}(FG)$ in $\cald_{F[G \times H]}$ is a model for
      $\varphi_{FG} \colon FG \to \cald_{FG}$ and analogously for $H$.  Hence we get a
      commutative diagram of rings
      \[
        \xymatrix@!C=9em{FG \otimes_F FH \ar[r]^{\varphi_{FG} \otimes \varphi_{FH}}
          \ar[d]_{\mu}^{\cong}
          &
          \cald_{FG} \otimes_F \cald_{FH}  \ar[d]^{\nu}
          \\
          F[G \times H] \ar[r]_{\varphi_{F[G \times H]}}
          &
          \cald_{F[G \times H]}
        }
      \]
      where the isomorphism $\mu$ and the map $\nu$ send $x \otimes y$ to $x \cdot y$.
      We get from the diagram above an
      isomorphism of $\cald_{F[G \times H]}$-chain complexes
      \begin{multline*}
        \cald_{F[G \times H]} \otimes_{F[G \times H]} C_*(X \times Y)
        \\
        \cong 
        \cald_{F[G \times H]} \otimes_{\cald_{FG} \otimes_F \cald_{FH}} (\cald_{FG} \otimes_{FG} C_*(X)
        \otimes_F  \cald_{FH} \otimes_{FH} C_*(Y)).
      \end{multline*}
      Since $\cald_{FG}$ and $\cald_{FH}$ are fields, the canoncial map of $\cald_{F[G \times H]}$-modules
      \begin{multline*}
        \cald_{F[G \times H]} \otimes_{\cald_{FG} \otimes_F \cald_{FH}}  H_n\bigl(\cald_{FG}
        \otimes_{FG} C_*(X) \otimes_F  \cald_{FH} \otimes_{FH} C_*(Y)\bigr)
      \\
      \cong H_n\bigl(\cald_{F[G \times H]} \otimes_{\cald_{FG} \otimes_F \cald_{FH}}  (\cald_{FG}
      \otimes_{FG} C_*(X) \otimes_F  \cald_{FH} \otimes_{FH} C_*(Y))\bigr)
    \end{multline*}
    is bijective. The K\"unneth Theorem yields an isomorphism of $\cald_{FG} \otimes_F \cald_{FH}$-modules
    \begin{multline*}
      \bigoplus_{p + q = n} H_p(\cald_{FG} \otimes_{FG} C_*(X)) \otimes_F H_q(\cald_{FH} \otimes_{FH} C_*(Y))
      \\
      \cong
      H_n\bigl(\cald_{FG} \otimes_{FG} C_*(X) \otimes_F  \cald_{FH} \otimes_{FH} C_*(Y)\bigr).
    \end{multline*}
    Hence we get  a $\cald_{F[G \times H]}$-isomorphism
    \begin{multline}
       H_n\bigl(\cald_{F[G \times H]} \otimes_{F[G \times H]} C_*(X \times Y)\bigr)
       \\
       \cong
       \bigoplus_{p + q = n}  \cald_{F[G \times H]} \otimes_{\cald_{FG} \otimes_F \cald_{FH}}  
   \bigl(H_p(\cald_{FG} \otimes_{FG} C_*(X)) \otimes_F H_q(\cald_{FH} \otimes_{FH} C_*(Y))\bigr).
\label{the:Kuenneth_formula:homology}
     \end{multline}

Let $M$ be any $FG$-module and $N$ be any $FH$-module. Since $\cald_{FG}$ and $\cald_{FH}$ are fields,
we get for the $F[G \times H]$-module $M \otimes_F N$ 
     
    \begin{eqnarray*}
        \lefteqn{\dim_{\cald_{F[G \times H]}}\bigl(\cald_{F[G \times H]} \otimes_{F[G \times H]} M \otimes_F N\bigr)}
        & &
        \\
        & = &
              \dim_{\cald_{F[G \times H]}}\bigl(\cald_{F[G \times H]} \otimes_{\cald_{FG} \otimes_F \cald_{FH}} (\cald_{FG}
              \otimes_F \cald_{FH} \otimes_{FG \otimes_F FH} M \otimes_F N)\bigr)
     \\
        & = &
              \dim_{\cald_{F[G \times H]}}\bigl(\cald_{F[G \times H]}
              \otimes_{\cald_{FG} \otimes_F \cald_{FH}} (\cald_{FG} \otimes_{FG} M \otimes_F 
              \cald_{FH} \otimes_{FH} N)\bigr)
      \\
        & = &
        \dim_{\cald_{FG}}(\cald_{FG} \otimes_{FG} M) \cdot \dim_{\cald_{FH}}(\cald_{FH} \otimes_{FH} N).       
      \end{eqnarray*}
     Now apply this formula to~\eqref{the:Kuenneth_formula:homology}.
    \end{proof}
  
    We do not know whether $G \times H$ is a weak Hughes group
    if $G$ and $H$ are weak Hughes groups. 
    Note  that the product of two Lewin groups is a Lewin group again,
    see~\cite[Theorem~3.7]{Jaikin-Zapirain(2021)}.

  %-----------------------------------------------------------------------------

  \subsection{The Euler-Poincar\'e formula}\label{subsec:The_Euler-Poincare_formula}

  \begin{theorem}[The Euler-Poincar\'e formula]\label{the:Euler-Poincare_formula}
    Let $G$ be  group and let $X$ be a finite free $G$-$CW$-complex.
    Then $X/G$ is a finite $CW$-complex and we denote by $\chi(X/G)$ its Euler characteristic

    \begin{enumerate}
    \item\label{the:Euler-Poincare_formula:char_0}
      We get
      \[
        \chi(X/G) = \sum_{n \ge 0} (-1)^n \cdot b_n^{(2)}(X;\caln(G));
    \]

  \item\label{the:Euler-Poincare_formula:char_p}
    If $G$  is a  weak Hughes group and $F$ is  any field, then we get 
    \[
      \chi(X/G) = \sum_{n \ge 0} (-1)^n \cdot b_n^{(2)}(X;\cald_{FG}).
    \]
  \end{enumerate}
\end{theorem}
\begin{proof}~\ref{the:Euler-Poincare_formula:char_0} See~\cite[Theorem~1.35~(2) on page~37]{Lueck(2002)}.
\\[1mm]~\ref{the:Euler-Poincare_formula:char_p}
This follows from the additivity of the dimension function $\dim_{\cald_{FG}}$ of the skew field $\cald_{FG}$.
\end{proof}

  %-----------------------------------------------------------------------------
\subsection{Poincar\'e duality}\label{subsec:Poincare_duality}

\begin{theorem}\label{the:Poincare_duality}
  Let $M$ be a compact connected  manifold of dimension $d$ with fundamental group $\pi$.
  Let $(\widetilde{M},\widetilde{M}|_{\partial M})$ be the free $\pi$-$CW$-pair for $\widetilde{M}|_{\partial M}$
      the preimage of $\partial M$ under the universal covering map $\widetilde{M} \to M$.

   \begin{enumerate}
   \item\label{the:Poincare_duality:char_0}
    We get for every $n \in \IZ^{\ge 0}$
    \[
      b_n^{(2)}(\widetilde{M};\caln(\pi)) = b_{d-n}^{(2)}(\widetilde{M},\widetilde{M}|_{\partial M};\caln(\pi));
    \]
  \item\label{the:Poincare_duality:char_p}
    Suppose  that $\pi$ is a weak Hughes   group. Then we get for every field $F$ and every $n \in \IZ^{\ge 0}$
    \[
      b_n^{(2)}(\widetilde{M};\cald_{F\pi}) = b^{(2)}_{d-n}(\widetilde{M},\widetilde{M}|_{\partial M};\cald_{F\pi}).
    \]
  \end{enumerate}
\end{theorem}
\begin{proof}
  We can assume without loss of generality that $M$ is orientable, otherwise pass to the
  orientation covering $\overline{M} \to M$ and use Theorem~\ref{the:restriction}.
  
  Given a $d$-dimensional finitely generated free $\IZ\pi$-chain complex $C_*$, let
  $C^{d-*}$ be the $d$-dimensional finitely generated free $\IZ\pi$-chain complex whose
  $n$-th differential is
    \[
      (-1)^{d-(n-1)} \cdot \hom_{\IZ \pi}(c_{d-(n-1)},\id_{\IZ \pi})
      \colon \hom_{\IZ \pi}(C_{d-n},\IZ \pi) \to \hom_{\IZ \pi}(C_{d-(n-1)};\IZ \pi)
  \]
  Now we get for $n \in \IZ^{\ge 0}$
  \begin{equation}
    \dim_{\caln(\pi)}(H_{d-n}(\caln(\pi) \otimes_{\IZ \pi} C^{d-*}))
    = \dim_{\caln(\pi)}(H_n(\caln(\pi) \otimes_{\IZ \pi} C_*))
  \label{the:chain_and_dual_char_0}
\end{equation}
from~\cite[Lemma~1.18 on page~24 and Theorem~6.24 on page~249]{Lueck(2002)}.

Next we show
\begin{equation}
  \dim_{\cald_{F\pi}}(H_{d-n}(\cald_{F\pi} \otimes_{\IZ \pi} C^{d-*}))
  = \dim_{\cald_{F\pi}}H_n(\cald_{F\pi} \otimes_{\IZ \pi} C_*),
  \label{the:chain_and_dual_char_p}
\end{equation}
provided that $\pi$ is weak Linnell group. Recall that we use on $\IZ \pi$ the involution
given by
$\overline{\sum_{g \in \pi} \lambda_g \cdot g} = \sum_{g \in \pi} \lambda_g \cdot g^{-1}$
to make sense of the dual left $\IZ\pi$-module $\hom_{\IZ \pi}(M,\IZ \pi)$ for a finitely
generated free left $\IZ\pi$-module $M$.  This involution induces an involution on $F\pi$
by the same formula which extends to an involution on $\cald_{F\pi}$, since for two Hughes
free $FG$-division rings there is precisely one isomorphism of $FG$-division rings between
them

This involution on
$\cald_{F\pi}$ is used to make sense of the left dual $\cald_{F\pi}$-module
$\hom_{\cald_{F\pi}}(N,\cald_{F\pi})$ for a finitely generated free left
$\cald_{F\pi}$-module $N$.  Then we obtain for every finitely generated free left
$\IZ \pi$-module $M$ an isomorphism of left $\cald_{FG}$-modules, natural in $M$,
\[
  \cald_{F\pi} \otimes_{\IZ \pi} \hom_{\IZ \pi}(M,\IZ \pi)
  \xrightarrow{\cong} \hom_{\cald_{F\pi}}(\cald_{F\pi} \otimes_{\IZ \pi} M; \cald_{F\pi})
\]
by sending $u \otimes \phi$ to the $\cald_{FG}$-homomorphism
$\cald_{F\pi} \otimes_{\IZ \pi} M \to \cald_{F\pi}$ which maps $v \otimes x$ to
$v\phi(x)\overline{u}$. Hence it suffices to show for a $d$-dimensional finitely generated
free $\cald_{F\pi}$-chain complex $D_*$
\begin{equation}
  \dim_{\cald_{F\pi}}(H_{d-n}(D^{d-*}))  = \dim_{\cald_{F\pi}}H_n(D_*).
  \label{the:chain_and_dual_char_p_special}
\end{equation}
Given  a  finitely generated free $\cald_{FG}$-module $N$, we abbreviate in the sequel
$N^* := \hom_{\cald_{F\pi}}(N,\cald_{F\pi})$.
Fix $n \in \IZ^{\ge 0}$. Recall that $\cald_{F\pi}$ is a division ring.
So we can find $\cald_{FG}$-submodules $U$, $V$, and, $W$ of $D_n$ such that
$U \oplus V \oplus W = D_n$ holds,  an epimorphism $p \colon D_{n+1} \to U$
and a monomorphism $i \colon W \to D_{n-1}$
such that $c_{n+1} \colon D_{n+1} \to D_n = U \oplus V \oplus W$ sends $x$ to $(p(x), 0,0)$
and $c_{n} \colon D_{n} = U \oplus V \oplus W \to D_{n-1}$ sends $(u,v,w)$ to $i(w)$.
Then $H_n(D_*) \cong_{\cald_{F\pi}} W$

Under the obvious identification $U^* \oplus V^* \oplus W^* = D_n^*$ we see that the
$c_{n+1}^* \colon D_n^* =U^* \oplus V^* \oplus W^* \to D_{n+1}^*$ sends $(\phi, \psi,\mu)$
to $p^*(\phi)$ and $c_{n}^* \colon D_{n-1}^* \to D_n^* = U^* \oplus V^* \oplus W^*$ sends
$\phi$ to $(0,0,i^*(\phi))$. As $p$ is surjective, $p^*$ is injective.  Since $i$ is
injective, $i^*$ is surjective. We conclude
\[
 H_{d-n}(D^{d-*}) \cong_{\cald_{F\pi}} W^* \cong_{\cald_{F\pi}} W \cong_{\cald_{F\pi}} H_n(D_*).
\]
This finishes the proof of~\eqref{the:chain_and_dual_char_p_special} and hence
of~\eqref{the:chain_and_dual_char_p}.  Since we get from Poincar\'e duality a
$\IZ \pi$-chain homotopy equivalence
    \[
      C^{d-*}(\widetilde{M}) \xrightarrow{\simeq_{\IZ \pi}} C_*(\widetilde{M},\widetilde{M}|_{\partial M}),
      \]
      the claim follows from~\eqref{the:chain_and_dual_char_0} and~\eqref{the:chain_and_dual_char_p}.
    \end{proof}

    %-----------------------------------------------------------------------------

\subsection{On the top $L^2$-Betti number}%
\label{subsec:On_the_top_L2_Betti_numbers}

\begin{lemma}\label{lem:injectivity_top_differential}
  Let $X$ be a $d$-dimensional proper 
  $G$-$CW$-complex.  Let $c_d \colon C_d(X;F) \to C_{d-1}(X;F)$
  be the $d$-th differential of the $FG$-chain complex of $\overline{X}$. Then:

\begin{enumerate}
\item\label{lem:injectivity_top_differential:caln(G)}
 The  following assertions are equivalent:

  \begin{enumerate}
  \item\label{lem:injectivity_top_differential::caln(G):injective}
    The $\caln(G)$-map
    $\id_{\caln(G)} \otimes_{\IC G}  c_d \colon \caln(G) \otimes_{\IC G} C_d(\overline{X};\IC)
    \to \caln(G) \otimes_{\IC G} C_{d-1}(\overline{X};\IC)$ is injective;
  \item\label{lem:injectivity_top_differential:vanishing_L2-Betti_number}
    We have $b_d^{(2)}(\overline{X};\caln(G)) = 0$;
  \end{enumerate}

\item\label{lem:injectivity_top_differential:cald_(FG)}
Suppose  that $G$ is a weak Hughes   group. The  following assertions are equivalent:

\begin{enumerate}
  \item\label{lem:injectivity_top_differential:cald_(FG):injective}
    The $\cald_{FG}$-map
    $\id_{\cald_{FG}} \otimes_{\IC G}  c_d \colon \cald_{FG} \otimes_{\IC G} C_d(\overline{X};F)
    \to \cald_{FG} \otimes_{FG} C_{d-1}(\overline{X};F)$ is injective;
  \item\label{lem:injectivity_top_differential:cald_(FG):vanishing_L2-Betti_number}
    We have $b_d^{(2)}(\overline{X};\cald_{FG}) = 0$.
  \end{enumerate}
\end{enumerate}
\end{lemma}
\begin{proof}~\ref{lem:injectivity_top_differential::caln(G):injective} Since $X$ is
  proper, $C_m(X;\IC)$ is a projective $\IC G$-module and hence
  $\caln(G) \otimes_{\IC G} C_m(X;\IC)$ is a projective $\caln(G)$-module for every
  $m \in \IZ^{\ge 0}$.  Hence it suffices to show for a $\caln(G)$-homomorphism
  $f \colon P \to Q$ of projective $\caln(G)$-modules that $\ker(f)$ is trivial if and
  only if $\dim_{\caln(G)}(\ker(f))$ vanishes. The only if statement is obvious, the if
  statement is proved as follows.

     The condition $\dim_{\caln(G)}(\ker(f)) = 0$ implies that every finitely generated projective
  $\caln(G)$-submodule $P$ of $(\ker(f))$ is trivial,
  see~\cite[Theorem~1.12~(1) on page~21 and Definition~6.6 on page 239]{Lueck(2002)}. 
Since $\caln(G)$ is semi-hereditary,
 see~\cite[Theorem~6.5 and Theorem~6.7~(i) on page 239]{Lueck(2002)},
 $\ker(f)$ is trivial.
\\[1mm]~\ref{lem:injectivity_top_differential:cald_(FG)}  This easy proof is left to the reader.
  \end{proof}

\begin{theorem}[On the top $L^2$-Betti number]\label{the:On_the_top_L2-Betti_number}
  Let $H$ be a subgroup of $G$. Let $X$ be a $d$-dimensional proper  (not
  necessarily finite or connected) $G$-$CW$-complex which we can also view as
  $d$-dimensional proper $H$-$CW$-complex $\res_G^H X$ by restriction.

  \begin{enumerate}
  \item\label{the:On_the_top_L2-Betti_number:special_F}
    Suppose $b_d^{(2)}(X;\caln(G)) = 0$ holds. Then we get
    \[b_d^{(2)}(\res_G^H;\caln(H)) = 0;
      \]

    \item\label{the:On_the_top_L2-Betti_number:general_F} Let $F$ be any field.  Suppose
      that $G$ is a weak Hughes group 
      and $b_d^{(2)}(X;\cald_{FG}) = 0$ holds.

  Then $H$ is a  weak Hughes  group  and we get
  \[
    b_d^{(2)}(\res_G^H X;\cald_{FH}) = 0.
    \]
\end{enumerate}
\end{theorem}
\begin{proof}~\ref{the:On_the_top_L2-Betti_number:special_F}
Recall that $\caln(G)$ is defined to be the algebra $\calb(L^2(G))^G$ of
  bounded linear operators $L^2(G) \to L^2(G)$ which are (left) $G$-equivariant, and that
  we have a canoncial inclusion of rings $\IC G \to \caln(G)$ sending
  $\sum_{g \in G} \lambda_g \cdot g \in G$ to the operator
  $L^2(G) \to L^2(G), \; u \mapsto \sum_{g \in G} \overline{\lambda_g} \cdot ug^{-1}$,
  see~\cite[Definition~1.1 on page~15]{Lueck(2002)}.  The inclusion $i \colon H \to G$ induces an
  injective homomorphism of complex algebras $\caln(i) \colon \caln(H) \to \caln(G)$,
  see~\cite[Definition~1.23 on page~29]{Lueck(2002)}. For any $\IC G$-module $M$ we a
  obtain an $\caln(H)$-homomorphism, natural in $M$,
  \begin{equation}
    U(M) \colon \caln(H) \otimes_{\IC H} \res_G^H M \to \caln(G) \otimes_{\IC G} M,
    \quad u \otimes x \mapsto \caln(i)(u) \otimes x.
 \label{U(M)}
\end{equation}
Next we want to show that $U(M)$ is injective for projective  $\IC G$-modules $M$.  As $U(M)$ is
compatible with direct sums over arbitrary index sets, it suffices to prove the claim for
$M = \IC G$.

  Let $T$ be a transversal of the projection  $\pr \colon G \to H\backslash G$,
  i.e., $T$ is a subset of $G$ such that $\pr|_T \colon T \to H\backslash G$ is a bijection.
  We will assume $e \in T$ for the unit $e \in G$.
  We obtain an isomorphism of left $\IC H$-modules
  \[
  \bigoplus_{t \in G/H} \IC H \xrightarrow{\cong} \IC G, \quad (x_t)_{t \in T} \mapsto \sum_{t \in T} x_t \cdot t.
  \]
  It induces a bijection
  \[
    \alpha \colon \bigoplus_{t \in T} \calb(L^2(H))^H \xrightarrow{\cong}  \calb(L^2(H))^H \otimes_{\IC H} \IC G,
    \quad (f_t)_{t \in T} \mapsto \sum_{t \in T} f_t \otimes t
  \]
  and an $H$-equivariant isomorphism of Hilbert spaces
  \[
    \xi \colon \overline{\bigoplus}_{t \in T} L^2(H) \xrightarrow{\cong} L^2(G),
    \quad (u_t)_{t \in T} \mapsto \sum_{t \in T} u_t \cdot t,
  \]
  where the source is the Hilbert space completion of the pre Hilbert space
  $\bigoplus_{t \in G/H} L^2(H)$.  We have the  isomorphism
  \[\beta \colon \calb(L^2(G))^G  \otimes_{\IC G} \IC G \xrightarrow{\cong} \calb(L^2(G))^G,
    \quad f \otimes \biggl(\sum_{g \in G} \lambda_g \cdot g\biggr)
    \mapsto   \sum_{g \in G} \overline{\lambda_g} \cdot \bigl(f \circ r _{g^{-1}}\bigr)
  \]
  where $r_{g^{-1} }\colon L^2(G) \to L^2(G)$ is right multiplication with $g^{-1}$. Hence
  it remains to show that
  \[
    \beta \circ U(\IC G) \circ \alpha \colon \bigoplus_{t \in T} \calb(L^2(H))^H  \to \calb(L^2(G))^G
  \]
  is injective. Given $s \in T$ and an element $f_s$ in the copy of $\calb(L^2(H))^H$
  belonging to $s$ in the source of $\alpha$, we get a commutative diagram of operators of
  Hilbert spaces
  \begin{equation}
    \xymatrix{\overline{\bigoplus}_{t \in T} L^2(H) \ar[r]^-{\xi}_-{\cong} \ar[d]_{\sigma_s}
      &
      L^2(G) \ar[d]^{r_{s^{-1}}}
        \\
      \overline{\bigoplus}_{t \in T} L^2(H) \ar[r]^-{\xi}_-{\cong}  \ar[d]_{\bigoplus_{t \in T} f_s}
      & L^2(G) \ar[d]^{\caln(i)(f_s)}
        \\  
       \overline{\bigoplus}_{t \in T} L^2(H) \ar[r]^-{\xi}_-{\cong} 
       & L^2(G)
     }
     \label{diagram_computing_i(f_s)_circ_r_s}
   \end{equation}
   where for $t_0, t_1 \in T$ the operator $(\sigma_s)_{t_0,t_1} \colon L^2(H) \to L^2(H)$
   is trivial if $\pr(t_0s^{-1}) \not= \pr(t_1)$ and is right multiplication with the
   element $t_0s^{-1}t_1^{-1} \in H$ if $\pr(t_0s^{-1}) = \pr(t_1)$.

   Now suppose that the
   element $(f_s)_{s \in T} \in \bigoplus_{t \in G/H} \calb(L^2(H))^H$ lies in the kernel
   of $\beta \circ U(\IC G) \circ \alpha$.  Then
   $\xi^{-1} \circ \left(\sum_{s \in T} \caln(i)(f_s) \circ r_{s^{-1}}\right)  \circ \xi\colon
   \overline{\bigoplus}_{t \in T} L^2(H) \to \overline{\bigoplus}_{t \in T} L^2(H)$ is
   trivial.

   We get for $t \in T$ that $(\sigma_s)_{t,e}$ is zero for $s \not = t$,
   since $\pr(ts^{-1}) = \pr(e) \implies ts^{-1} \in H \implies \pr(t) = \pr(s) \implies s = t$.
   Fix $t \in T$. Let $j_t \colon L^2(H) \to  \overline{\bigoplus}_{t \in T} L^2(H)$
   be the inclusion of the summand belonging to $t \in T$ and
   $\pr_e \colon \overline{\bigoplus}_{t \in T} L^2(H) \to L^2(H)$
   be the projection onto the summand  belonging to $e \in T$.
   Then the  composite
   \[
   \pr_{e} \circ \; \xi^{-1} \circ \left(\sum_{s \in T} i(f_s) \circ r_{s^{-1}}\right)  \circ \xi \circ j_t \colon
   L^2(H) \to L^2(H)
 \]
 is zero and sends $u$ to $f_t(u)$ because of the commutative
 diagram~\eqref{diagram_computing_i(f_s)_circ_r_s}.  Hence $f_t = 0$ for all $t \in T$.
 This finishes the proof that the natural map $U(M)$ of~\eqref{U(M)} is injective for
 every projective $\IC G$-module $M$.

   If $c_d \colon C_*(X) \to C_*(X)$ is the $d$-th differential
   in the cellular $\IC G$-chain complex $C_*(X;\IC)$, we get a commutative
   diagram
     \[
       \xymatrix@!C=17em{\caln(H) \otimes_{\IC H} C_d(\res_G^H X;\IC) \ar[r]^-{U(C_d(X;\IC))}
         \ar[d]_{\id_{\caln(H)} \otimes_{\IC H} \res_G^Hc_d}
         &
         \caln(G) \otimes_{\IC G} C_d(X;\IC)   \ar[d]^{\id_{\caln(G)} \otimes_{\IC G} c_d}
         \\
         \caln(H) \otimes_{\IC H} C_{d-1}(\res_G^H X;\IC) \ar[r]_-{U(C_{d-1}(X;\IC))}
         &
         \caln(G) \otimes_{\IC G} C_{d-1}(X;\IC) 
       }
     \]
     whose horizontal arrows are injective. Hence the left vertical arrow is injective if
     the right vertical arrow is injective.

     Hence it suffices to show that $\ker(\id_{\caln(G)} \otimes_{\IC G} c_d)$ is trivial
     if and only if we have
     $\dim_{\caln(G)}(\ker(\id_{\caln(H)} \otimes_{\IC G } c_d)) = 0$.  This has already
     been proved in Lemma~\ref{lem:injectivity_top_differential}.
     \\[1mm]~\ref{the:On_the_top_L2-Betti_number:general_F} The proof is analogous to the
     one of assertion~\ref{the:On_the_top_L2-Betti_number:special_F} but one has to
     consider division rings $\cald_{FG}$ and $\cald_{FH}$ instead of $\caln(G)$ and
     $\caln(H)$.  Since the $FG$-division ring $\iota \colon FG \to \cald_{FG}$ is
     Linnell, the multiplication map $\mu_H$ induces an injection of
     $\cald_{FH}$-$FG$-bimodules $\mu_H \colon \cald_{FH} \otimes_{FH} FG \to \cald_{FG}$.
     Hence we get for any projective $FG$-module $M$ an injective $\cald_{FH}$-homomorphism,
     natural in $M$,
     \[
     V(M) \colon \cald_{FH} \otimes_{FH} \res_{FG}^{FH} M \xrightarrow{\cong}  \cald_{FG} \otimes_{FG} M.
     \]

 Since $G$ is torsionfree and $X$ is proper, the $FG$-chain complex $C_*(X;F)$ is a free  $FG$-chain complex.
     Let $c_d \colon C_*(X;F) \to C_*(X;F)$ be  the $d$-th differential
     in the cellular $FG$-chain complex $C_*(X;F)$. We get a commutative   diagram
     \[
       \xymatrix@!C=17em{\cald_{FH} \otimes_{FH} C_d(\res_G^H X;F) \ar[r]^-{V(C_d(X;F))}
         \ar[d]_{\id_{\cald_{FH}} \otimes_{FH} \res_G^H c_d}
         &
         \cald_{FG}\otimes_{FG} C_d(X;F)   \ar[d]^{\id_{\cald_{FG}} \otimes_{FG} c_d}
         \\
         \cald_{FH}\otimes_{FH} C_{d-1}(\res_G^H X;F) \ar[r]_--{V(C_{d-1}(X;F))}
         &
         \cald_{FG} \otimes_{FG} C_{d-1}(X;\IC) 
       }
     \]
     whose horizontal arrows are injective. We conclude that  the left vertical arrow is injective if
     the right vertical arrow is injective. Hence we get
     $b_d(X;\cald_{FG}) = 0 \implies b_d(\res_G^H X;\cald_{FH}) = 0$.

     This finishes the proof of Theorem~\ref{the:On_the_top_L2-Betti_number}.
  \end{proof}

  Theorem~\ref{the:On_the_top_L2-Betti_number}~\ref{the:On_the_top_L2-Betti_number:special_F}
  has already been proved for countable $G$ and proper $G$-actions on simplicial complexes
  in~\cite[Theorem~1.8]{Gaboriau-Nous(2021)}. We have given a different proof here which extends to a proof
of Theorem~\ref{the:On_the_top_L2-Betti_number}~\ref{the:On_the_top_L2-Betti_number:general_F}.

% -----------------------------------------------------------------------------

\subsection{Mapping tori}%
\label{subsec:Mapping_tori}

\begin{theorem}[Vanishing of $L^2$-Betti numbers of mapping tori]%
  \label{the:vanishing_of_L2-Betti_numbers_for_mapping_torus}
  Let $f\colon X \rightarrow X$ be a cellular selfmap of a connected $CW$-complex $X$ and
  let $\pi_1(T_f) \xrightarrow{\phi} G \xrightarrow{\psi}\IZ$ be a factorization of the
  canonical epimorphism into epimorphisms $\phi$ and $\psi$. Let 
  $i_*\colon \pi_1(X) \to \pi_1(T_f)$ is the map induced by the canonical inclusion
  $i \colon X  \to T_f$.  Consider $n \in \IZ^{\ge 1}$.  Let $\overline{T_f}$ be the covering of $T_f$
  associated to $\phi$, which is a free $G$-$CW$-complex.

  \begin{enumerate}
  \item\label{the:vanishing_of_L2-Betti_numbers_for_mapping_torus:(NG)} Suppose that
    $b_n^{(2)}(G \times_{\phi \circ i_*} \widetilde{X};\caln(G)) < \infty$ and
    $b_{n-1}^{(2)}(G \times_{\phi \circ i_*} \widetilde{X};\caln(G)) < \infty$ holds.  Then we
    get
    \[
      b_n^{(2)}(\overline{T_f};\caln(G)) = 0;
    \]

 \item\label{the:vanishing_of_L2-Betti_numbers_for_mapping_torus:general_F}
Suppose  $G$ is a weak Hughes group and that
$b_n(G \times_{\phi \circ i_*} \widetilde{X};\cald_{FG}) < \infty$ and
$b_{n-1}(G \times_{\phi \circ i_*} \widetilde{X};\cald_{FG}) < \infty$ holds.
 Then we get
\[
b_p^{(2)}(\overline{T_f};\cald_{FG}) = 0.
\]
\end{enumerate}
\end{theorem}
\begin{proof}~\ref{the:vanishing_of_L2-Betti_numbers_for_mapping_torus:(NG)}
  This is proved in~\cite[Theorem~6.63 on page~270]{Lueck(2002)}.
  \\[1mm]~\ref{the:vanishing_of_L2-Betti_numbers_for_mapping_torus:general_F}
The proof of~\cite[Theorem~6.63 on page~270]{Lueck(2002)} carries over to this setting.
\end{proof}

% -----------------------------------------------------------------------------

\subsection{Fibrations}%
\label{subsec:fibration}

\begin{theorem}[$L^2$-Betti number and fibrations]\label{the:L2-Betti_numbers_and_fibrations}
Let $F \xrightarrow{i} E \xrightarrow{p} B$ be a fibration
of connected $CW$-complexes.
Let $p_*\colon  \pi_1(E) \xrightarrow{\phi} G \xrightarrow{\psi} \pi_1(B)$
be a factorization of the map induced by $p$ into epimorphisms
$\phi$ and $\psi$. Let $i_*\colon  \pi_1(F) \to \pi_1(E)$
be the homomorphism induced by the inclusion $i$.
Consider $d \in \IZ^{\ge 1}$.  Then:

\begin{enumerate}
\item\label{the:L2-Betti_numbers_and_fibrations:N(G)_d_not_infty}
Suppose  that $\pi_1(B)$ contains an element
of infinite order or finite subgroups of arbitrarily large order.
Suppose that $b_n^{(2)}(G \times_{\phi \circ i_*}\widetilde{F};\caln(G))= 0$
for $n \le d-1$ and
$b_d^{(2)}(G \times_{\phi \circ i_*}\widetilde{F};\caln(G)) < \infty$
hold.

Then we get for $n \le d$
\[
  b_n^{(2)}(G \times_{\phi} \widetilde{E};\caln(G)) = 0;
\]

\item\label{the:L2-Betti_numbers_and_fibrations:N(G)_d_is_infty}
Suppose that $b_n^{(2)}(G \times_{\phi \circ i_*}\widetilde{F};\caln(G))= 0$
holds for every  $n  \in \IZ^{\ge 0}$.

Then we get for every $n \in \IZ^{\ge 0}$
\[
  b_n^{(2)}(G \times_{\phi} \widetilde{E};\caln(G)) = 0;
\]

\item\label{the:L2-Betti_numbers_and_fibrations:general_F_d_not_infty}
 Assume that  $G$ is a weak Hughes group and that $\pi_1(B)$ contains an element
 of infinite order or finite subgroups of arbitrarily large order. Suppose that
 $b_n(G \times_{\phi \circ i_*}\widetilde{F};\cald_{FG})= 0$
for $n \le d-1$ and
$b_d(G \times_{\phi \circ i_*}\widetilde{F};\cald_{FG}) < \infty$
hold.

Then we get for $n \le d$
\[
  b_n^{(2)}(G \times_{\phi} \widetilde{E};\cald_{FG}) = 0;
\]

\item\label{the:L2-Betti_numbers_and_fibrations:general_F_d_is_infty}
Assume that  $G$ is a weak Hughes group.
Suppose that $b_n(G \times_{\phi \circ i_*}\widetilde{F};\cald_{FG})= 0$
holds for every $n  \in \IZ^{\ge 0}$.

Then we get for every $n \in \IZ^{\ge 0}$
\[
  b_n^{(2)}(G \times_{\phi} \widetilde{E};\cald_{FG}) = 0;
  \]
\end{enumerate}
\end{theorem}
\begin{proof}~\ref{the:L2-Betti_numbers_and_fibrations:N(G)_d_not_infty}
  This is proved in~\cite[Theorem~6.67 on page~272]{Lueck(2002)} for $d \in \IZ^{\ge 1}$
  \\[1mm]~\ref{the:L2-Betti_numbers_and_fibrations:N(G)_d_is_infty}
  This is proved in~\cite[Lemma~6.66 on page~272]{Lueck(2002)} for $d \in \IZ^{\ge 1}$
  \\[1mm]~\ref{the:L2-Betti_numbers_and_fibrations:general_F_d_not_infty}
  The proof of assertion~\ref{the:L2-Betti_numbers_and_fibrations:N(G)_d_not_infty}
  carries over to this setting.
  \\[1mm]~\ref{the:L2-Betti_numbers_and_fibrations:general_F_d_is_infty}
   The proof of assertion~\ref{the:L2-Betti_numbers_and_fibrations:N(G)_d_is_infty}
  carries over to this setting.
\end{proof}

%-----------------------------------------------------------------------------

\subsection{$3$-manifolds}\label{subsec:3-manifolds}

Let $M$ be a compact connected $3$-manifold $M$. It  is called \emph{prime}
if for any decomposition of $M$ as a connected sum $M_1 \#  M_2$,
$M_1$ or $M_2$ is homeomorphic to $S^3$, and is called 
\emph{irreducible}
if every embedded $2$-sphere bounds an
embedded $3$-disk. If $M$ is prime , it  is either irreducible or homeomorphic
to the total space over an $S^1$-bundle over $S^1$, see~\cite[Lemma 3.13 on page~28]{Hempel(1976)}. 
There is a \emph{prime decomposition} of $M$, i.e.,
one can write $M$ as a connected sum
\begin{eqnarray*}
M  & = & M_1 \# M_2 \# \ldots \# M_r,
\end{eqnarray*}
where each $M_j$ is prime. If $M$ is orientable, this prime decomposition is
unique up to renumbering and orientation preserving homeomorphism,
see~\cite[Theorem 3.15 on page~31  and Theorem~3.21 on pages~35]{Hempel(1976)}.
By the Sphere Theorem, see~\cite[Theorem 4.3 on page~40]{Hempel(1976)}, an irreducible 
3-manifold is aspherical if and only if
it is a 3-disk or has infinite fundamental group.

Recall that the $L^2$-Betti numbers $b_n^{(2)}\bigl(\widetilde{M};\caln(\pi_1(M))\bigr)$
have been computed for all $3$-manifolds $M$ by
Lott-L\"uck~\cite[Theorem~0.1]{Lott-Lueck(1995)} assuming Thurston's Geometrization
Conjecture, which is nowadays known to be true
by~\cite{Kleiner-Lott(2008),Morgan-Tian(2014)} following the spectacular outline of
Perelman.  In this section we want extend this result to $L^2$-betti numbers over arbitrary fields.

We first recall the answer for the classical $L^2$-Betti numbers.  For a group $G$ denote
by $\chi^{(2)}(G) = \chi^{(2)}(EG;\caln(G))$ is \emph{$L^2$-Euler characteristic},
see~\cite[Definition~6,79 on page~277]{Lueck(2002)}. If $G$ is virtually homotopy finite,
i.e, it contains a subgroup $H$ of finite index such that there is a finite model for
$BH$, then the \emph{virtual Euler characteristic} $\chi_{\virt}(G)$, which is sometimes
also called the \emph{rational valued group Euler characteristic}, is defined to be
$\frac{\chi(BH)}{[G:H]}$, and we have the equality, see~\cite[Remark~6.81 on  page~280]{Lueck(2002)}
  \begin{equation}
    \chi^{(2)}(G) = \chi_{\virt}(G).
    \label{chi_upper_(2)(G)_is_chi_virt(G)}
  \end{equation}

\begin{theorem}[$L^2$-Betti numbers of $3$-manifolds]%
\label{the:L2-Betti_numbers_of_3-manifolds_over_caln(pi)}
  Let $M$ be a compact connected  $3$-manifold with infinite fundamental  group $\pi = \pi_1(M)$.
  Let $M = M_1 \# \cdots \# M_r$ be its prime decomposition.

  Then the
  $L^2$-Betti numbers of the universal covering $\widetilde{M}$ are given by
  \begin{eqnarray*}
    b_1^{(2)}(\widetilde{M};\caln(\pi)) & = & - \chi^{(2)}(\pi_1(M)); 
    \\
    b_2^{(2)}(\widetilde{M};\caln(\pi)) & = & \chi(M) - \chi^{(2)}(\pi_1(M));   
    \\
    b_n^{(2)}(\widetilde{M};\caln(\pi)) & = & 0 \quad   \text{for}\; n \notin \{1,2\},
  \end{eqnarray*}
  and we have
  \begin{eqnarray*}
    \chi^{(2)}(\pi_1(M))
    & = &
    (r-1) - \sum_{j=1}^r \chi^{(2)}(\pi_1(M_j))
    \\
    & = &
    (r-1) - \sum_{j=1}^r \frac{1}{|\pi_1(M_j)|} +
   \left|\{C \in \pi_0(\partial M) \mid C \cong S^2\}\right|  - \chi(M).
  \end{eqnarray*}
  where we put $\frac{1}{|\pi_1(M_j)|}$ to be $0$ if $\pi_1(M_j)$ is infinite.
\end{theorem}
\begin{proof}
  First we show
  \begin{eqnarray*}
    \chi^{(2)}(\pi_1(M))
    & = &
    (r-1) - \sum_{j=1}^r \chi^{(2)}(\pi_1(M_j))
    \\
    & = &
    (r-1) - \sum_{j=1}^r \frac{1}{\mid \pi_1(M_j)\mid} +
   \left|\{C \in \pi_0(\partial M) \mid C \cong S^2\}\right|  - \chi(M).
  \end{eqnarray*}
  The first equation follows from~\cite[~Theorem~6.80~(2) and~(8) page~277]{Lueck(2002)}.
  The last equation follows from the facts that each $M_j$ is aspherical with infinite
  fundamental group,   the total space over an $S^2$-bundle over $S^1$,
  or a connected compact  manifold with finite fundamental group,
  whose universal covering is homotopy equivalent to $S^3$ or $D^3$.

  In Lott-L\"uck~\cite[Theorem~0.1]{Lott-Lueck(1995)} only orientable compact
  $3$-manifolds are treated.  The non-orientable case can be reduced to the orientable
  case by passing to the orientation covering since $b_n^{(2)}(\widetilde{M};\caln(\pi))$,
  $\chi^{(2)}(\pi_1(M))$, and $\chi(M)$ are multiplicative under finite coverings.
   \end{proof}

Next we formulate the main result of this section which follows also from~\cite[Example~A.8]{Fisher(2025phd)}.

\begin{theorem}[$L^2$-Betti numbers of $3$-manifolds for arbitrary fields]%
\label{the:L2-Betti_numbers_of_3-manifolds_for_arbitrary_fields}
  Let $M$ be a compact connected $3$-manifold with infinite fundamental  group $\pi = \pi_1(M)$
  and $F$ be any field. Then:

  \begin{enumerate}

  \item\label{the:L2-Betti_numbers_of_3-manifolds_for_arbitrary_fields:Hughes}
    The fundamental group $\pi$ is virtually locally indicable and virtually a Hughes group.
    (In particular $b_n^{(2)}(\widetilde{M};\cald_{F\pi})$ is defined, see Remark~\ref{rem:Virtually_weak_Hughes_groups}.)

    \item\label{the:L2-Betti_numbers_of_3-manifolds_for_arbitrary_fields:equality}
      We have the equality
      \[
        b_n^{(2)}(\widetilde{M};\caln(\pi)) = b_n^{(2)}(\widetilde{M};\cald_{F\pi}).
      \]

    \end{enumerate}
  \end{theorem}

  \begin{remark}\label{rem:finite_pi}
    Let $M$ be a compact connected $3$-manifold with finite fundamental  group $\pi = \pi_1(M)$ and $F$ be any field.
    Then $b_n^{(2)}(\widetilde{M}, \cald_{F\pi})$ agrees with
    $\frac{b_n(\widetilde{M})}{|\pi|}$, where  $b_n(\widetilde{M})$ is the  Betti number of the  universal covering $\widetilde{M}$,
which is given by
\[b_n(\widetilde{M} )=
  \begin{cases}
    1 & \text{if} \; n = 0;
    \\
    0 & \text{if}\; n = 1;
    \\
    \frac{\chi(\partial M)}{2}  -1 & \text{if}\; n = 2 \;\text{and}\;  \partial M \not = \emptyset;
    \\
    0 & \text{if}\; n = 2 \;\text{and}\;    \partial M = \emptyset;
    \\
    1 & \text{if}\; n = 3 \;\text{and}\; \partial M = \emptyset;
    \\
    0 & \text{if}\; n = 3 \;\text{and}\; \partial M \not= \emptyset;
    \\
    0 & \text{if}\; n \ge 4.
  \end{cases}
\]
\end{remark}

The proof of Theorem~\ref{the:L2-Betti_numbers_of_3-manifolds_for_arbitrary_fields} needs
some preparations.

\

 \begin{lemma}\label{lem:second_Betti_number_of_irreducible_3-,manifolds_with_boundary}
   Let $M$ be an irreducible $3$-manifold with infinite fundamental group and non-empty
   boundary $\partial M$. Then its fundamental group $\pi = \pi_1(M)$ is virtually 
   a $\RALI$-group,
   see Definition~\ref{def:RALI-group},  and in particular virtually a weak Hughes group, and we get for any field $F$
    \begin{eqnarray*}
    b^{(2)}_1(\widetilde{M};\cald_{F\pi})  & = & -\chi(M);
    \\
    b^{(2)}_n(\widetilde{M};\cald_{F\pi}) & = & 0 \quad \text{if} \;n \not= 1.
    \end{eqnarray*}                                       
  \end{lemma}
\begin{proof}
  Since $M$ is an irreducible connected compact $3$-manifold with infinite fundamental
  group, it is aspherical and hence
  $b^{(2)}_n(\widetilde{M};\cald_{F\pi}) = b^{(2)}_n(E\pi;\cald_{F\pi})$ for $n \in \IZ^{\ge 0}$.
  Since $\partial M$ is non-empty, $M$ is homotopy equivalent to a $2$-dimensional
  $CW$-complex.  We conclude from~\cite[Theorem~1.1]{Kielak-Linton(2023grouprings)} that
  that there is a subgroup of finite index $G \subseteq \pi$ for which there exists 
  group extension $1 \to \Gamma \to G \xrightarrow{\pr} \IZ \to 1$ for some
  (not necessarily finitely generated) countable free group $\Gamma$. 
  The lower central series
  $\Gamma = \Gamma_0 \supseteq \Gamma_1 \supseteq \Gamma_2 \supseteq \Gamma_2 \supseteq
  \cdots$ of $\Gamma$ consists of characteristic normal subgroups such that each quotient
  $\Gamma/\Gamma_i$ is torsionfree nilpotent and
  $\bigcap_{i = 0}^{\infty} \Gamma_i = \{1\}$, see~\cite[Section~2.1 and~2.2]{Duchamp-Krob(1992)}.
  Let  $\phi \colon \Gamma \to \Gamma$ be the automorphism coming from conjugation with an
  element $w \in \pi$ which is mapped under $\pr \colon \pi \to \IZ$ to a generator. Then
  $\phi(\Gamma_i) = \Gamma_i$ and we denote by
  $\overline{\phi}_i \colon \Gamma/\Gamma_i \xrightarrow{\cong} \Gamma/\Gamma_i$ the
  induced automorphism. Note that $\Gamma_i$ is a normal subgroup of
  $\pi = \Gamma \rtimes_{\phi} \IZ$.  Put
  $Q_i = \pi/\Gamma_i \cong \Gamma/\Gamma_i \rtimes_{\overline{\phi}_i} \IZ$. Hence we get
  a sequence of normal subgroups of $\pi$
\[\Gamma = \Gamma_0 \supseteq  \Gamma_1 \supseteq  \Gamma_2 \supseteq  \Gamma_3 \supseteq  \cdots
\]
such that $\bigcap_{i = 0}^{\infty} \Gamma_i = \{1\}$ holds. Put
$Q_i = \pi/\Gamma_i \cong \Gamma/\Gamma_i \rtimes_{\overline{\phi}_i} \IZ$.  Since
$\Gamma/\Gamma_i$ is a torsionfree nilpotent group, $Q_i$ is amenable and locally
indicable. Hence $G$ is a \RALI-group and $\pi$ is virtually a \RALI-group.
We conclude from
Theorem~\ref{the:Approximation_over_any_F}~\ref{the:Approximation_over_any_F:General_F}
\begin{multline*}
b^{(2)}_n(EG;\cald_{FG}) = \dim_{\cald_{FG}}\bigl(H_2(\cald_{FG} \otimes C_*(EG))\bigr)
\\
=
\lim_{i \to \infty} \dim_{\cald_{FQ_i}}\bigl(H_n(\cald_{FQ_i} \otimes C_*(EG/\Gamma_i))\bigr).
\end{multline*}
Since $Q_i$ is amenable and locally indicable, $\cald_{FQ_i}$is the Ore localization
$S^{-1}FQ_i$ of $FQ_i$ with respect to the multiplicative closed subset
$S = FG \setminus \{0\}$ of $FG$. Since $S^{-1}FQ_i$ is flat over $FQ_i$, we get
\[
  H_n(\cald_{FQ_i} \otimes_{\IZ G} C_*(EG/\Gamma_i)) \cong S^{-1}FQ_i \otimes_{FQ_i} H_n(B\Gamma_i;F).
\]
As $\Gamma_i$ is a subgroup of the free group $\Gamma$ and hence itself free, there is a
$1$-dimensional $CW$-model for $B\Gamma_i$. Therefore $H_n(B\Gamma_i;F)$ and hence also
$H_n(\cald_{FQ_i} \otimes_{\IZ Q_i} C_*(EG/\Gamma_i))$ vanish for $n \ge 2$. This implies
$b^{(2)}_n(EG;\cald_{FG}) = 0$ or $n \ge 2$.
Now Theorems~\ref{the:zeroth_L2-Betti_number} and~\ref{the:Euler-Poincare_formula} imply
\begin{eqnarray*}
    b^{(2)}_1(EG;\cald_{FG})  & = & -\chi(\widetilde{M}/G) ;
    \\
    b^{(2)}_n(EG;\cald_{FG}) & = & 0 \quad \text{if} \;n \not= 1.
    \end{eqnarray*}    

    Since $\chi(\widetilde{M}/G) = [\pi :G] \cdot \chi(M)$ holds,
    Lemma~\ref{lem:second_Betti_number_of_irreducible_3-,manifolds_with_boundary} follows from
   Remark~\ref{rem:Virtually_weak_Hughes_groups}.
\end{proof}

In the above proof, we used a special case of the following lemma.  

\begin{lemma}\label{lem:ALI_and_RALI_and_extensions}
  Let $G=\Gamma\rtimes A$. Suppose that $A$ is \ALI, i.e., amenable and locally indicable,
  and that $\Gamma$ has a chain of characteristic subgroups
  $\Gamma=\Gamma_0\supset\Gamma_1\supset\dots$ such that
  $\bigcap_{i = 0}^{\infty} \Gamma_i= \{1\}$ and each $\Gamma/\Gamma_i$ is {\ALI}. Then $G$ is
  \RALI.
\end{lemma}
\begin{proof}
  Since $\Gamma_i$ is characteristic, the action of $A$ on $\Gamma$ preserves $\Gamma_i$,
  so $\Gamma_i$ is normal in $G=\Gamma\rtimes A$, and the quotient $G/\Gamma_i$ can be
  identified with the semidirect product $(\Gamma/\Gamma_i)\rtimes A$. Since both $A$ and
  $\Gamma/\Gamma_i$ are \ALI, this semidirect product is, as well. This proves that $G$ is
  \RALI.
\end{proof}

\begin{remark}\label{rem:examples_of_RALI-groups}
  If $\Gamma$ is residually torsion-free nilpotent (\RTFN), then it satisfies the
  hypothesis in Lemma~\ref{lem:ALI_and_RALI_and_extensions}, because for each $i$ it has a
  maximal torsion-free $i$-step nilpotent quotient $Q_i$, this quotient is amenable and
  locally indicable, so we can take $\Gamma_i$ to be the kernel of
  $\Gamma\rightarrow Q_i$. Note that we have inclusions
\[
\textup{special groups}\subset\mbox{subgroups of } \RAAG \subset \RTFN\subset \RALI\subset \LI,
\]
where $\RAAG$ stands for  right-angled Artin group and $\LI$ for locally indicable.
So, in particular, all special groups and all subgroups of right angled Artin groups
satisfy the hypothesis of Lemma~\ref{lem:ALI_and_RALI_and_extensions}.
\end{remark}

\begin{lemma}\label{lem:fundamental_group_of_fibered_3-manifolds}
  Let $M$ be a $3$-manifold which fibers, i.e., there is a local trivial fiber
  bundle $F \to M \to S^1$ with a compact $2$-manifold as fiber.

  Then $\pi_1(M)$ is a $\RALIRF_p$-group for every prime $p$,
  see Definition~\ref{def:strongly_RALIRF}, and in particular a weak Hughes group,
and we get for every field $F$ and $n \in \IZ^{\ge 0}$
  \[
  b^{(2)}_n(\widetilde{M};\cald_{F\pi}) = 0.
  \]
\end{lemma}
\begin{proof}The proof of~\cite[Lemma~7.5]{Friedl-Lueck(2019Euler)} implies that
  $\pi_1(M)$ is residually (torsionfree finitely generated nilpotent). Since every
  torsionfree finitely generated nilpotent group is residually $p$-finite,
  see~\cite[Theorem~2.1]{Gruenberg(1957)}, amenable,  and locally indicable,
  $\pi_1(M)$ is a $\RALIRF_p$-group for every prime $p$.
  We get $b^{(2)}_n(\widetilde{M};\cald_{F\pi}) = 0$ for $n \in \IZ^{\ge 0}$ from
  Theorem~\ref{the:vanishing_of_L2-Betti_numbers_for_mapping_torus}.
\end{proof}

\begin{lemma}\label{lem:3-manifolds_and_locally_indicable_and_Hughes_groups}
  Let $G$ be a finitely generated group which is the fundamental group of a connected
  $3$-dimensional manifold $M$.  Then:

  \begin{enumerate}
  \item\label{lem:3-manifolds_and_locally_indicable_and_Hughes_groups:virtually_locally_indicable}
    The group $G$ is virtually locally indicable;

  \item\label{lem:3-manifolds_and_locally_indicable_and_Hughes_groups:locally_indicable_implies_Hughes}
  If $G$ is locally indicable, then $G$ is  a Hughes group in the sense of
  Definition~\ref{def:(weak)_Hughes_group};

  \item\label{lem:3-manifolds_and_locally_indicable_and_Hughes_groups:virtually_Hughes}
    The group $G$ is virtually a Hughes group in the sense of
    Definition~\ref{def:(weak)_Hughes_group}.
  \end{enumerate}
\end{lemma}
\begin{proof}We can assume without loss of generality that $M$ is compact by~\cite{Scott(1973)}.  
  \\[1mm]~\ref{lem:3-manifolds_and_locally_indicable_and_Hughes_groups:virtually_locally_indicable}
  We can assume without of loss of generality
  that $M$ is compact connected orientable, otherwise pass to the orientation covering. 
  One can write $G$ as a finite amalgamated  product $\ast_{i = 1}^r G_i$,
  where each $G_i$ is infinite cyclic or the fundamental group  of an irreducible compact connected orientable  
  $3$-manifold, see~\cite[Lemma~3.13 on page~28 and Theorems~3.15 on page~31]{Hempel(1976)}. 

  If $H_0$ and $H_1$ are virtually locally indicable groups, then the free amalgamated
  product $H_0 \ast H_1$ is virtually locally indicable by the following argument. There
  is an exact sequence
  $\{1\} \to F \to H_0 \ast H_1 \xrightarrow{p} H_0 \times H_1 \to \{1\}$, where $p$ is
  the canoncial projection and $F$ is a free group,
  see~\cite[Theorem~2]{Lyndon(1973)}. Obviously $H_0 \times H_1$ is virtually locally
  indicable. Let $K \subseteq H_0 \times H_1$ be a subgroup  of finite index which
  is locally indicable. Hence we get a short exact sequence
  $\{1\} \to F \to p^{-1}(K) \to K \to \{1\}$. Since $F$ and $L$ are locally indicable,
  $p^{-1}(L)$ is locally indicable.  As $p^{-1}(L)$ has finite index in $H_0 \ast H_1$,
  the group $H_0 \ast H_1$ is virtually locally indicable.

  Hence we can assume without loss of generality that $G = \pi_1(M)$ for an orientable
  compact irreducible $3$-manifold $M$ and $G$ is infinite.

Suppose that $M$ has a non-empty boundary. Then there exists an extension
$\{1\} \to F \to \pi_1(M) \to \IZ \to \{1\}$ for a free group $F$
by~\cite[Theorem~1.1]{Kielak-Linton(2023grouprings)}.  This implies that $\pi_1(M)$ is
locally indicable as $F$ and $\IZ$ are locally indicable.  Hence we can assume
additionally that $\partial M$ is empty. If $\pi_1(M)$ is solvable, the claim follows
from~\cite[Theorem~2.20]{Aschenbrenner-Friedl-Wilton(2015)}.  If $\pi_1(M)$ is not
solvable, the claim follows from~\cite[Diagram~1]{Aschenbrenner-Friedl-Wilton(2015)}.
This finishes the proof that $G$ is virtually locally indicable.
\\[1mm]~\ref{lem:3-manifolds_and_locally_indicable_and_Hughes_groups:locally_indicable_implies_Hughes}
This is proved in~\cite[Theorem~1.3]{Fisher-Sanchez-Peralta(2023)}.
\\[1mm]~\ref{lem:3-manifolds_and_locally_indicable_and_Hughes_groups:virtually_Hughes}
This follows from
assertions\ref{lem:3-manifolds_and_locally_indicable_and_Hughes_groups:virtually_locally_indicable}
and~\ref{lem:3-manifolds_and_locally_indicable_and_Hughes_groups:locally_indicable_implies_Hughes}~.
\end{proof}

 An  irreducible 3-manifold with empty or toroidal boundary
  is called a \emph{graph manifold} if all components in the Jaco-Shalen-Johannson decomposition
  are Seifert fibered manifolds.

\begin{lemma}\label{lem:graph_manifolds}
  Let $M$ be a $3$-manifold which is a graph manifold and whose fundamental group $\pi = \pi_1(M)$ is infinite.
  Then we get for every field $F$ and $n \in \IZ^{\ge 0}$
  \[
  b^{(2)}_n(\widetilde{M};\cald_{F\pi}) = 0.
\]
\end{lemma}
\begin{proof}
  In view of Remark~\ref{rem:Virtually_weak_Hughes_groups} and
  Lemma~\ref{lem:3-manifolds_and_locally_indicable_and_Hughes_groups}~%
\ref{lem:3-manifolds_and_locally_indicable_and_Hughes_groups:virtually_Hughes}
  we can assume without loss of generality that $G$ is a non-trivial Hughes group since a
  finite covering of a graph manifold is again a graph manifold.
  
  The pieces in the Jaco-Shalen decomposition of $M$ are Seifert manifolds $M_1$, $M_2$,
  \ldots, $M_r$ with incompressible toroidal boundary. For each $i \in M_i$ the inclusion
  $M_i \to M$ induces an injection from $\pi_1(M_i)$ into $\pi_1(M)$ and hence is a weak
  Hughes group by
  Lemma~\ref{lem:subgroups_of_(weak)_Hughes-groups_are_(weak)-Hughes_groups}.  Since
  $\pi_1(M_i)$ is infinite , there exists a finite covering
  $p \colon \overline{M_i} \to M$ such that there is a principal $S^1$-fiber bundle $S^1 \to \overline{M_i} \to F$
  for some orientable compact $2$-manifold $F$, see~\cite[pages~426,~427, and~436]{Scott(1983)}.  Hence
  $b_n^{(2)}(M_i;\cald_{F[\pi_1(M_i)]})$ vanishes for all $n \in \IZ$ by
  Lemma~\ref{the:restriction}~\ref{the:restriction:general_F} and
  Theorem~\ref{the:vanishing_of_L2-Betti_numbers_for_mapping_torus}.  Each component $C$
  of $\partial M_i$ is a torus and hence $b_n^{(2)}(\widetilde{C};\cald_{F[\pi_1(C)]})$
  vanishes for all $n \in \IZ$ by
  Theorem~\ref{the:vanishing_of_L2-Betti_numbers_for_mapping_torus}. Moreover the
  inclusion $C \to M$ induces an injection $\pi_1(C) \to \pi_1(M)$. Now a Mayer-Vietoris
  argument and Theorem~\ref{the:induction}~\ref{the:induction:general_F} show that
  $b_n(\widetilde{M};\cald_{F\pi}) = 0$ holds for every $n \in \IZ$.
\end{proof}

\begin{lemma}[$L^2$-Betti numbers over fields for prime $3$-manifolds]%
\label{lem:L2-Betti_numbers_over_fields_for_prime_3-manifolds}
Let $M$ be a prime  $3$-manifold such that its fundamental group  $\pi$ is infinite.

Then $\pi$ is virtually an infinite  weak Hughes group and we get for every field $F$:
    \begin{eqnarray*}
    b^{(2)}_1(\widetilde{M};\cald_{F\pi})  & = & -\chi(M);
    \\
   b^{(2)}_n(\widetilde{M};\cald_{F\pi}) & = & 0 \quad \text{if} \;n \not= 1;
      \\
      b^{(2)}_n(\widetilde{M};\cald_{F\pi}) & = & b^{(2)}_n(\widetilde{M};\caln(\pi))  \; \text{for every}\; n \in \IZ^{\ge 0}.
    \end{eqnarray*}                          
  \end{lemma}
  \begin{proof}
    If $M$ is homeomorphic to the total space over an $S^2$-bundle over $S^1$,
    the claim follows from Theorem~\ref{the:vanishing_of_L2-Betti_numbers_for_mapping_torus}.
    Since any  prime $3$- manifold which is not homeomorphic 
    to the total space over an $S^2$-bundle over $S^1$,
    is irreducible, see~\cite[Lemma~3.13 on page~28]{Hempel(1976)} and
    we have Lemma~\ref{lem:second_Betti_number_of_irreducible_3-,manifolds_with_boundary},
    we can assume without loss of generality that $M$ is an irreducible closed
    $3$-manifold with infinite fundamental group $\pi$. If $M$ is a closed graph manifold, the
    claim follows from Lemma~\ref{lem:graph_manifolds}.  Suppose that
    $M$ is a not a graph manifold. Then by the proof of the Virtual Fibering Theorem due
    to Agol, Liu, Przytycki-Wise, and
    Wise~\cite{Agol(2008),Agol(2013),Liu(2013),Przytycki-Wise(2018),
      Przytycki-Wise(2014),Wise(2012raggs),Wise(2012hierachy)} there exists a finite
    normal covering $p \colon \overline{M} \to M$ and a fiber bundle
    $F \to \overline{M} \to S^1$ for some compact connected orientable surface $F$. Now
    the claim follows from Lemma~\ref{the:restriction}~\ref{the:restriction:general_F} and
    Lemma~\ref{lem:fundamental_group_of_fibered_3-manifolds}.
\end{proof}

Now we are ready to give the proof of Theorem~\ref{the:L2-Betti_numbers_of_3-manifolds_for_arbitrary_fields}.

\begin{proof}[Proof of Theorem~\ref{the:L2-Betti_numbers_of_3-manifolds_for_arbitrary_fields}]
Since $b_n^{(2)}(\widetilde{M};\cald_{F\pi})$,
$\chi^{(2)}(\pi_1(M))$, and $\chi(M)$ are multiplicative under finite coverings,
see Theorem~\ref{the:restriction} and~\cite[Theorem~6.80~(7) on page~279]{Lueck(2002)},and we have
Lemma~\ref{lem:3-manifolds_and_locally_indicable_and_Hughes_groups:virtually_Hughes}~%
\ref{lem:3-manifolds_and_locally_indicable_and_Hughes_groups}, we can assume loss of generality that
$M$ is orientable and $\pi = \pi_1(M)$ is a Hughes group. 
We conclude
\[
  b_n^{(2)}(\widetilde{M};\cald_{F\pi}) = 0\; \text{for} \; n \notin \{1,2\}
    \]
    from Theorem~\ref{the:zeroth_L2-Betti_number}, Theorem~\ref{the:Poincare_duality},
    and the fact that $M$ is homotopy equivalent to $2$-dimension $CW$-complex if $\partial M$ is not empty.
    Theorem~\ref{the:Euler-Poincare_formula} implies
    \[\chi(M) = b_2^{(2)}(\widetilde{M};\cald_{F\pi}) - b_1^{(2)}(\widetilde{M};\cald_{F\pi}).
     \]
     Since the classifying map $M \to B\pi_1(M)$ is $2$-connected, we get 
     \[b_1^{(2)}(\pi;\cald_{F\pi}) = b_n^{(2)}(\widetilde{M};\cald_{F\pi}).
      \]
      Hence it suffices to show
      \[
        b_1^{(2)}(\pi;\cald_{F\pi})  = \chi^{(2)}(\pi).
     \]
      Since
\[
\pi_1(M) \cong \ast_{j = 1}^r \pi_1(M_j)
\]
holds, we get from a Mayer-Vietoris argument using Theorem~\ref{the:induction}
\begin{eqnarray*}
  b_1^{(2)}(\pi;\cald_{F\pi}) & = & (r - 1) + \sum_{j = 1}^r b_1^{(2)}(\pi_1(M_j);\cald_{F[\pi_1(M_j)]});
  \\
  \chi^{(2)}(\pi) & = & (r - 1) + \sum_{j = 1}^r \chi^{(2)}(\pi_1(M_j)).
\end{eqnarray*}
 Hence it suffices to show     
 \[b_1^{(2)}(\pi_1(M_j);\cald_{F[\pi_1(M_j)]}) = \chi^{(2)}(\pi_1(M_j))
 \]
 for $j = 1,2  \ldots , r$.
 If $\pi_1(M_j)$ is finite, this follows from  Remark~\ref{rem:finite_pi}.
Suppose  that $\pi_1(M_j)$ is infinite.
 Since $\chi^{(2)}(\pi_1(M_j)) = b_1^{(2)}(\widetilde{M_j}; \caln(\pi_1(M_j)))$
 holds by Theorem~\ref{the:L2-Betti_numbers_of_3-manifolds_for_arbitrary_fields}, it remains to show
  \[b_1^{(2)}(\pi_1(M_j);\cald_{F[\pi_1(M_j)]}) = b_1^{(2)}(\pi_1(M_j);\caln(\pi_1(M_j))).
  \]
  This follows  Lemma~\ref{lem:cald_(FG)_is_cald(FG_subseteq_calu(G))}~%
\ref{lem:cald_(FG)_is_cald(FG_subseteq_calu(G)):equality_of_L2-Betti_numbers}
and Lemma~\ref{lem:L2-Betti_numbers_over_fields_for_prime_3-manifolds}.
This finishes the proof of 
Theorem~\ref{the:L2-Betti_numbers_of_3-manifolds_for_arbitrary_fields}.
\end{proof}

%-----------------------------------------------------------------------------

\subsection{Groups containing normal infinite  amenable subgroups}%
\label{subsec:Groups_containing_normal_infinite_elementary_subgroups}

Given a group $G$, define
\begin{equation}
  b_n^{(2)}(G;\caln(G)) := b_n(EG;\caln(G)).
  \label{b_n_upper_(2)(G;caln(G))}
\end{equation}
Given a weak Linnell group  $G$ and any field $F$, define
\begin{equation}
  b_n^{(2)}(G;\cald_{FG}) := b_n^{(2)}(EG;\cald_{FG}).
  \label{b_n_upper_(2)(G;cald((FG)}
\end{equation}

\begin{theorem}[Vanishing of $L^2$-Betti numbers for groups containing a normal infinite amenable subgroup]%
  \label{Vanishing_of_L2-Betti_numbers_for_groups_containing_a_normal_infinite_amenable_subgroup}
  Let $G$ be a group which contains a normal infinite amenable subgroup.

  \begin{enumerate}

  \item\label{Vanishing_of_L2-Betti_numbers_for_groups_containing_a_normal_infinite_amenable_subgroup:caln(G)}
  Then we get $b_n^{(2)}(G;\caln(G)) = 0$ for $n \in \IZ^{\ge 0}$;

\item\label{Vanishing_of_L2-Betti_numbers_for_groups_containing_a_normal_infinite_amenable_subgroup:arbitrary_F}
  If $G$ is a weak Hughes  group,   then we get  $b^{(2)}_n(G;\cald_{FG}(G)) = 0$ for $n \in \IZ^{\ge 0}$.
\end{enumerate}
\end{theorem}
\begin{proof}~\ref{Vanishing_of_L2-Betti_numbers_for_groups_containing_a_normal_infinite_amenable_subgroup:caln(G)}
 This is proved in~\cite[Theorem~7.2~(1) on page~294]{Lueck(2002)}.
 \\[1mm]~\ref{Vanishing_of_L2-Betti_numbers_for_groups_containing_a_normal_infinite_amenable_subgroup:arbitrary_F}
 Let $H \subseteq G$ be a normal infinite amenable  subgroup.
 We conclude from Theorem~\ref{the:L2-Betti_numbers_and_fibrations}~\ref{the:L2-Betti_numbers_and_fibrations:general_F_d_is_infty}
 applied to the fibration $BH \to BG \to BG/H$ that it suffices to prove the claim for $H$.  Hence
 we can assume without loss of generality that $G$ is locally indicable and amenable.
 Remark~\ref{rem:Amenable_locally_indicable_groups} implies $b_n(G;\cald_{FG}) = 0$ for $n \in \IZ^{\ge 1}$
as EG is contractible. We get $b^{(2)}_0(G;\cald_{FG}) = 0$ from
Theorem~\ref{the:zeroth_L2-Betti_number}~\ref{the:zeroth_L2-Betti_number:general_F}.
\end{proof}

The next theorem taken from~\cite[Corollary~1.13]{Lueck(2013l2approxfib)}.

\begin{theorem}\label{the:Groups_containing_a_normal_infinite_nice_subgroups}
Let $M$ be an aspherical closed manifold with fundamental group $G = \pi_1(M)$.
Suppose that $M$ carries a non-trivial $S^1$-action or suppose
that $G$ contains a non-trivial  elementary amenable normal subgroup. Let
$G = G_0 \supseteq G_1 \supseteq G_2 \supseteq \cdots$ be a sequence of
in $G$ normal subgroups of finite index $[G:G_i]$ such that $\bigcap_{i = 0}^{\infty} G_i = \{1\}$ holds

Then we get for every field $F$ and every  $n \ge 0$
\begin{eqnarray*}
\lim_{i \to \infty} \frac{b_n(G_i\backslash \widetilde{M};F)}{[G:G_i]}  
& = & 
0.
\end{eqnarray*}
\end{theorem}

The next result is due to Ab\'ert-Nikolov~\cite[Theorem~3]{Abert-Nikolov(2012)}.

\begin{theorem}\label{the_Abert-Nikolov}
  Let $G$ be a finitely presented residually finite group $G$ which contains a normal
  infinite amenable subgroup.  Let $G = G_0 \supseteq G_1 \supseteq G_2 \supseteq \cdots$
  be a sequence of in $G$ normal subgroups of finite index $[G:G_i]$ such that
  $\bigcap_{i = 0}^{\infty} G_i = \{1\}$ holds

  Then the rank gradient $RG(G;(G_i)_{i\in I})$ vanishes and we get for every field $F$
  \begin{eqnarray*}
    \lim_{i \to \infty} \frac{b_1(G_i;F)}{[G:G_i]}  
    & = & 
          0.
  \end{eqnarray*}
\end{theorem}

Theorems~\ref{the:Groups_containing_a_normal_infinite_nice_subgroups} and
Theorem~\ref{the_Abert-Nikolov} are interesting in connection with
Theorem~\ref{the:Identification_with_adhoc_definition}
and~\ref{Vanishing_of_L2-Betti_numbers_for_groups_containing_a_normal_infinite_amenable_subgroup}.

% ----------------------------------------------------------------------------

\subsection{The first $L^2$-Betti number and group extensions}%
\label{subsec:The_first_L2_Betti_number_and_group_extensions}

\begin{theorem}[The first $L^2$-Betti number and extensions]%
\label{the:The_first_L2_Betti_number_and_group_extensions}
Let $1 \to K \to G \to Q \to 1$ be an extension  of infinite groups. Suppose that
that $Q$ contains an element of infinite order or finite subgroups of arbitrarily large order.

\begin{enumerate}

\item\label{the:The_first_L2_Betti_numbers_and_extensionsN(G)}
  If $b_1^{(2)}(K;\caln(K)) < \infty$ holds, then $b_1^{(2)}(G;\caln(G)) = 0$; 

 \item\label{the:The_first_L2_Betti_numbers_and_extensions:general_F}
 Suppose that $G$ is a weak Hughes group  and  that $b^{(2)}_1(K;\cald_{FK}) < \infty$ holds.
     Then $b^{(2)}_1(G;\cald_{FG}) = 0$.
 \end{enumerate}
\end{theorem}
\begin{proof} This follows from Theorem~\ref{the:zeroth_L2-Betti_number}
and Theorem~\ref{the:L2-Betti_numbers_and_fibrations}  applied to the fibration $BK \to BG \to BQ$.
\end{proof}

% ---------------------------------------------------------------------------- be an extension of g

\subsection{$L^2$-Betti numbers and $S^1$-actions}\label{L2-Betti_numbers_and_S1_actions}

\begin{theorem}[$L^2$-Betti numbers and $S^1$-actions]\label{the:S1-action_and_L2-Betti_numbers}
Let $X$ be a connected $S^1$-$CW$.
Suppose that for one orbit $S^1/H$ (and hence for all orbits) the inclusion
into $X$ induces a map on $\pi_1$ with infinite image.
(In particular the $S^1$-action has no fixed points.)
Let $\widetilde{X}$ be the universal covering of $X$ with the canonical
$\pi_1(X)$-action.

\begin{enumerate}
\item\label{the:S1-action_and_L2-Betti_numbers:N(G)}
We get $b^{(2)}_n\bigl(\widetilde{X};\caln(\pi_1(X))\bigr)$ for all $n \in \IZ^{\ge 0}$;

\item\label{the:S1-action_and_L2-Betti_numbers:general_F}
  Suppose  that $\pi_1(X)$ is a weak Hughes group. Then we get
  $b^{(2)}_n(\widetilde{X};\cald_{F[\pi_1(X)]})$ for all $n \in \IZ^{\ge 0}$;
\end{enumerate}
\end{theorem}
\begin{proof}~\ref{the:S1-action_and_L2-Betti_numbers:N(G)}
This is proved in~\cite[Theorem~6.65~on page~271]{Lueck(2002)}.
\\[1mm]~\ref{the:S1-action_and_L2-Betti_numbers:general_F}
The proof of~\cite[Theorem~6.65~on page~271]{Lueck(2002)} carries over to this setting.
\end{proof}

\begin{theorem}\label{the:fixed_point_free_S1-actions_on_aspherical_closed_manifolds_and_L2-Betti_numbers}
Let $M$ be an aspherical closed manifold with
non-trivial $S^1$-action. 

\begin{enumerate}

\item~\label{the:fixed_point_free_S1-actions_on_aspherical_closed_manifolds_and_L2-Betti_numbers_no_fixed_points}
Then the action has no fixed
points and the inclusion of every $S^1$-orbit into $X$ induces an
injection on the fundamental groups. 

\item\label{the:fixed_point_free_S1-actions_on_aspherical_closed_manifolds_and_L2-Betti_numbers:N(G)}

  We get $b_n^{(2)}\bigl(\widetilde{M};\caln(\pi_1(M))\bigr)= 0$ for $n \ge \IZ^{\ge 0}$;

\item\label{the:fixed_point_free_S1-actions_on_aspherical_closed_manifolds_and_L2-Betti_numbers:general_F}
If $\pi_1(X)$ is  weak Hughes  group, then we get $b^{(2)}_n(\widetilde{M};\cald_{F[\pi_1(M)]})= 0 $ for $n \ge \IZ^{\ge 0}$.
\end{enumerate}
\end{theorem}
\begin{proof}
Assertion~\ref{the:fixed_point_free_S1-actions_on_aspherical_closed_manifolds_and_L2-Betti_numbers}
follows from~\cite[Lemma~1.42 on page~45]{Lueck(2002)}.
Now apply Theorem~\ref{the:S1-action_and_L2-Betti_numbers}.
\end{proof}

%------------------------------------------------------------------------

  \subsection{Selfmaps of aspherical closed manifolds}%
  \label{subsec:Selfmaps_of_aspherical_closed_manifolds}

  A group $G$ is \emph{Hopfian} if any surjective group homomorphism $f\colon G \to G$ is
  an isomorphism.  Examples of Hopfian groups are residually finite
  groups~\cite{Malcev(1940)},~\cite[Corollary 41.44]{Neumann(1967)}.  It is not true that
  a subgroup of finite index of a Hopfian group is again Hopfian,
  see~\cite[Theorem~2]{Baumslag-Solitar(1962)}. At least any subgroup of finite index in a
  residually finite group is again residually finite and hence Hopfian.

\begin{theorem}\label{the:L2-Betti_numbers_for_aspherical_closed_manifold_with_selfmap}
Let $M$ be an aspherical closed orientable manifold. Suppose that any
normal subgroup of finite index of its fundamental group is Hopfian.

\begin{enumerate}
\item\label{the:L2-Betti_numbers_for_aspherical_closed_manifold_with_selfmap:Q}
  Suppose that there is a selfmap $f\colon  M \to M$ of degree
$\deg(f)$ different from $-1$, $0$, and $1$.

Then we get $b_n^{(2)}(\widetilde{M})  =  0$ for $n \in \IZ^{\ge 0}$;

\item\label{the:L2-Betti_numbers_for_aspherical_closed_manifold_with_selfmap:F_p} Let $p$
  be a prime. Suppose that $\pi_1(M)$ is a weak Hughes group and that there is a selfmap
  $f\colon M \to M$ such that its degree $\deg(f)$ is different from $-1$, $0$, and $1$
  and not divisible by $p$.  Let $F$ be any field of characteristic $p$.

Then we get $b^{(2)}_n(\widetilde{M};\cald_{F[\pi_1(M)]})  =  0$ for $n \in \IZ^{\ge 0}$.
\end{enumerate}
\end{theorem}
\begin{proof}~\ref{the:L2-Betti_numbers_for_aspherical_closed_manifold_with_selfmap:Q}
This is proved in~\cite[Theorem~14.40 on Page~499]{Lueck(2002)}.
\\[1mm]~\ref{the:L2-Betti_numbers_for_aspherical_closed_manifold_with_selfmap:F_p}
Fix an integer $m \ge 1$.
Let $q_m \colon  \overline{M} \to M$ be the covering of $M$ associated to the image of
$\pi_1(f^m)\colon  \pi_1(M) \to \pi_1(M)$. By elementary covering theory there is a map
$\overline{f^m}\colon  M \to \overline{M}$ satisfying $q_m \circ \overline{f^m} = f^m$.
Since $\deg(f^m) = \deg(\overline{f^m}) \cdot \deg(q_m)$ and $\deg(q_m) = [\pi_1(M) \colon
\im(\pi_1(f))]$, we conclude
\begin{eqnarray}
[\pi_1(M): \im(\pi_1(f^m))] & \le & |\deg(f^m)|  =  |\deg(f)|^m.
\label{[pi_1(X):pi_1(f_upper_m)]_le_m_cdot_deg(f)}
\end{eqnarray}

Consider a map $g\colon  N_1 \to N_2$ of two closed connected oriented $d$-dimensional
manifolds of degree $\deg(g)$ which is not divisible by $p$. Abbreviate $G_1 = \pi_1(N_1)$ and
$G_2 = \pi_1(N_2)$. Then
there is a diagram of $\IZ G_2$-chain complexes which commutes up to
$\IZ G_2$-homotopy
\[
  \xymatrix@!C=17em{\hom_{\IZ G_1}(C_{d-*}(\widetilde{N_1}),\IZ G_1) \otimes_{\IZ[\pi_1(g)]} \IZ G_2
    \ar[d]_{({?} \cap [N_1])\otimes_{\IZ[\pi_1(g)]} \id_{\IZ G_2}}
 &
 \hom_{\IZ G_2}(C_{d-*}(\widetilde{N_2}),\IZ G_2) \ar[l]_-{g^*}
 \ar[d]^{{?} \cap g_*([N_1])}
\\
C_*(\widetilde{N_1}) \otimes_{\IZ[\pi_1(g)]} \IZ G_2
\ar[r]_-{g_*}
&
C_*(\widetilde{N_2}).
}
\]
Applying ${?} \otimes_{\IZ G_2}\cald_{FG_2}$ and then  homology yields
a commutative diagram of finitely generated 
$\cald_{FG_2}$-modules for $n \in \IZ^{\ge 0}$
\[
\xymatrix@!C=19em{H_{m}\bigl(\hom_{\IZ G_1}(C_{d-*}(\widetilde{N_1}),\IZ G_1) \otimes_{\IZ[\pi_1(g)]}
\cald_{FG_2}\bigr) \ar[d]
&
H_{m}\bigl(\hom_{\IZ G_2}(C_{d-*}(\widetilde{N_2},\IZ G_2)\otimes_{\IZ G_2} \cald_{FG_2})\bigr)
\ar[l] \ar[d]^{H_*(? \cap g_*([N_1]))}
\\
H_n\bigl(C_*(\widetilde{N_1}) \otimes_{\IZ[\pi_1(g)]} \cald_{FG_2}\bigr)
\ar[r]_-{H_n(g_* \otimes_{\IZ[\pi_1(g)]} \id_{\cald_{FG_2}})}
&
H_n\bigl(C_*(\widetilde{N_2})\otimes_{\IZ G_2} \cald_{FG_2}\bigr).}
\]
The right vertical arrow is bijective since $\deg(g)$ is invertible in $F$ and hence in $\cald_{FG_2}$,
we have $g_*([N_1]) = \deg(g) \cdot [N_2]$, and
${?} \cap [N_2] \colon  \hom_{\IZ G_2}(C_{n-*}(\widetilde{N_2}),\IZ G_2) \to C_*(\widetilde{N_2})$ is a $\IZ[G_2]$-chain
homotopy equivalence by Poincar\'e duality.
Hence the morphism of finitely generated Hilbert $\caln(G_2)$-modules
\[
  H_n(g_* \otimes_{\IZ[\pi_1(g)]} \id_{\cald_{FG_2}})  \colon
  H_n\bigl(C_*(\widetilde{N_1}) \otimes_{\IZ[\pi_1(g)]} \cald_{FG_2}\bigr)
  \to H_n(C_*(\widetilde{N_2}) \otimes_{\IZ G_2} \cald_{FG_2})
  \]
is surjective. This implies
\[
  \dim_{\cald_{FG_2}}\bigl(H_n\bigl(C_*(\widetilde{N_1}) \otimes_{\IZ[\pi_1(g)]} \cald_{FG_2}\bigr)\bigr)
  \ge b^{(2)}_n(\widetilde{N_2};\cald_{FG_2}).
  \]
  Let $e_n(N_1)$ be the number of $n$-cells in a fixed triangulation of $N_1$. Then
  \[
  \dim_{\cald_{FG_2}}\bigl(H_n\bigl(C_*(\widetilde{N_1}) \otimes_{\IZ[\pi_1(g)]} \cald_{FG_2}\bigr)\bigr)
  \le 
  \dim_{\cald_{FG_2}}\bigl(C_n(\widetilde{N_1}) \otimes_{\IZ[\pi_1(g)]} \cald_{FG_2}\bigr)
  \le
 e_p(N_1).
\]
We conclude for $n \in \IZ^{\ge 0}$
\begin{equation}
b^{(2)}_n(\widetilde{N_2};\cald_{FG_2})  \le e_n(N_1).
\label{b_n(widetilde(N_2);cald_FG_2)_le_e_n(N_1)}
\end{equation}
If we apply~\eqref{b_n(widetilde(N_2);cald_FG_2)_le_e_n(N_1)} to $g = \overline{f^n}$, we get
\[
b^{(2)}_n(\widetilde{\overline{M}};\cald_{F[\pi_1(\overline{M})]})  \le  e_n(M).
\]
Since $b^{(2)}_n(\widetilde{\overline{M}};\cald_{F[\pi_1(\overline{M})]}) = [\pi_1(M) : \im(\pi_1(f^m))] \cdot
b_n(\widetilde{M};\cald_{F[\pi_1(M)]})$ holds by
Theorem~\ref{the:restriction}~\ref{the:restriction:general_F}, we get
\begin{equation}
b^{(2)}_n(\widetilde{M};\cald_{F[\pi_1(M)]})  \le  \frac{e_n(M)}{[\pi_1(M) : \im(\pi_1(f^m))]}.
\label{b_n(widetilde(M);cald_(Fpi_1(M)))_le_e_n/index}
\end{equation}
Hence it suffices to show that there is no integer $m$ such that
$\im(\pi_1(f^m)) = \im(\pi_1(f^{k}))$ holds for all $k \ge m$ since then the limit for
$m \to \infty$ of the right-hand side of~\eqref{b_n(widetilde(M);cald_(Fpi_1(M)))_le_e_n/index} is zero.

Suppose that such $m$ exists.  Then the composite
$\overline{f^m} \circ q_m\colon \overline{M} \to \overline{M}$ induces an epimorphism on
$\pi_1(\overline{M})$ by the following argument. Consider $x \in \pi_1(\overline{M})$.
Since $\im(\pi_1(f^{2m})) = \im(\pi_1(f^m)) = \im(\pi_1(q_m))$ and
$f^{2m} = q_m \circ \overline{f^m} \circ q_m \circ \overline{f^m}$ hold, we can find
$y \in \pi_1(M)$ satisfying
$\pi_1(q_m) \circ \pi_1(\overline{f^m} \circ q_m) \circ \pi_1(\overline{f^m})(y)=
\pi_1(q_m)(x)$.  Put $z = \pi_1(\overline{f^m})(y)$. Then we get
$\pi_1(q_m)\bigl(\pi_1(\overline{f^m} \circ q_m)(z)\bigr) = \pi_1(q_m)(x)$.  As
$\pi_1(q_m)$ is injective, we conclude $\pi_1(\overline{f^m} \circ q_m)(z) = x$. Since
$\pi_1(\overline{M})$ is Hopfian by assumption,
$\overline{f^m} \circ q_m\colon \overline{M} \to \overline{M}$ induces an isomorphism on
$\pi_1(\overline{M})$. Since $\overline{M}$ is aspherical, $\overline{f^m} \circ q_m$ is a
homotopy equivalence. This implies
\[
1  =  \deg(\overline{f^m} \circ q_m)  = 
\deg(\overline{f^m}) \cdot \deg(q_m).
\]
This shows that $|\deg(q_m)| = 1$ and hence $\pi_1(f^m)$ is surjective.
Since $\pi_1(M)$ is Hopfian and $M$ aspherical, we conclude
$|\deg(f^m)| = |\deg(f)^m| = 1$. This contradicts the assumption
$\deg(f) \notin \{-1,0,1\}$. This finishes the proof of
Theorem~\ref{the:L2-Betti_numbers_for_aspherical_closed_manifold_with_selfmap}.
\end{proof}

  %-----------------------------------------------------------------------------

\subsection{Approximation results}\label{subsec:Approximation_results}

\begin{definition}\label{def:class_calg}
Let $\calg$ be the smallest class of groups $G$ satisfying:
\begin{enumerate}

\item\label{def:class_calg:amenable_quotient}
Amenable quotient\\
Let $H \subset G$ be a normal subgroup. Suppose that $H \in \calg$ and
the quotient $G/H$ is
an amenable. Then $G \in \calg$;

\item\label{def:class_calg:direct_limit}
Colimits\\
If $G = \colim_{i \in I} G_i$ is the colimit of the
directed system $\{G_i \mid i \in I\}$ of groups indexed by the
directed set $I$ and each $G_i$
belongs to $\calg$, then $G$ belongs to $\calg$;

\item\label{def:class_calg:inverse_limit}
Inverse limits\\
If $G = \lim_{i \in I} G_i$ is the limit of the
inverse system $\{G_i \mid i \in I\}$ of groups indexed by the
directed set $I$ and each $G_i$
belongs to $\calg$, then $G$ belongs to $\calg$;

\item\label{def:class_calg:subgroups}
Subgroups\\
If $H$ is isomorphic to a subgroup of the group $G$ with $G \in \calg$,
then $H \in \calg$;

\item\label{def:class_calg:quotient_with_finite_kernel}
Quotients with finite kernel\\
Let $1 \to K \to G \to Q \to 1$ be an exact sequence of groups. If $K$
is finite and $G$ belongs to $\calg$, then $Q$ belongs to $\calg$;

\item\label{def:class_calg:sofic} $G$ is a sofic group.
\end{enumerate}
\end{definition}

\begin{theorem}\label{the:Approximation_over_any_F}
Let $G$ be group and $G = G_0 \supseteq G_1 \supseteq G_2 \supseteq \cdots$ be a sequence of
in $G$ normal subgroups such that $\bigcap_{i = 0}^{\infty} G_i = \{1\}$ holds.
Let $X$ be a connected free $G$-$CW$-complex of finite type.

\begin{enumerate}

  \item\label{the:Approximation_over_any_F:special_F}
Suppose either  that $G$ is virtually locally indicable or  that each
quotient $G/G_i$ is belongs to $\calg$.

Then we get for every $n \in \IZ^{\ge 0}$
\[
  b_n^{(2)}(X;\caln(G)) = \lim_{i \to \infty} b_n^{(2)}(X/G_i;\caln(G/G_i));
\]

\item\label{the:Approximation_over_any_F:General_F}
  Suppose that each each quotient $G/G_i$ is  weak Hughes group.

  Then $G$ is a  weak Hughes group
 and  we get for  every $n \in \IZ^{\ge 0}$ and every field $F$
\[
  b^{(2)}_n(X;\cald_{FG}) = \lim_{i \to \infty} b_n^{(2)}(X/G_i;\cald_{F[G/G_i]}).
\]

\end{enumerate}
\end{theorem}
\begin{proof}~\ref{the:Approximation_over_any_F:special_F}
  This follows from~\cite[Theorem~5]{Elek-Szabo(2005)},~%
\cite[Theorem~1.5]{Jaikin-Zapirain+Lopez-Alvarez(2020)}~\cite[Chapter~13]{Lueck(2002)},
and~\cite[Theorem~1.21]{Schick(2001b)} taking~\cite{Schick(2002a)} into account.
\\[1mm]~\ref{the:Approximation_over_any_F:General_F}
   Since each $G/G_i$ is locally indicable, $G$ is locally indicable.

   Let $H$ be a finitely generated subgroup of $G$. Put $H_i = G_i \cap H$.  Then
   $H = H_0 \supseteq H_1 \supseteq H_2 \supseteq \cdots$ is a sequence of in $H$ normal
   subgroups such that $\bigcap_{i = 0}^{\infty} H_i = \{1\}$ holds. For each $i \in I$,
   $H/H_i$ is a subgroup if $G/G_i$ and hence admits a Hughes free $F[H/H_i]$-division
   ring by Lemma~\ref{lem:subgroups_of_(weak)_Hughes-groups_are_(weak)-Hughes_groups}~%
\ref{lem:subgroups_of_(weak)_Hughes-groups_are_(weak)-Hughes_groups:subgroups}.  Hence
   $H$ admits a Hughes free $F[H/H_i]$-division ring
   by~\cite[Theorem~1.2]{Jaikin-Zapirain(2021)}.  This implies that $H$ is a Hughes
   group. Hence $G$ is a Hughes group by
   Lemma~\ref{lem:subgroups_of_(weak)_Hughes-groups_are_(weak)-Hughes_groups}~%
\ref{lem:subgroups_of_(weak)_Hughes-groups_are_(weak)-Hughes_groups_subgroups:fin.gen.}.

   Next we want to show for every finitely presented $\IZ G$-module $M$
\begin{equation}
     \dim_{\cald_{FG}}(\cald_{FG} \otimes_{\IZ G} M)
     = \lim_{i \to \infty} \dim_{\cald_{F[G/G_i]}}(\cald_{F[G/G_i]}\otimes_{\IZ G} M).
     \label{the:Approximation_over_any_F:General_F_auxiliary}
   \end{equation}
   If $G$ is finitely generated, this follows from~\cite[Theorem~1.2]{Jaikin-Zapirain(2021)}.
   In the general case note that for any
   finitely presented $\IZ G$-module $M$ there exists a finitely generated subgroup
   $H \subseteq G$ and a  finitely presented $\IZ H$-module $N$ such that $M$ and
   $FG \otimes_{FH} N$ are $FG$-isomorphic. Put $H_i = H \cap G_i$. Then
   $H = H_0 \supseteq H_1 \supseteq H_2 \supseteq \cdots$ is a sequence of in $H$ normal
   subgroups such that $\bigcap_{i = 0}^{\infty} H_i = \{1\}$ holds and we get
   \[
   \dim_{\cald_{FH}}(\cald_{FH} \otimes_{\IZ H} N) = \lim_{i \to \infty} \dim_{\cald_{F[H/H_i]}}(\cald_{F[H/H_i]}\otimes_{\IZ H} N).
   \]
   Since we have
   \begin{eqnarray*}
     \cald_{FG} \otimes_{\IZ G} M
     &\cong_{\cald_{FG}} &
     \cald_{FG} \otimes_{\cald_{FH} } \otimes_{\IZ H} N;
     \\
     \cald_{F[G/G_i]}\otimes_{\IZ G} M
     &\cong_{\cald_{F[G/G_i]}} &
      \cald_{F[G/G_i]}\otimes_{\cald_{F[H/H_i]}}  \cald_{F[H/H_i]}\otimes_{\IZ H} N,
   \end{eqnarray*}                          
   equation~\ref{the:Approximation_over_any_F:General_F_auxiliary} follows.

   Note that the cokernel $\cok(c_n)$ of the $n$-th differential
   $c_d \colon C_d(X) \to C_{d-1}(X)$ of the cellular $\IZ G$-chain complex $C_*(X)$ is a
   finitely presented $FG$-module for every $n \in \IZ^{\ge 1}$. One easily checks
   \begin{multline*}
    \dim_{\cald_{FG}}(H_n(\cald_{FG} \otimes_{\IZ G} C_*(X)))
    \\
    = 
     \dim_{\cald_{FG}}(\cald_{FG} \otimes_{FG}\cok(c_n))
     + \dim_{\cald_{FG}}(\cald_{FG} \otimes_{FG}\cok(c_{n-1}))
     \\
     - \dim_{\IZ G}(C_{n-1}(X));
\end{multline*}
\begin{multline*}
  \dim_{\cald_{F[G/G_i]}}(H_n(\cald_{F[G/G_i]} \otimes_{\IZ G} C_*(X)))
  \\
  =
     \dim_{\cald_{F[G/G_i]}}(\cald_{F[G/G_i]}\otimes_{\IZ G}\cok(c_n))
     + \dim_{\cald_{F[G/G_i]}}(\cald_{F[G/G_i]}\otimes_{\IZ G}\cok(c_{n-1}))
     \\
     - \dim_{\IZ [G/G_i]}(\IZ[G/G_i ]\otimes_{\IZ G} C_{n-1}(X)),
\end{multline*}
and
\[
\dim_{\IZ G}(C_{n-1}(X)) =  \dim_{\IZ [G/G_i]}(\IZ[G/G_i] \otimes_{\IZ G} C_{n-1}(X)).
\]
Now the claim follows from~\eqref{the:Approximation_over_any_F:General_F_auxiliary}.
   \end{proof}

%%%%%%%%%%%%%%%%%%%%%%%%%%%%%%%%%%%%%%%%%%%%%%%%%%%%%%%%%%%%%%%%%%%%%
%%%%%%%%%%%%%%%%%%%%%%%%%%%%%% Section 4 %%%%%%%%%%%%%%%%%%%%%%%%%%%%%%%%
%%%%%%%%%%%%%%%%%%%%%%%%%%%%%%%%%%%%%%%%%%%%%%%%%%%%%%%%%%%%%%%%%%%%%

\typeout{---------- Section 4:  $L^2$-Betti numbers over over locally indicable  weak Lewin  groups  ---------------}

%-----------------------------------------------------------------------------

\section{$L^2$-Betti numbers in arbitrary characteristic over locally indicable  weak Lewin  groups and \RALI-groups}%
\label{sec:L2-Betti_numbers_in_arbitrary_characteristic_over_locally_indicable_weak_Lewin_groups_and_RALI-groups}

  %-----------------------------------------------------------------------------

\subsection{(Weak) Lewin groups}\label{subsec:(Weak)_Lewin_groups}

\begin{definition}[Lewin group]\label{def:Lewin_group}
  A group $G$ is called  a \emph{Lewin group} if for  an field  $F$ with $G$-action  we have
  for the crossed product ring $F \ast G$:
\begin{itemize}
\item There exists a Hughes free $FG$-division ring  $\cald_{F \ast G}$;
\item It is the universal division ring of fractions of $F \ast G$.
\end{itemize}
\end{definition}

Note that $\cald_{F\ast G}$ is unique up to unique isomorphism of $F\ast G$-division ring if it exists, and
the map $F \ast G \to \cald_{E \ast G}$ is injective and epic.

We will only need the following weaker version.

\begin{definition}[Weak Lewin group]\label{def:weak_Lewin_group}
A group $G$ is called  a \emph{weak Lewin group} if for  any field $F$ we get for the group ring $FG$:
\begin{itemize}
\item There exists a Hughes free $FG$-division ring  $\cald_{FG}$;
\item It is the universal division ring of fractions of $FG$.
\end{itemize}
\end{definition}

If $G$ is a locally indicable weak Lewin
group, then it is a weak Linnell group and a weak Hughes group, and the
Definition~\ref{def:L2_Betti_numbers_over_field} of the $n$th $L^2$-Betti number over the
field $F$ of an arbitrary $G$-$CW$-complex $X$ applies.

Jaikin-Zapirain states  the following conjecture in~\cite[Conjecture~1]{Jaikin-Zapirain(2021)}.

\begin{conjecture}[Locally indicable  groups are Lewin groups]%
  \label{con:Locally_indicibale_groups_are_Lewin_groups}
  Every locally indicable group is a Lewin group (and in particular a Hughes group).
\end{conjecture}

If the condition that $G$ is locally indicable is dropped in
Conjecture~\ref{con:Locally_indicibale_groups_are_Lewin_groups}, then there are
counterexamples for $F = \IQ$ and the group ring $FG$, see~\cite[Proposition~4.1] {Jaikin-Zapirain(2021)}.

\begin{remark}[The property weak Lewin is a residual property]%
\label{rem:weak_Lewin_is_a-residual_property}
  The notion of a weak Lewin group has the advantage that a residually weak Lewin group is
  again a weak Lewin group.  This follows from~\cite[Theorem~1.2]{Jaikin-Zapirain(2021)}
  by inspecting the proof of~\cite[Corollary~1.3]{Jaikin-Zapirain(2021)}, which is based
  on~\cite[Proposition~3.5]{Jaikin-Zapirain(2021)} in which one does only need using the
  notation there that each $G_i$ is a weak Lewin group and not that it is amenable and
  locally indicable.
\end{remark}

\begin{remark}[Amenable locally indicable groups]\label{rem:Amenable_locally_indicable_groups}
  Let $G$ be an amenable locally indicable group and let $F$ be any field.  Since $G$ is
  locally indicable, $FG$ has no non-trivial zero-divisors.  This follows from the proof
  of~\cite[Theorem~12]{Higman(1940)}. Since $G$ is amenable and has no non-trivial
  zero-divisors, $FG$ satisfies the Ore condition with respect to the multiplicative
  subset $S$ of non-zero elements, see~\cite[Example~8.16 on
  page~324]{Lueck(2002)}. Moreover, the universal ring of fractions  $\cald_{FG}$
  agrees with the Ore localization
  $S^{-1}FG$ for the multiplicative subset $S = FG \setminus \{0\}$ and hence $G$  is a weak Lewin group,
  weak Linnell group, and weak Hughes group,
  see~\cite[Corollary~3.4]{Jaikin-Zapirain(2021)}.

  This implies that for a
  $G$-$CW$-complex $X$ we get
\begin{eqnarray*}
  b^{(2)}_n(X;\cald_{FG})
  & := &
 \dim_{\cald_{FG}}\bigl(H_n(\cald_{FG} \otimes_{\IZ G} C_*(X))\bigr)
  \\
  &  = &
\dim_{S^{-1}FG}\bigl(H_n(S^{-1}FG \otimes_{\IZ G} C_*(X))\bigr)
  \\
  & = &
  \dim_{S^{-1}FG}\bigl(H_n(S^{-1}FG \otimes_{FG} F \otimes_{\IZ} C_*(X))\bigr)
  \\
  & = &
 \dim_{S^{-1}FG}\bigl(S^{-1}FG \otimes_{FG} (H_n(F \otimes_{\IZ} C_*(X)))\bigr),
\end{eqnarray*}
since $S^{-1}FG$ is flat as $FG$-module. In  particular  we get $b_n(X;\cald_{FG}) = 0$,
provided that $H_n(F \otimes_{\IZ} C_*(X))$ vanishes.
\end{remark}

 %-----------------------------------------------------------------------------

\subsection{\RALI-groups}\label{subsec:RALI-groups}

The favourite case is the following one.

\begin{definition}[\RALI-group]\label{def:RALI-group} We call a group $G$ residually (amenable
  and locally indicable) or briefly a \emph{\RALI-group} if it admits a \emph{normal
    \RALI-chain}, i.e., a  descending sequence of in $G$ normal subgroups
  \[
    G = G_0 \supseteq G_1 \supseteq G_2 \supseteq G_3 \supseteq \cdots
  \]
  such that $\bigcap_{i = 0} ^{\infty} G_i = \{1\}$ holds and each quotient group $G/G_i$
  is both amenable and locally indicable.
\end{definition}

Note that any \RALI-group is a locally indicable group.

\begin{theorem}\label{the:Residually_(locally_indicable_and_amenable)_groups}
  Let $G$ be a group and let $F$ be a field. Then:

  \begin{enumerate}
  \item\label{the:Residually_(locally_indicable_and_amenable)_groups:Levin}
    If $G$ is a \RALI-group, then $G$ is a weak Lewin group;

\item\label{the:Residually_(locally_indicable_and_amenable)_groups:weak_monotonicity}
  
  Let $p \colon G \to Q$ an epimorphism of groups, $F$ be a field, and $M$ be a finitely
  presented $FG$-module.  Let $p_*M$ be the $FQ$-module $FQ \otimes_{FG} M$ given by
  induction with $p$.

  \begin{enumerate}
  \item\label{the:Residually_(locally_indicable_and_amenable)_groups:weak_monotonicity:(1)}
   If  $G$ is a weak Lewin group and $Q$ is a weak Hughes  group, then 
    \[
    \dim_{\cald_{FG}}(\cald_{FG} \otimes_{FG} M) \le \dim_{\cald_{FQ}}(\cald_{FQ} \otimes_{FQ} p_* M);
  \]

  \item\label{the:Residually_(locally_indicable_and_amenable)_groups:weak_monotonicity:(2)}
    If the field $F$ satisfies  $\IQ \subseteq F \subseteq \IC$, $G$ is a weak Lewin group,
    $Q$ is torsionfree and satisfies the strong Atiyah Conjecture, e.g., $Q$ is locally indicable, then 
     \[
    \dim_{\caln(G)}(\caln(G) \otimes_{FG} M)\le \dim_{\caln(Q)}(\caln(Q) \otimes_{FG} M).
  \]
\end{enumerate}
\end{enumerate}
\end{theorem}
\begin{proof}~\ref{the:Residually_(locally_indicable_and_amenable)_groups:Levin}
  This follows from~\cite[Corollary~1.3]{Jaikin-Zapirain(2021)} or just from
  Remarks~\ref{rem:weak_Lewin_is_a-residual_property} and~\ref{rem:Amenable_locally_indicable_groups}.
    \\[1mm]~\ref{the:Residually_(locally_indicable_and_amenable)_groups:weak_monotonicity}.
   We  begin with~\ref{the:Residually_(locally_indicable_and_amenable)_groups:weak_monotonicity:(1)}
  We get a Sylvester module rank function $\rk_{FQ}$ on $FQ$ by
  sending a finitely presented $FQ$-module $N$ to
  $\dim_{\cald_{FQ}}(\cald_{FQ} \otimes_{FQ} N)$. We also have the Sylvester
  module rank function $\rk_{FG}$ on $FG$ sending a finitely presented $FG$-module $M$ to
  $\dim_{\cald_{FG}}(\cald_{FG} \otimes_{FG} M)$. Define a Sylvester module rank
  function $p^* \rk_{FQ}$ on $FG$ by sending a finitely presented $FG$-module $M$ to
  $\rk_{FQ}(p_*M)$. Since the $FG$-division ring $\cald_{FG}$ is the universal division
  ring of fractions of $FG$, we conclude  $p^* \rk_{FQ} \le \rk_{FG}$.

  Next we prove~\ref{the:Residually_(locally_indicable_and_amenable)_groups:weak_monotonicity:(2)}.
  Define a Sylvester module rank function $\rk_{FQ}$ on $FQ$ by
  sending a finitely presented $FQ$-module $N$ to
  $\dim_{\cald(FQ \subseteq \calu(G))}(\cald(FQ \subseteq \calu(Q)) \otimes_{FQ} N)$.  Recall that
  $\dim_{\cald(FQ \subseteq \calu(G))}(\cald(FQ \subseteq \calu(Q)) \otimes_{FQ} N)
  = \dim_{\caln(Q)}(\caln(Q) \otimes_{FQ} N)$
  holds by~\eqref{b_n:_upper_(2)((X;caln(G))_in_terms_of_cald(FG_subseteq_calu(G))}.
  We also have the Sylvester
  module rank function $\rk_{FG}$ on $FG$ sending a finitely presented $FG$-module $M$ to
  $\dim_{\cald_{FG}}(\cald_{FG} \otimes_{\cald_{FG}} M)$. Recall that 
  $\dim_{\cald_{FG}}(\cald_{FG} \otimes_{FG} M) = \dim_{\caln(G)}(\caln(G) \otimes_{FG} M)$
  holds by Lemma~\ref{lem:cald_(FG)_is_cald(FG_subseteq_calu(G))}~%
\ref{lem:cald_(FG)_is_cald(FG_subseteq_calu(G)):isomorphism}. 
  Define a Sylvester module rank
  function $p^* \rk_{FQ}$ on $FG$ by sending a finitely presented $FG$-module $M$ to
  $\rk_{FQ}(p_*M)$. Since the $FG$-division ring $\cald_{FG}$ is the universal division
  ring of fractions of $FG$, we conclude  $ \rk_{FG} \le p^* \rk_{FQ}$.
  
  This finishes the proof of Theorem~\ref{the:Residually_(locally_indicable_and_amenable)_groups}.
\end{proof}

%-----------------------------------------------------------------------------

\subsection{Monotonicity}\label{subsec:Monotonicity}

The main reason why we want to consider weak Lewin groups and \RALI-groups is the  next result which does
not hold in the previous setting, where we worked with  weak Hughes groups. This will be
relevant when we will deal with a question due to Gromov and Wise in
Section~\ref{sec:A_conjecture_due_to_Gromov_and_Wise}.

\begin{theorem}[Monotonicity]\label{the:Monotonicity}
  Let $p \colon G \to Q$ be an epimorphism of groups with $K$ as kernel.
  Let $X$ be a free connected $G$-$CW$-complex of finite type. Let $F$ be any field. 

  \begin{enumerate}
  \item\label{the:Monotonicity:arbitray_F}
    Suppose  that $G$ is a locally indicable weak Lewin group and $Q$ is a weak Hughes group.
    Then we get for   every $n \in \IZ^{\ge 0}$
    \[
    b^{(2)}_n(X;\cald_{FG})  \le  b^{(2)}_n(X/K;\cald_{FQ}).
  \]
  If we additionally assume that $G$ is non-trivial and $Q$ is trivial, we get
  \[
    b^{(2)}_1(X;\cald_{FG})  \le  b_1(X/G;F) -  1;
   \]

 \item\label{the:Monotonicity:cal(G)} Suppose that $G$ is a locally indicable weak Lewin group and that $Q$
   is torsionfree and satisfies the strong Atiyah Conjecture, e.g., $Q$ is locally
   indicable.  Then we get for every $n \in \IZ^{\ge 0}$
     \[
    b_n^{(2)}(X;\caln(G))  \le  b_n^{(2)}(X/K;\caln(Q)).
  \]
  If we additionally assume that $G$ is non-trivial and $Q$ is trivial, we get
  \[
    b_1^{(2)}(X;\caln(G))  \le  b_1(X;F) - 1.
   \]
\end{enumerate}
\end{theorem}
\begin{proof}\ref{the:Monotonicity:arbitray_F}
  Let $c_n \colon C_n(X;F) \to  C_{n-1}(X;F)$ be the  $n$-th differential of the cellular finitely generated free
  $FG$-chain complex of $C_*(X;F)$. We get from Additivity
  \begin{eqnarray}\label{b_n_upper_(2)_in_terms_cokernels}
    &
    \\
    b^{(2)}_n(X;\cald_{FG})
    & = &
    \dim_{\cald_{FG}}(\ker(\id_{\cald_{FG}} \otimes_{FG} \; c_{n}))  - \dim_{\cald_{FG}}(\im(\id_{\cald_{FG}} \otimes_{FG} \; c_{n+1}))
    \nonumber
    \\
    & = &
    \dim_{\cald_{FG}}(\cok(\id_{\cald_{FG}} \otimes_{FG} \; c_{n})) + \dim_{\cald_{FG}}(\cok(\id_{\cald_{FG}} \otimes_{FG}  \; c_{n+1}))
   \nonumber
    \\
     & & \quad  \quad -    \dim_{\cald_{FG}}(\cald_{FG} \otimes_{\IZ G} C_{n-1}(X))
    \nonumber
    \\
    & = &
    \dim_{\cald_{FG}}(\cald_{FG}\otimes_{FG}\cok(c_{n})) + \dim_{\cald_{FG}}(\cald_{FG}\otimes_{FG}\cok(c_{n+1})) 
    \nonumber
    \\
    & &  \quad  \quad  -    \dim_F(C_{n-1}(X/G;F)),
    \nonumber
  \end{eqnarray}
  and analogously using $FQ \otimes_{FG}\cok(c_{n}) \cong_{FQ} \cok(\id_{FQ} \otimes_{FG} \; c_n)$
     \begin{multline*}
     b^{(2)}_n(X/K;\cald_{FQ})
     = 
    \dim_{\cald_{FQ}}(\cald_{FQ} \otimes_{FG}\cok(c_{n}))
       \\
      + \dim_{\cald_{FQ}}(\cald_{FQ} \otimes_{FQ} \cok(c_{n+1})) - \dim_F(C_{n-1}(X/G;F)).
    \end{multline*}
    Moreover, we conclude from Theorem~\ref{the:zeroth_L2-Betti_number}~\ref{the:zeroth_L2-Betti_number:general_F}
    \[
    \begin{array}{lclcl}
      \dim_{\cald_{FG}}(\cald_{FG}\otimes_{FG}\cok(c_{1}))
      & = &
       b^{(2)}_0(X;\cald_{FG})
      & = &
       0;
      \\
      \dim_{\cald_{FQ}}(\cald_{FQ} \otimes_{FG} \cok(c_{1}))
      & = &
       b_0(X/G;F)
       &= &
       1.
    \end{array}
    \]
Now assertion~\ref{the:Monotonicity:arbitray_F}
follows from Theorem~\ref{the:Residually_(locally_indicable_and_amenable)_groups}~%
\ref{the:Residually_(locally_indicable_and_amenable)_groups:weak_monotonicity}.
\\[1mm]~\ref{the:Monotonicity:cal(G)} The proof is analogous to the one of
  assertion~\ref{the:Monotonicity:arbitray_F}.
\end{proof}

\begin{remark}\label{rem:counterexamples_to_Monotonicity}
Given an integer $l \ge1$ and a sequence $r_1$, $r_2$,
$\ldots$, $r_l$ of non-negative rational numbers, one can construct
a group $G$ such that $BG$ is of finite type and
\begin{eqnarray*}
b_n^{(2)}(G) := b_n^{(2)}(EG;\caln(G)) & = &
\left\{\begin{array}{lll} r_n & & \mbox{ for } 1 \le n \le l;\\
0 & & \mbox{ for } l+1 \le n;
\end{array}
\right.
\\
b_n(G) {:=} b_n(BG) & = & 0 \hspace{13mm} \mbox{ for } n \ge 1,
\end{eqnarray*}
holds, see~\cite[Example~1.38 on page~41]{Lueck(2002)}.

So Monotonicity as stated in Theorem~\ref{the:Monotonicity} for $p \colon G \to \{1\}$
does not hold in general. 

There also is a torsionfree  counterexample.
Let $H$ be the fundamental group of a hyperbolic homology $3$-sphere $M = BH$. Consider $G=H \ast H$.
Then $b_1(G)=0$ while $b_1^{(2)}(G) = 1$. Note  that both $H$ and $G$ are virtually special but not special.
Moreover, $G$ is torsionfree but is neither locally indicable nor a weak  Lewin group. In view of
Conjecture~\ref{con:Torsionfree_groups_are_Linnell_groups} it is possible  that  $G$ is a Linnell group.
\end{remark}

\begin{example}\label{exa:condition_RALI-needed_for_d_ge_3}
  Fix $n \in \IZ^{\ge 1}$
  We get from~\eqref{b_n_upper_(2)_in_terms_cokernels}  and Remark~\ref{rem:counterexamples_to_Monotonicity}
  a group $G$ together with a map
  $f \colon \IZ G^a \to \IZ G^b$ satisfying
  \[
    \dim_{\caln(G)}(\caln(G) \otimes_{\IZ G} \cok(f))  \ge n +  \dim_{\IQ}(\IQ \otimes_{\IZ G}  \cok(f)).
  \]
  Next we explain that we can additionally  arrange that $\id_{\IQ} \otimes_{\IZ G} f$ is injective.
  In the sequel we denote by $\overline{f} \colon \IZ^a \to \IZ^b$ the map given by
  $\id \otimes_{\IZ G} f \colon \IZ \otimes_{\IZ G} \IZ G^a \to  \IZ \otimes_{\IZ G} \IZ G^b$
  under the obvious identifications $\IZ \otimes_{\IZ G} \IZ G^a = \IZ^a$ and
  $\IZ \otimes_{\IZ G} \IZ G^b = \IZ^b$. The image $\im(\overline{f})$ is a $\IZ$-submodule of $\IZ^b$
  and hence a finitely generated free $\IZ$-module. Hence we can find a $\IZ$-map $s \colon \IZ^a \to \IZ^c$
  such that $s|_{\ker(\overline{f})} \colon \ker(\overline{f}) \to \IZ^c$ is bijective.
  Choose a $\IZ G$-map  $r \colon \IZ G^a \to \IZ G^c$ with $\overline{r} = s$. Consider the $\IZ G$-map
  \[
    g \colon  \IZ G^a \to \IZ G^b \oplus \IZ G^c, \quad x \mapsto (f(x),r(x)).
  \]
  Obviously  we get a surjection $\cok(g) \to \cok(f)$ which implies
  \[
    \dim_{\caln(G)}(\caln(G) \otimes_{\IZ G} \cok(g)) \ge  \dim_{\caln(G)}(\caln(G) \otimes_{\IZ G} \cok(f)).
  \]
  Now $\overline{g}$ can be written as the composite
  \[\IZ^a \xrightarrow{\pr \times s} \im(\overline{f}) \oplus \IZ^c 
\xrightarrow{\begin{pmatrix} i & 0 \\ 0 & \id_{\IZ^c} \end{pmatrix}}
\IZ^b \oplus \IZ^c
\]
where $\pr \colon \IZ^a \to \im(f)$ is the epimorphism induced by $\overline{f}$ and
$i \colon \im(f) \to \IZ^b$ is the inclusion.  Since $\pr \times s$ is bijective, we conclude 
that $\overline{g}$ is injective and $\cok(\overline{g}) \cong \cok(\overline{f})$ holds.
Hence $\id_{\IQ} \otimes_{\IZ} \overline{g}$ is injective  and we have
\begin{multline}
\dim_{\caln(G)}(\caln(G) \otimes_{\IZ G} \cok(g)) \ge  \dim_{\caln(G)}(\caln(G) \otimes_{\IZ G} \cok(f))
\\ 
\ge  n + \dim_{\IQ}(\IQ \otimes_{\IZ G} \cok(f)) = n + \dim_{\IQ}(\IQ \otimes_{\IZ G} \cok(g)).
\end{multline}

  Fix $d \ge 3$.  By attaching $(b+c)$ trivial $(d-1)$-cells and $a$ $d$-cells to
  $Y$ in an appropriate way to the $2$-skeleton $EG_2$ of $EG$, we get a $d$-dimensional
  simply connected finite $G$-$CW$-complex $\widetilde{X}$ such that the cellular
  $\IZ G$-chain complex $C_*(\widetilde{X})$ is the direct sum of the $2$-dimensional
  cellular $\IZ G$-chain complex $C_*(EG_2)$ and   the $\IZ G$-chain complex concentrated
  in dimension $d$ and $(d-1)$ having $g$ as $d$-th differential. We conclude
  from~\eqref{b_n_upper_(2)_in_terms_cokernels} for the finite connected $d$-dimensional
  $CW$-complex $X = \widetilde{X}/G$ that $G = \pi_1(X)$,
  $b_d^{(2)}\bigl(\widetilde{X};\caln(\pi_1(X))\bigr) \ge  n$ and $b_d(X;\IC) = 0$ hold.

  Let $Y$ be  the pushout $X \cup_{X_1} \cone(X_1)$. Since $\cone(X_1)$ is a finite $2$-dimensional
  contractible $CW$-complex, $Y$ is a simply connected finite $CW$-complex of dimension
  $d$ containing $X$ as a subcomplex such that we have
  $b_d^{(2)}\bigl(\widetilde{X};\caln(\pi_1(X))\bigr) \ge n $ and
  $b_d^{(2)}\bigl(\widetilde{Y};\caln(\pi_1(X))\bigr) = b_d(Y;\IQ) = b_d(X;\IQ) = 0$.

  This example shows that for $d \ge 3$ some conditions about the fundamental groups are
  necessary in Theorem~\ref{the:vanishing_of_the_top_L2-Betti_numbers_and_RALI-towers}.
\end{example}

%-----------------------------------------------------------------------------

\subsection{Approximation by subgroups of finite index and identification with  adhoc definitions}%
\label{subsec:Approximation_by_subgroups_of_finite_index_and_Idenitfication_with_adhoc_definitions}

 Let $G$ be  a residually finite group. Let 
  $G = H_0 \supseteq H_1 \supseteq H_2  \supseteq \cdots$ be any  sequence of in $G$ normal subgroups
  such that $\bigcap_{i = 0}^{\infty} H_i = \{1\}$ holds and each quotient $G/H_i$ is finite.

  Then we get for every free $G$-$CW$-complex $X$ of finite type and any field $F$ of characteristic zero,
  see~\cite[Theorem~0.1]{Lueck(1994c)}
  \begin{equation}
    b_n^{(2)}(X;\caln(G)) = \lim_{i \to \infty} \frac{b_n(X/H_i;F)}{[G:H_i]}
   \label{Approximation_finite_index_char_zero}
 \end{equation}   
 where $b_n(Z;F)$ is for a $CW$-complex $Z$ of finite type the classical $n$th Betti
 number defined by $\dim_F(H_n(Z;F))$. Later in this subsection   we explain what happens for arbitrary
 fields.

 There are the following ad hoc definitions of $L^2$-Betti numbers over a field $F$ for a
 connected $CW$-complex $Y$ of finite type which are motivated
 by~\eqref{Approximation_finite_index_char_zero}:

\begin{enumerate}

\item Define 
\[
\beta^{\inf}_{n}(Y;F):=\inf_{Y'\rightarrow Y}{b_n(Y';F)\over |Y'\rightarrow Y|}
\]
where the infimum is taken over all regular coverings whose order $|Y'\rightarrow Y|$ is finite.
This agrees with
\[
 \quad \inf\left\{\left. \frac{b_n(\widetilde{Y}/H;F)}{[G:H]} \; \right| \;
H \subseteq \pi_1(Y)  \;\textup{normal subgroup of finite index}\right\};
\]
\begin{remark}
  In~\cite{Avramidi-Okun-Schreve(2024)}, $\beta^{\inf}$ is defined using all finite
  covers, not just the regular ones, and is shown to agree with the skew field Betti
  number for $\RTFN$ groups. We show in
  Theorem~\ref{the:Identification_with_adhoc_definition} that the present definition also
  agrees with the skew field Betti number.
\end{remark}

\item One may also  consider the version where the orders are $p$-powers, Namely, define 
\[
\beta^{\inf,p}_{n}(Y;F):=\inf_{Y'\rightarrow Y}{b_n(Y';F)\over |Y'\rightarrow Y|},
\]
where the infimum is taken over all regular coverings whose order is a finite $p$-power. This agrees with
\begin{multline*}
 \quad \quad  \quad \inf\left\{\left. \frac{b_n(\widetilde{Y}/H;F)}{[G:H]} \; \right| \;
   H \subseteq \pi_1(Y)  \;\textup{normal subgroup with} \; [\pi_1(Y) :H] = p^m \right.\\
   \left.\;\text{for some} \; m \in \IZ^{\ge 0}\right\}.
\end{multline*}

\item Note it is not at all clear whether these notions are multiplicative under finite
  coverings. Therefore there also are the following  definitions in the literature, where
  multiplicativity is forced by the definition

  \[
    \underline{\beta}_{n}(Y;F):=\sup _{Y'\rightarrow Y} \frac{\beta^{\inf}_{n}(Y';F)}{|Y'\rightarrow Y|}
    = \sup _{Y'\rightarrow Y} \left(\inf_{Y''\rightarrow Y'}\frac{b_n(Y'';F)}{|Y''\rightarrow Y|}\right)
\]
where $Y'\rightarrow Y$ and $Y''\rightarrow Y'$ run over all coverings whose order is finite.

One may also consider the $p$-power version

\[
    \underline{\beta}^p_{n}(Y;F) :=\sup _{Y'\rightarrow Y} \frac{\beta^{\inf,p}_{n}(Y';F)}{|Y'\rightarrow Y|}
    = \sup _{Y'\rightarrow Y} \left(\inf_{Y''\rightarrow Y'}\frac{b_n(Y'';F)}{|Y''\rightarrow Y|}\right)
\]
where $Y'\rightarrow Y$ and $Y''\rightarrow Y'$ run over all coverings whose order is a finite $p$-power;

\item One may also interchange the infimum and the supremum and define

  \[
    \overline{\beta}_{n}(Y;F) :=
   \inf _{Y'\rightarrow Y} \left(\sup_{Y''\rightarrow Y'}\frac{b_n(Y'';F)}{|Y''\rightarrow Y|}\right)
\]
where $Y'\rightarrow Y$ and $Y''\rightarrow Y'$ run over all coverings whose order is
finite, Again one may consider the $p$-power version
\[
  \overline{\beta}^p_{n}(Y;F) :=
    \inf _{Y'\rightarrow Y} \left(\sup_{Y''\rightarrow Y'}\frac{b_n(Y'';F)}{|Y''\rightarrow Y|}\right)
\]
where $Y'\rightarrow Y$ and $Y''\rightarrow Y'$ run over all coverings whose order is a
finite $p$-power;

\end{enumerate}

Obviously we have the following inequalities 
\[
\beta^{\inf}_{n}(Y;F) \le   \underline{\beta}_{n}(Y;F) \le \overline{\beta}_{n}(Y;F);
\]
\[
\beta^{\inf}_{n}(Y;F) \le \beta^{\inf,p}_{n}(Y;F),
\]
and
\[
\beta^{\inf,p}_{n}(Y;F) \le   \underline{\beta}^p_{n}(Y;F) \le\overline{\beta}^p_{n}(Y;F).
\]
If $F$ has characteristic zero, some of  these notions can be expressed in terms of
$L^2$-Betti numbers.

\begin{theorem}[Identification with adhoc definition in characteristic zero]\label{the:ad_hoc_and_F_char_zero}
  Let $F$ be a field of characteristic zero and $Y$ be a $CW$-complex of finite type with
  fundamental group $\pi$.

  \begin{enumerate}
  \item\label{lem:ad_hoc_and_F_char_zero:res_fin} Let $\widehat{\pi}$ be the maximal
    residually finite quotient of $\pi$, see Subsection~\ref{subsec:Maximal_residually_calc_coverings},
    and $\widehat{Y}$ be the corresponding
    $\widehat{\pi}$-covering $\widehat{Y} \to Y$. Then we get
    \[
     b_n^{(2)}(\widehat{Y};\caln(\widehat{\pi})) = \underline{\beta}_{n}(Y;F) = \overline{\beta}_{n}(Y;F).
       \]
       If $\pi$ is residually finite, then we get 
        \[
     b_n^{(2)}(\widetilde{Y};\caln(\pi)) = \underline{\beta}_{n}(Y;F) = \overline{\beta}_{n}(Y;F);
   \]

     \item\label{lem:ad_hoc_and_F_char_zero:res_p_fin} Let $p$ be a prime. Let
       $\widehat{\pi}^p$ be the maximal residually $p$-finite quotient of $\pi$,
       see Subsection~\ref{subsec:Maximal_residually_calc_coverings}, and
       $\widehat{Y}$ be the corresponding $\widehat{\pi}^p$-covering
       $\widehat{Y}^p \to Y$. Then we get
    \[
      b_n^{(2)}(\widehat{Y}^p;\caln(\widehat{\pi}^p)) 
      = \underline{\beta}^p_{n}(Y;F) = \overline{\beta}^p_{n}(Y;F).
       \]
       If $\pi$ is residually $p$-finite, then
        \[
     b_n^{(2)}(\widetilde{Y};\caln(\pi)) = \underline{\beta}^p_{n}(Y;F) = \overline{\beta}^p_{n}(Y;F)
       \]
       holds.
  \end{enumerate}
\end{theorem}
\begin{proof}
  This follows from~\cite[Corollary~2.6]{Avramidi-Okun-Schreve(2024)} 
  and~\cite[Theorem~2.3~(1)]{Lueck(1994c)}.
 \end{proof}

 Next we investigate what happens in prime characteristic. For this purpose we need a
 stronger condition than being a residually finite or residually $p$-finite \RALI-group.

 \begin{definition}[\RALIRF-group]\label{def:strongly_RALIRF} We call a
   group $G$ a \emph{\RALIRF-group} if it admits a \emph{normal \RALIRF-chain}, i.e., a
   descending sequence of in $G$ normal subgroups
  \[
    G = G_0 \supseteq G_1 \supseteq G_2 \supseteq G_3 \supseteq \cdots
  \]
  such that $\bigcap_{i = 0} ^{\infty} G_i = \{1\}$ holds and each quotient group $G/G_i$
  is amenable, locally indicable, and residually finite.

  Let $p$ be a prime. We call a group $G$ a \emph{\RALIRF$_p$-group} if it admits a
  \emph{normal \RALIRF$_p$-chain}, i.e., a, descending sequence of in $G$ normal subgroups
  \[
    G = G_0 \supseteq G_1 \supseteq G_2 \supseteq G_3 \supseteq \cdots
  \]
  such that $\bigcap_{i = 0} ^{\infty} G_i = \{1\}$ holds and each quotient group $G/G_i$
  is amenable, locally indicable, and residually $p$-finite.

\end{definition}

Note that a \RALIRF-group or \RALIRF$_p$-group respectively is a $\RALI$-group which is 
residually finite or $p$-residually finite respectively. We do not now whether every
\RALI-group, which is residually finite or $p$-residually finite respectively, is a 
\RALIRF-group or a \RALIRF$_p$-group respectively.

We later will need the following lemma.

\begin{lemma}\label{lem:monotonicity_for_p-groups}
  Let $p$ be a prime and $G$ be a group. Let $X$ be a free $G$-$CW$-complex of finite type.
  Consider two normal subgroups
      $L_0 \subseteq L_1 \subseteq G$ of $G$ such that $[G:L_0]$ is finite and a
      $p$-power. Then we get
      \[
        \frac{b_n(X/L_0;\IF_p)}{[G:L_0]} \le \frac{b_n(X/L_1;\IF_p)}{[G:L_1]}.
      \]
    \end{lemma}
    \begin{proof}
      This follows from~\cite[Lemma~4.1]{Bergeron-Linnell-Lueck-Sauer(2014)} by inspecting
      the proof of~\cite[Theorem~1.6]{Bergeron-Linnell-Lueck-Sauer(2014)}.
    \end{proof}

\begin{theorem}[Identification with adhoc definition in prime characteristic]\label{the:Identification_with_adhoc_definition}
  Let $G$ be a group.  Let $F$ be any field.
  Consider any free $G$-$CW$-complex $X$ of finite type. 
  Then:

  \begin{enumerate}
  \item\label{the:Identification_with_adhoc_definition:le_infimum}
    Suppose that $G$ is a residually finite $\RALI$-group. If $X$ is simply connected,
    we get
    \[
      b^{(2)}_n(X;\cald_{FG}) \le  \beta^{\inf}_{n}(X/G;F);
    \]

    \item\label{the:Identification_with_adhoc_definition:is_infimum}
      Assume that $G$ is a countable \RALIRF-group. Then:

      \begin{enumerate}
       
      \item\label{the:Identification_with_adhoc_definition:is_infimum:sequence}
        
      There exists a descending sequence of in $G$ normal subgroups
  \[
    G = L_0 \supseteq L _1 \supseteq L_2 \supseteq L_3 \supseteq \cdots
  \]
  such that each index $[G:L_i]$ is finite, $\bigcap_{i = 0}^{\infty} L_i = \{1\}$ holds,
  and we have
  \[
    b^{(2)}_n(X;\cald_{FG}) =   \lim_{i \to \infty} \frac{b_n(X/L_i;F)}{[G:L_i]};
  \]

\item\label{the:Identification_with_adhoc_definition:is_infimum:equality}
  If $X$ is simply connected, we get
    \[
      \quad \quad b^{(2)}_n(X;\cald_{FG}) =   \beta^{\inf}_{n}(X/G;F)   = \underline{\beta}_{n}(X/G;F);
  \]
\end{enumerate}

\item\label{the:Identification_with_adhoc_definition:RALIRF_p} 
  Let $p$ be a prime. Assume $G$ is a countable \RALIRF$_p$-group. Then:

  \begin{enumerate}
  \item\label{the:Identification_with_adhoc_definition:RALIRF_p:sequences}
 There exists a sequence of in $G$ normal subgroups
  \[
    G = L_0 \supseteq L _1 \supseteq L_2 \supseteq L_3 \supseteq \cdots
  \]
  such that each index $[G:L_i]$ is a finite $p$-power,
  $\bigcap_{i = 0}^{\infty} L_i = \{1\}$ holds, and we have
  \[
   b^{(2)}_n(X;\cald_{FG}) =   \lim_{i \to \infty} \frac{b_n(X/L_i;F)}{[G:L_i]};
  \]

  \item\label{the:Identification_with_adhoc_definition:RALIRF_p:equalities_all_F} 
    If $X$ is simply connected, we have the equalities
    \[
      \quad \quad   b_n^{(2)}(X;\cald_{FG}) = \beta^{\inf}_{n}(X/G;F)   =\beta^{\inf,p}_{n}(X/G;F)
       = \underline{\beta}_{n}(X/G;F)   = \underline{\beta}^p_{n}(X/G;F) 
     \]

   \item\label{the:Identification_with_adhoc_definition:RALIRF_p:equalities_F_char_p}
     If $X$ is simply connected and $F$ has characteristic $p$, then we have the equalities
    \begin{multline*}
      \quad \quad   b^{(2)}_n(X;\cald_{FG}) = \beta^{\inf}_{n}(X/G;F)   =\beta^{\inf,p}_{n}(X/G;F)
      \\
      = \underline{\beta}_{n}(X/G;F)   = \underline{\beta}^p_{n}(X/G;F) =  \overline{\beta}^p_{n}(X/G;F);
    \end{multline*}
  \end{enumerate}

\item\label{the:Identification_with_adhoc_definition:locally_indicable_amenable}
  Let $G$ be an amenable locally indicable   group which is residually finite.
  Then for every  descending sequence of in $G$ normal subgroups
  \[
    G = L_0 \supseteq L _1 \supseteq L_2 \supseteq L_3 \supseteq \cdots
  \]
  such that each index $[G:L_i]$ is finite and $\bigcap_{i = 0}^{\infty} G_i = \{1\}$
  holds and every field $F$ we get
  \[
    b^{(2)}_n(X;\cald_{FG}) =   \lim_{i \to \infty} \frac{b_n(X/L_i;F)}{[G:L_i]}.
  \]
    \end{enumerate}
\end{theorem}
\begin{proof}
  Recall that a \RALI-group $G$ is a locally indicable Lewin  group and in particular
  a weak Linnell group and  weak Hughes group,
  see Theorem~\ref{the:Residually_(locally_indicable_and_amenable)_groups}~%
\ref{the:Residually_(locally_indicable_and_amenable)_groups:Levin}.
  \\[1mm]~\ref{the:Identification_with_adhoc_definition:le_infimum}
 Let $H \subseteq G$ be a normal subgroup with $[G:H] < \infty$. We get from
 Theorem~\ref{the:restriction}~\ref{the:restriction:general_F}  applied to $H \subseteq G$ and from 
 Theorem~\ref{the:Monotonicity}~\ref{the:Monotonicity:arbitray_F} applied to $H \to \{1\}$
 and $\res_G^{H} X$ 
 \begin{eqnarray*}
   \frac{b^{(2)}_n(X;\cald_{FG})}{[G:H]}
   & = &
   b^{(2)}_n(\res_{G}^{H} X;\cald_{FH});
   \\
   b^{(2)}_n(\res_{G}^{H}X;\cald_{FH})
   & \le &
  b_n(X/H;F).
 \end{eqnarray*}
 This implies  $b^{(2)}_n(X;\cald_{FG}) \le  \frac{b_n(X/H;F)}{[G:H]}$. Hence we get
 \begin{multline*}
    b^{(2)}_n(X;\cald_{FG}) \le  \inf\left\{\left. \frac{b_n(X/H;F)}{[G:H]} \; \right| \;
      H \subseteq G \;\textup{normal subgroup of finite index}\right\}
    \\
    = \beta^{\inf}_{n}(X/G;F).
  \end{multline*}%
~\ref{the:Identification_with_adhoc_definition:is_infimum} We begin with the proof of
  assertion~\ref{the:Identification_with_adhoc_definition:is_infimum:sequence}.  Choose an
  enumeration $G = \{g_1, g_2, g_3, \ldots\}$ of the elements of the countable
  $\RALIRF$-group $G$ and a sequence
  $G = G_0 \supseteq G_1 \supseteq G_2 \supseteq \cdots$ of in $G$ normal subgroups such
  that $\bigcap_{i = 0}^{\infty} G_i = \{1\}$ holds and each quotient $G/G_i$ is
  residually finite, amenable, and locally indicable. By passing to a subsequence of the
  sequence $(G_i)_{i \ge 0}$ we can arrange that $g_i \not \in G_i$ holds for
  $i = 1,2, \ldots$. We conclude from
  Theorem~\ref{the:Approximation_over_any_F}~\ref{the:Approximation_over_any_F:General_F}
 \[
   b^{(2)}_n(X;\cald_{FG}) = \lim_{i \to \infty} b_n(X/G_i;\cald_{F[G/G_i]}).
 \]
 By passing to a subsequence of the sequence $(G_i)_{i \ge 0}$, we can arrange that for
 $i \in \IZ^{\ge 1}$ we get
 \begin{equation}
 \left|b^{(2)}_n(X;\cald_{FG})  - b^{(2)}_n(X/G_i;\cald_{F[G/G_i]})\right| \le \frac{1}{2i}.
\label{the:Identification_with_adhoc_definition:(1)}                                               
\end{equation}
Next we construct a descending  sequence of normal subgroups
$G = L_0 \supseteq L_1 \supseteq L_2 \supseteq \cdots $
 of $G$ such that $G_i \subseteq L_i$, $g_i \notin L_i$,  and
 \begin{equation}
\left|b^{(2)}_n(X/G_i;\cald_{F[G/G_i]})  - \frac{b_n(X/L_i;F)}{[Q:L_i]}\right| \le  \frac{1}{2i}
  \label{the:Identification_with_adhoc_definition:(2)}
\end{equation}
 hold for $i \in \IZ^{\ge 1}$. The induction beginning is $L_0 = G_0$.
The induction step from $(i-1)$ to $i \ge 1$ is done as follows.

Put $Q_i = G/G_i$. Since $Q_i$ is residually finite, we can find a sequence
$Q_i = K_0 \supseteq K_1 \supseteq K_2 \supseteq$ of 
normal subgroups of $Q_i$ of finite index satisfying $\bigcap_{j \ge 0} K_j = \{1\}$. By
intersecting each $K_i$ with the image of $L_{i-1}$ under the projection
$\pr_i \colon G \to Q_i = G/G_i$, we can arrange $\pr_i^{-1}(K_j) \subseteq L_{i-1}$ for
$j \in \IZ^{\ge 1}$.   We conclude from Remark~\ref{rem:Amenable_locally_indicable_groups}
and~\cite[Theorem~0.2~(iii)]{Linnell-Lueck-Sauer(2011)} for
the free $Q_i$-$CW$-complex $Y_i = X/G_i$ of finite type
 \[
 b^{(2)}_n(Y_i;\cald_{FQ_i}) = \lim_{j \to \infty} \frac{b_n(Y_i/K_j;F)}{[Q:K_j]}.
\]
The image of $g_i$ under the projection $\pr_i \colon G \to Q_i = G/G_i$ is different
from the unit, since $g_i \not\in G_i$ holds.  Hence we can choose $j \in \IZ^{\ge 1}$
such that
\begin{equation}
\left|b^{(2)}_n(Y_i;\cald_{FQ_i})  - \frac{b_n(Y_i/K_j;F)}{[Q_i:K_j]}\right| \le  \frac{1}{2i}
  \label{the:Identification_with_adhoc_definition:(3)}
\end{equation}
and $\pr_i(g_i) \notin K_j$ hold.  Now put $L_i = \pr_i^{-1}(K_j)$. Then $G_i \subseteq L_i \subseteq L_{i-1}$ holds,
$L_i$ is a normal subgroup of $G$ with finite index
$[G:L_i] = [Q_i : K_j]$, we have $g_i \notin L_i$, and  $Y_i/K_j = X/L_i$ is valid.
Moreover,~\eqref{the:Identification_with_adhoc_definition:(3)}
  implies~\eqref{the:Identification_with_adhoc_definition:(2)}.

  Now we conclude from~\eqref{the:Identification_with_adhoc_definition:(1)}
  and~\eqref{the:Identification_with_adhoc_definition:(2)}
  that we get for $i \in \IZ^{\ge 1}$
  \begin{equation}
 \left|b^{(2)}_n(X;\cald_{FG})  - \frac{b_n(X/L_i;F)}{[G:L_i]}\right| \le \frac{1}{i}.
\label{the:Identification_with_adhoc_definition:(4)}                                               
\end{equation}
Hence we get a sequence of in $G$ normal subgroups
$G = L_0 \supseteq L_1 \supseteq L_2 \supseteq \cdots$ such that $[G: L_i]$ is finite,  $\bigcap_{i = 0}^{\infty} L_i = \{1\}$ holds,
and we have
\[
b^{(2)}_n(X;\cald_{FG}) = 
\lim_{i \to \infty} \frac{b_n(X/L_i;F)}{[G:L_i]}.
\]
This finishes the proof of assertion~assertion~\ref{the:Identification_with_adhoc_definition:is_infimum:sequence}.

Next we proof assertion~\ref{the:Identification_with_adhoc_definition:is_infimum:equality}.
We conclude from assertion~\ref{the:Identification_with_adhoc_definition:is_infimum:sequence}
\begin{eqnarray*}
  b^{(2)}_n(X;\cald_{FG})
  & = &
\lim_{i \to \infty} \frac{b_n(X/L_i;F)}{[G:L_i]}
\\
& \ge &
\inf\left\{\left. \frac{b_n(X/H;F)}{[G:H]} \; \right| \;
        H \subseteq G \;\textup{normal subgroup of finite index}\right\}
\\
& = &
      \beta^{\inf}_{n}(X/G;F).
\end{eqnarray*}
This together with assertion~\ref{the:Identification_with_adhoc_definition:le_infimum}
implies
\[
  b^{(2)}_n(X;\cald_{FG}) = \beta^{\inf}_{n}(X/G;F).
\]
Since $b^{(2)}_n(X;\cald_{FG})$ is multiplicative under the passage to subgroups of finite index
 by Theorem~\ref{the:restriction}~\ref{the:restriction:special_F}, we conclude
   \[
    \beta^{\inf}_{n}(X;F)  = \underline{\beta}_{n}(X/G;F).
   \]
This finishes the proof of assertion~\ref{the:Identification_with_adhoc_definition:is_infimum}.
\\[1mm]~\ref{the:Identification_with_adhoc_definition:RALIRF_p}
    In the proof of~\ref{the:Identification_with_adhoc_definition:is_infimum:sequence}
    we can arrange  that each index $[G:G_i]$ and $[G_i :K_j]$ is a finite $p$-power. Thus we obtain by the same argument
     a sequence of in $G$ normal subgroups
  \[
    G = L_0 \supseteq L _1 \supseteq L_2 \supseteq L_3 \supseteq \cdots
  \]
  such that each index $[G:L_i]$ is finite $p$-power, $\bigcap_{i= 1}^{\infty} L_i = \{1\}$ holds, and we have
  \[
    b^{(2)}_n(X;\cald_{FG}) =   \lim_{i \to \infty} \frac{b_n(X/L_i;F)}{[G:L_i]}.
  \]

  This finishes the proof of assertion~\ref{the:Identification_with_adhoc_definition:RALIRF_p:sequences}.
  
  Next we prove assertion~\ref{the:Identification_with_adhoc_definition:RALIRF_p:equalities_all_F} 
   Since $G$ is residually $p$-finite, we get from assertions~\ref{the:Identification_with_adhoc_definition:is_infimum:equality}
   and~\ref{the:Identification_with_adhoc_definition:RALIRF_p:sequences}
   \[
     \beta^{\inf,p}_{n}(X/G;F)  \le b^{(2)}_n(X;\cald_{FG}) = \beta^{\inf}_{n}(X/G;F)   \le \beta^{\inf,p}_{n}(X/G;F).
  \]
  This implies
  \[
    b^{(2)}_n(X;\cald_{FG}) = \beta^{\inf}_{n}(X;F)   =\beta^{\inf,p}_{n}(X;F).
   \]
   Since $b^{(2)}_n(X;\cald_{FG})$ is multiplicative under the passage to subgroups of finite index
 by Theorem~\ref{the:restriction}~\ref{the:restriction:special_F}, we conclude
   \[
     b^{(2)}_n(X;\cald_{FG}) = \beta^{\inf}_{n}(X/G;F)   =\beta^{\inf,p}_{n}(X/G;F)
     = \underline{\beta}_{n}(X/G;F)   = \underline{\beta}^p_{n}(X/G;F).
   \]

   Finally we prove assertion~\ref{the:Identification_with_adhoc_definition:RALIRF_p:equalities_F_char_p}.
   Lemma~\ref{lem:monotonicity_for_p-groups} implies
   \[\overline{\beta}^p_{n}(X/G;F)  =\beta^{\inf,p}_{n}(X/G;F).
   \]
   Now the claim follows assertion~\ref{the:Identification_with_adhoc_definition:RALIRF_p:equalities_all_F}.
 
     This finishes the proof of assertion~\ref{the:Identification_with_adhoc_definition:RALIRF_p}.
   \\[1mm]~\ref{the:Identification_with_adhoc_definition:locally_indicable_amenable}
     This follows from Remark~\ref{rem:Amenable_locally_indicable_groups}
and~\cite[Theorem~0.2~(iii)]{Linnell-Lueck-Sauer(2011)}.
   This finishes the proof of Theorem~\ref{the:Identification_with_adhoc_definition}.
 \end{proof}

 Let $X$ be a simply connected free $G$-$CW$-complex.
 We have $b^{(2)}_n(X;\cald_{FG}) \le \overline{\beta}_{n}(X/G;F)$ if $G$ is a
 residually finite $\RALI$-group. We do not know whether
 $b^{(2)}_n(X;\cald_{FG})$ and $\overline{\beta}_{n}(X/G;F) $
 agree if $G$ is a  \RALIRF-group. If $G$ is a \RALIRF$_p$-group and $F$ has characteristic $p$,
 we have $b^{(2)}_n(X;\cald_{FG}) = \overline{\beta}^p_{n}(X/G;F)$, but we do not know
 whether $\overline{\beta}^p_{n}(X/G;F)$ and $\overline{\beta}_{n}(X/G;F) $ agree.

 We do not know whether assertion~\ref{the:Identification_with_adhoc_definition:locally_indicable_amenable}
 of Theorem~\ref{the:Identification_with_adhoc_definition} extends to residually finite $\RALI$-groups.

\begin{remark}
  For residually torsion free nilpotent (RTFN) groups, Theorem 4.17(ii)(b)
  is~\cite[Corollary~4.1]{Avramidi-Okun-Schreve(2024)} and the equality
  $b_n^{(2)}(X;\mathcal D_{\mathbb F_pG})=\beta^{\inf,p}_n(X/G;\mathbb F_p)$
  is~\cite[Theorem~2.7]{Avramidi-Okun-Schreve(2024subdivision)}. The
  inequality $b^{(2)}_n(X;\mathcal D_{FG})\leq\limsup_ib_n(X/G_i;\mathbb F)/|G/G_i|$ for  a \RALI-group $G$ and 
  any residual chain of finite index normal subgroups $\{G_i\}$ is~\cite[Theorem~D]{Fisher-Hughes-Leary(2024)}.
\end{remark}

\begin{remark}\label{rem:Advantage_of_the_new_definition}
  The following observations shall illustrate why it is much easier to work with the
  definition of $b^{(2)}_n(X;\cald_{FG}) $ to be
  $\dim_{\cald_{FG}}\bigl(H_n(\cald_{FG} \otimes_{\IZ G} C_*(X))\bigr)$ instead of the adhoc
  definition to be
\[
 \inf\left\{\left. \frac{b_n(X/H;F)}{[G:H]} \; \right| \;
H \subseteq G \;\textup{normal subgroup of finite index}\right\}.
\]

An obvious example is Theorem~\ref{the:Euler-Poincare_formula}~\ref{the:Euler-Poincare_formula:char_p}
which is not at all clear for the ad hoc versions.

Another instance are Mayer-Vietoris arguments.
Let $G$ be a $\RALI$-group. Consider a $G$-$CW$-complex $X$ which can be written as the
union of two $G$-$CW$-complexes $X_1$ and $X_2$ whose intersection $X_1 \cap X_2$ is
denoted by $X_0$. Suppose that $b^{(2)}_n(X_i;\cald_{FG})$ vanishes for every $i \in \{0,1,2\}$
and every $n \in \IZ^{\ge 0}$.  Then $b^{(2)}_n(X;\cald_{FG})$ vanishes every
$n \in \IZ^{\ge 0}$ by the obvious Mayer-Vietoris sequence
\begin{multline*}
  \cdots \to H_n(\cald_{FG} \otimes_{\IZ G} C_*(X_0))
  \to H_n(\cald_{FG} \otimes_{\IZ G} C_*(X_1)) \oplus H_n(\cald_{FG} \otimes_{\IZ G} C_*(X_2))
  \\
  \to H_n(\cald_{FG} \otimes_{\IZ G} C_*(X))
  \to H_{n-1}(\cald_{FG} \otimes_{\IZ G} C_*(X_0))
  \\
  \to H_{n-1}(\cald_{FG} \otimes_{\IZ G} C_*(X_1)) \oplus H_{n-1}(\cald_{FG} \otimes_{\IZ G} C_*(X_2))
  \to \cdots.
\end{multline*}
The favourite situation, where this will be applied is the following.  Consider a
$CW$-complex $Y$ with sub $CW$-complexes $Y_1, Y_2$, and $Y_0$ satisfying
$Y = Y_1 \cup Y_2$ and $Y_0 = Y_1 \cap Y_2$. Suppose (for simplicity) that $Y_0$, $Y_1$,
and $Y_2$ and hence also $Y$ are connected. Denote by $\widetilde{Y_i}$ and $\widetilde{Y}$
the universal coverings of $Y$ and $Y_i$.  Assume that the inclusions $Y_i \to Y$ induce
injections on the fundamental groups and that $\pi_1(Y)$ is a weak Hughes group. Then $\pi_1(Y_i)$ is
a subgroup of $\pi_1(Y)$ and hence itself a weak Hughes group for $i = 0,1,2$
by Lemma~\ref{lem:subgroups_of_(weak)_Hughes-groups_are_(weak)-Hughes_groups}. Furthermore
suppose that $b^{(2)}_n(\widetilde{Y_i};\cald_{F[\pi_1(Y_i)]})$ vanishes for every
$i \in \{0,1,2\}$ and every $n \in \IZ^{\ge 0}$. Then we conclude from
Theorem~\ref{the:induction}~\ref{the:induction:general_F} and the argument above
applied in the case $G = \pi_1(Y)$, $X = \widetilde{Y}$, and
$X_i = \widetilde{Y}|_{Y_i} = \pi_1(Y) \times_{\pi_1(Y_i)} \widetilde{Y_i}$ that
$b^{(2)}_n(\widetilde{Y},\cald_{F[\pi_1(X)]})$ vanishes every $n \in \IZ^{\ge 0}$.
\end{remark}

    \begin{remark}\label{remark:caveat} One may think that the following statement
      is true, which would simplify many proofs, but we will give an example that this is not the case: If
      \[
    G = G_0 \supseteq G_1 \supseteq G_2 \supseteq G_3 \supseteq \cdots
  \]
  is  a sequence of in $G$ normal subgroups with $\bigcap_{i = 0}^{\infty} G_i = \{1\}$
  and $H \subseteq G$ is a normal subgroup of finite index, then there exists a natural number
  $i$ with $G_i \subseteq H$. Here is a counterexample.

Fix two distinct prime numbers $p$ and $q$.
  Put $G = \IZ$, $G_i = p^i \cdot \IZ$,
  and $H = q \cdot \IZ$.

   Here is a
  residually $p$-finite example, where the index $[G:H]$ is a $p$-power for the same prime $p$. Choose
  $x = 1 + a_1 \cdot p + a_2 \cdot p^2 + \cdots \in (\IZ\widehat{_p})^{\times} $ such that
  its class in $\IQ\widehat{_p}$ is not in the image of $\IQ \to \IQ\widehat{_p}$.  Let
  $\pr_i \colon \IZ\widehat{_p} \to \IZ/p^i$ be the canonical projection for
  $i \in \IZ^{\ge 1}$.  In the sequel we view $\IZ$ as a subring of $\IZ\widehat{_p}$.
  Put $G = \IZ^2$ and $H = \{(mp,n) \mid (m,n) \in \IZ^2\}$.  Let $G_i$ be the kernel of
  the epimorphism $G \to \IZ/p^i$ sending $(m,n)$ to $\pr_i(m \cdot x - n)$. Obviously
  $[G:G_i] = p^i$, $G_{i+1} \subseteq G_i$, and $G_i$ is not contained in $H$ for
  $i \in \IZ^{\ge 1}$. Next we show $\bigcap_{i = 1}^{\infty} G_i = \{0\}$.  Consider
  $(a,b) \in \bigcap_{i = 1}^{\infty} G_i$. Then we get $\pr_i(a \cdot x - b) = 0$ for
  $i \in \IZ^{\ge 1}$. This implies $a \cdot x = b \in \IZ\widehat{_p}$.  Suppose that
  $a \not= 0$ or $b \not= 0$ holds. Since $x \in (\IZ\widehat{_p})^{\times}$, we get
  $a \not=0$ and $b \not= 0$ and hence $x = \frac{b}{a}$ in $\IQ\widehat{_p}$.  Since $x$
  is not in the image of $\IQ \to \IQ\widehat{_p}$, we get a contradiction.  Hence
  $(a,b) = (0,0)$.

\end{remark}

%-----------------------------------------------------------------------------

\subsection{Comparing zero characteristic  and prime characteristic }%
\label{subsec:Comparing_zero_characteristic_and_prime_characteristic}

In this subsection we want to prove the following theorem.

   \begin{theorem}\label{the_comparing_L2-Betti_numbers_over_Q_and_F_p}
  Let $G$ be a \RALI-group and let $X$ be a free
  finite $G$-$CW$-complex.

  \begin{enumerate}

  \item\label{the_comparing_L2-Betti_numbers_over_Q_and_F_p:all_primes}
  
  For every field $F$  and $n \in \IZ$ we have the inequality
  \[
  b_n^{(2)}(X;\caln(G)) \le b^{(2)}_n(X;\cald_{F G});
\]

\item\label{the_comparing_L2-Betti_numbers_over_Q_and_F_p:almost_all_primes} There is a
  finite set of primes $\calp_X$ depending on $X$ such that for every prime $p$ with
  $p \notin \calp_X$ and $n \in \IZ$ we have the equality
  \[
  b_n^{(2)}(X;\caln(G)) = b^{(2)}_n(X;\cald_{\IF_pG}).
\]
\end{enumerate}
\end{theorem}

Its proof needs some preparation.

Let $G$ be an amenable locally indicable group.  Note that then $G$ is torsionfree and
satisfies satisfies the strong Atiyah Conjecture,
see~\cite[Theorem~1.1]{Jaikin-Zapirain+Lopez-Alvarez(2020)}. Moreover, for every field $F$
the group ring $FG$ is an integral domain, i.e., all zero-divisors are trivial, and
satisfies the Ore condition with respect to the multiplicative subset $S$ of non-zero
elements, see Remark~\ref{rem:Amenable_locally_indicable_groups}. In particular the Ore
localization $S^{-1} FG$ exists and is a skew field,

Consider $A = (a_{i,j}) \in \M_{m,n}(\IZ G)$ and a prime number $p$. Let
$\overline{A} = (\overline{a_{i,j}})$ be the element in $\M_{m,n}(\IF_p G)$ obtained from
$A$ by applying to each entry the obvious surjective ring homomorphism
$\IZ G \to \IF_p G$. Note that $\IZ G \subseteq S^{-1}\IQ G$ and
$\IF_p G \subseteq S^{-1}\IF_pG$ holds and hence we can consider $A$ as an element of
$\M_{m,n}(S^{-1}\IQ G)$ and $\overline{A}$ as an element of $\M_{m,n}(S^{-1}\IF_p G)$.

\begin{lemma}\label{lem:invertibility_over_Q_and_F_p_ALI}
Let $G$ be an amenable  locally indicable group. Consider $A \in \M_n(\IZ G)$
for some $n \in \IZ^{\ge 1}$.
\begin{enumerate}

\item\label{lem:invertibility_over_Q_and_F_p_ALI:all_primes}
  We get for every prime $p$
  \[
  A \in \GL_n(S^{-1}\IF_p G) \implies A \in \GL_n(S^{-1}\IQ G);
  \]
\item\label{lem:invertibility_over_Q_and_F_p_ALI:almost_all_primes}
  There is a finite set  of primes $\calp_A$ depending on $A$
  such that we get for all primes $p$ with $p \notin \calp_A$ 
  \[
  A \in \GL_n(S^{-1}\IF_p G) \Longleftrightarrow  A \in \GL_n(S^{-1}\IQ G)
  \]
\end{enumerate}
\end{lemma}
\begin{proof}~\ref{lem:invertibility_over_Q_and_F_p_ALI:all_primes} We use induction over
  $n$. The induction beginning $n = 1$ follows directly from the implication
  \[\overline{A} \in \GL_1(S^{-1}\F_p G) \Longleftrightarrow \overline{a_{1,1}} \not = 0
    \implies a_{1,1} \not=  0 \Longleftrightarrow A \in \GL_1(S^{-1}\IQ G).
  \]
  The induction step from $(n-1)$ to $n \ge 2$ is done as follows.  Assume
  $A \in \GL_n(S^{-1}\IF_p G)$. Then at least one of the entries $a_{i,1}$ of the first
  column must satisfy $\overline{a_{i,1}} \not= 0$. We can assume without loss of
  generality $\overline{a_{1,1}} \not= 0$.  By the Ore condition we can find for
  $i = 2, \ldots, n$ elements $b'_{i,1}$ and $c'_{i}$ in $\IQ G$ with $c'_{i} \not= 0$ and
  $b'_{i,1}a_{1,1} = c'_{i}a_{i,1}$ in $\IQ G$.  By possibly multiplying $b'_{i,1}$ and
  $c'_{i}$ with a common element in $\IZ^{\ge 1}$, we can arrange that
  $b'_{i,1}, c'_{i} \in \IZ G$,  $c'_{i} \not = 0$, and $b'_{i,1}a_{1,1} = c'_{i}a_{i,1}$ in $\IZ G$ hold. If
$b'_{i,1} = 0$ holds, put $u_i = v_i = 0$, $b'_{i,1} = b_{i,1}$, and $c'_{i} = c_{i}$. If $b'_{i,1} \not= 0$ holds,
  choose $u_i,v_i \in \IZ^{\ge 0}$ and elements $b_{i,1}, c_{i}$ in $\IZ G$ such that both
  elements $\overline{b_{i,1}}$ and $\overline{c_{i}}$ in $\IF_p G$ are non-trivial and we
  have $b_{i,1}' = p^{u_i} b_{i,1}$ and $c_i' = p^{v_i} c_i$. Then we get
  $p^{u_i} \cdot b_{i,1}a_{1,1} = p^{v_i} \cdot c_{i}a_{i,1}$ in $\IZ G$. Next we show  $u_i  \ge  v_i$.
  Suppose $u_i < v_i$.  Then we get $b'_{i,1} \not= 0$ and $b_{i,1}a_{1,1} = p^{v_i -u_i} \cdot c_{i}a_{i,1}$.
  This   implies $\overline{b_{i,1}}\overline{a_{1,1}} = 0$ in $\IF_pG$.  Since both
  $\overline{b_{i,1}}$ and $\overline{a_{1,1}}$ are non trivial in $\IF_pG$ and $\IF_pG$
  has no non-trivial zero-divisors, we get $\overline{b_{i,1}}\overline{a_{1,1}} \not= 0$
  in $\IF_pG$, a contradiction.  Hence we must have $u_i \ge v_i$. This implies
  $p^{u_i-v_i} \cdot b_{i,1}a_{1,1} = c_{i}a_{i,1}$.  Define an element $D = (d_{i,j})$ in
  $\M_{n}(\IZ G)$ by
  \[
    d_{i,j} =
    \begin{cases}
      1
      &
     \text{if} \;(i,j = (1,1);
    \\
    - p^{u_i-v_i} \cdot b_{i,1} 
    &
    \text{if} \; 2 \le i \le n, j = 1;
    \\
    c_i
    &
    \text{if} \;i = j, 2 \le i \le n;
    \\
    0
    &
    \text{otherwise}.
  \end{cases}
  \]
  Then there is an element $E \in M_{n-1}(\IZ G)$ such that
  $D \cdot A$ is a block matrix of the form
  \[
   D \cdot A
    =
    \begin{pmatrix}
      a_{1,1} & \ast
      \\
      0 & E
    \end{pmatrix}
  \]
  and we have $D \in \GL_{n}(S^{-1}\IQ G)$ and $\overline{D} \in   \GL_{n}(S^{-1}\IF_pG)$.
  As $\overline{A}$ belongs to $\GL_{n}(S^{-1}\IF_pG)$ and
  $\overline{a_{1,1}}$ is a unit in $S^{-1}\IF_p G$, we get
  $\overline{E} \in \GL_{n-1}(S^{-1}\IF_p G)$.  By induction hypothesis
  $E \in \GL_{n-1}(S^{-1}\IQ G)$. Since $a_{1,1}$ is a unit in $S^{-1}\IQ G$, we have
  $A \in \GL_n(S^{-1}\IQ G)$.
  \\[1mm]~\ref{lem:invertibility_over_Q_and_F_p_ALI:almost_all_primes} We use induction
  over $n$. The induction beginning $n = 1$ is proved as follows.  Consider
  $A \in \M_1(\IZ G)$.  If $A$ is the zero matrix, take $\calp_A = \emptyset$. If
  $A$ is not the the zero matrix, let $\calp_A$ be the finite set of primes $p$ for which
  $\overline{a_{1,1}} = 0$.  Suppose that $A \in \GL_1(S^{-1}\IQ G)$. Then
  $a_{1,1} \not= 0$. We get for every prime $p$ satisfying $p \notin \calp_A$ that
  $\overline{a_{1,1}} \not = 0$ and hence $\overline{A} \in \GL_1(S^{-1}\IF_pG)$ hold. Now
  the induction beginning follows from
  assertion~\ref{lem:invertibility_over_Q_and_F_p_ALI:all_primes}.

  The induction step from $(n-1)$ to $n \ge 2$ is done as follows. We begin with the case
  $A \in \GL_n(S^{-1}\IQ G)$.  Analogously to the argument in the proof of
  assertion~\ref{lem:invertibility_over_Q_and_F_p_ALI:all_primes} one shows that one can
  assume without loss of generality that $a_{1,1} \not= 0$ holds and that there is a
  matrix $D = (d_{i,j}) \in \M_n(\IZ G)$ such that the entries on its diagonal are all
  non-trivial, every element of the shape $d_{i,j}$ for $i \not = j$ and $i \not = 1$ is
  trivial and we get an element $E \in \M_{n-1}(\IZ G)$ satisfying
  \[
   D \cdot A
    =
    \begin{pmatrix}
      a_{1,1} & \ast
      \\
      0 & E
    \end{pmatrix}.
  \]
  As $D$ and $A$ belong to $\GL_n(S^{-1}\IQ G)$ and $a_{1,1}$ is a unit in $S^{-1}\IQ G$,
  we have $E \in \GL_n(S^{-1} \IQ G)$. By induction hypothesis there is a finite set of
  primes $\calp_E$ such that for any prime $p \notin \calp_E$ we have
  $\overline{E} \in \GL_{n-1}(S^{-1}\IF_p G)$. Let $\calp_A^1$ be the finite set of
  primes for which both $\overline{D} \notin \GL_{n}(S^{-1}\IF_pG)$ and $\overline{a_{1,1}} = 0$ hold.
  Put $\calp_A = \calp_E \cup \calp_A^1$. Then for any prime
  $p \not \in \calp_A$ we have, $\overline{D} \in \GL_{n}(S^{-1}\IF_pG)$ and 
  $\overline{E} \in \GL_{n-1}{S^{-1}} \IF_p G)$ and
  $\overline{a_{1,1}}$ is a unit in $S^{-1}\IF_p G$ which implies
  $\overline{A} \in \GL_n(S^{-1}\IF_p G)$.
  
  If $A \notin \GL_n(S^{-1} \IQ G)$, we put $\calp_A = \emptyset$.

  This finishes the proof of Lemma~\ref{lem:invertibility_over_Q_and_F_p_ALI} because of
  assertion~\ref{lem:invertibility_over_Q_and_F_p_ALI:all_primes}.
\end{proof}

\begin{lemma}\label{lem:ranks_over_Q_and_F_p_ALI}
  Let $G$ be an amenable  locally indicable group. Consider a $\IZ G$-homomorphism
  $f \colon \IZ G^m \to \IZ G^n$. 
\begin{enumerate}

\item\label{lem:ranks_over_Q_and_F_p_ALI:all_primes}
  We get for every prime $p$
  \[
    \dim_{S^{-1} \IF_p G}(\im(\id_{S^{-1} \F_p G} \otimes_{\IZ G} f))
    \le \dim_{S^{-1} \IQ G}(\im(\id_{S^{-1} \IQ G} \otimes_{\IZ G} f));
  \]
\item\label{lem:ranks_over_Q_and_F_p_ALI:almost_all_primes}
  There is a finite set of primes $\calp_f$  depending on $f$
  such that we get for all primes $p$ with $p \notin \calp_f$ 
  \[
    \dim_{S^{-1} \IF_p G}(\im(\id_{S^{-1} \F_p G} \otimes_{\IZ G} f))
  = \dim_{S^{-1} \IQ G}(\im(\id_{S^{-1} \IQ G} \otimes_{\IZ G} f)).
  \]
\end{enumerate}
\end{lemma}
\begin{proof}
  Let $A \in \M_{m,n}(\IZ G)$ be the matrix describing $f$. We can consider $A$ as an
  element in $\M_{m,n}(S^{-1} \IQ G)$, and
  $\dim_{S^{-1} \IQ G}(\im(\id_{S^{-1} \IQ G} \otimes_{\IZ G} f))$ agrees with the rank
  $\rk_{S^{-1} \IQ G}(A)$ of $A$.  Moreover $\rk_{S^{-1} \IQ G}(A)$ agrees with the
  maximum over all $k \in \{1,2 \ldots, \min\{m,n\}\}$ for which the exists a
  $(k,k)$-submatrix $A'$ of $A$ with $A' \in \GL_k(S^{-1} \IQ G)$.
  We can
  consider $\overline{A}$ as an element in $\M_{m,n}(S^{-1} \IF_p G)$ and
  $\dim_{S^{-1} \IF_p G}(\im(\id_{S^{-1} \F_p G} \otimes_{\IZ G} f))$ is the maximum over
  all $k \in \{1,2 \ldots, \min\{m,n\}\}$ such that the exists a $(k,k)$-submatrix
  $\overline{A}'$ of $A$ with $\overline{A}' \in \GL_k(S^{-1} F_pG)$.  Now
  Lemma~\ref{lem:ranks_over_Q_and_F_p_ALI} follows from
  Lemma~\ref{lem:invertibility_over_Q_and_F_p_ALI}.
\end{proof}

\begin{lemma}\label{lem:Betti-numbers_over_Q_and_F_p_ALI}
  Let $G$ be an amenable locally indicable group. Consider a finitely generated free
  $\IZ G$-chain complex $C_*$  of finite dimension.
\begin{enumerate}

\item\label{lem:Betti_over_Q_and_F_p_ALI:all_primes}
  We get for every prime $p$ and $n \in \IZ^{\ge 0}$
  \[
    \dim_{S^{-1} \IQ G}(H_n(S^{-1} \IQ G \otimes_{\IZ G} C_*))
    \le  \dim_{S^{-1} \IF_p G}(H_n(S^{-1} \IF_p G \otimes_{\IZ G} C_*));
   \]
\item\label{lem:Betti_over_Q_and_F_p_ALI:almost_all_primes}
  There is a finite set  of primes $\calp_{C_*}$ depending on $C_*$
  such that we get for all primes $p$ with $p \notin \calp_{C_*}$ 
  \[
    \dim_{S^{-1} \IQ G}(H_n(S^{-1} \IQ G \otimes_{\IZ G} C_*))
    =  \dim_{S^{-1} \IF_p G}(H_n(S^{-1} \IF_p G \otimes_{\IZ G} C_*)).
   \]
\end{enumerate}
\end{lemma}
\begin{proof}
  Let $c_n \colon C_n \to C_{n-1}$ be the $n$-th differential of $C_*$. Then we get 
  \begin{multline*}
    \dim_{S^{-1} \IQ G}(H_n(S^{-1} \IQ G \otimes_{\IZ G} C_*))
= 
\dim_{S^{-1} \IQ G}(S^{-1} \IQ G \otimes_{\IZ G} C_n)
\\-  \dim_{S^{-1} \IQ G}(\im(\id_{S^{-1} \IQ G} \otimes_{\IZ G} c_{n+1}))
- \dim_{S^{-1} \IQ G}(\im(\id_{S^{-1} \IQ G} \otimes_{\IZ G} c_n)),
\end{multline*}
and
\begin{multline*}
    \dim_{S^{-1} \IF_pG}(H_n(S^{-1} \IF_pG \otimes_{\IZ G} C_*))
= 
\dim_{S^{-1} \IF_pG}(S^{-1} \IF_pG \otimes_{\IZ G} C_n)
\\-  \dim_{S^{-1} \IF_pG}(\im(\id_{S^{-1} \IF_pG} \otimes_{\IZ G} c_{n+1}))
- \dim_{S^{-1} \IF_pG}(\im(\id_{S^{-1} \IF_pG} \otimes_{\IZ G} c_n)),
\end{multline*}
    Since we have
    \[\dim_{S^{-1} \IQ G}(S^{-1} \IQ G \otimes_{\IZ G} C_n)    = \dim_{\IZ G}(C_n)
      =  \dim_{S^{-1} \IF_pG}(S^{-1} \IF_pG \otimes_{\IZ G} C_n),
    \]
    Lemma~\ref{lem:Betti-numbers_over_Q_and_F_p_ALI} follows from
     Lemma~\ref{lem:ranks_over_Q_and_F_p_ALI}.
   \end{proof}

   \begin{remark}\label{rem:Convergence_of_L22_Betti_numbers}
     Let $G$ be a \RALI-group and let $X$ be a free $G$-$CW$-complex of finite type.  Let
     $G = G_0 \supseteq G_1 \supseteq G_2 \supseteq \cdots$ be a sequence of in $G$ normal
     subgroups such that $\bigcap_{i = 0}^{\infty} G_i = \{1\}$ holds and each quotient
     $G/G_i$ is a \RALI-group.  Let $F$ be any field.

    Then we conclude for $i,j \in \IZ^{\ge 0}$ with $i > j$ from Monotonicity applied to  the projection
    $G/G_i \to G/G_j$, see Theorem~\ref{the:Monotonicity}
    \[
    b^{(2)}_n(X/G_i;\cald_{F[G/G_i]}) \le b^{(2)}_n(X/G_j;\cald_{F[G/G_j]}).
  \]
  We conclude from Theorem~\ref{the:Approximation_over_any_F}~\ref{the:Approximation_over_any_F:General_F}
  \[b^{(2)}_n(X;\cald_{FG})  = \lim_{i \to \infty} b^{(2)}_n(X/G_i;\cald_{F[G/G_i]}).
  \]
  Note that $b^{(2)}_n(X;\cald_{FG})$ and $b^{(2)}_n(X/G_i;\cald_{F[G/G_i]})$ are
  integers.  Hence the sequence $b^{(2)}_n(X/G_i;\cald_{F[G/G_i]})$ is monotone decreasing
  and there exists $i_0 \in \IZ^{\ge 0}$ satisfying
  \[b^{(2)}_n(X;\cald_{FG}) = b^{(2)}_n(X/G_i;\cald_{F[G/G_i]}) \quad \text{for}\; i \ge i_0.
  \]
  \end{remark}
  Now we are ready to prove Theorem~\ref{the_comparing_L2-Betti_numbers_over_Q_and_F_p}.

\begin{proof}[Proof of Theorem~\ref{the_comparing_L2-Betti_numbers_over_Q_and_F_p}]%
~\ref{the_comparing_L2-Betti_numbers_over_Q_and_F_p:all_primes}
Because of Lemma~\ref{lem:cald_(FG)_is_cald(FG_subseteq_calu(G))}~%
\ref{lem:cald_(FG)_is_cald(FG_subseteq_calu(G)):dependence_on_F}
  we only have to treat the case, where $F = \IF_p$ for some prime number $p$.
Choose a normal \RALI-chain, i.e., a descending sequence of in $G$ normal subgroups
  \[
    G = G_0 \supseteq G_1 \supseteq G_2 \subset G_3 \supseteq \cdots
  \]
  such that $\bigcap_{i = 0} ^{\infty} G_i = \{1\}$ holds and each quotient group $G/G_i$
  is both amenable and locally indicable. 
  Because of  Remark~\ref{rem:Convergence_of_L22_Betti_numbers} there exists
  $i \in \IZ^{\ge 0}$ with
  \begin{eqnarray*}
    b^{(2)}_n(X;\cald_{\IQ G})
    & = &
    b^{(2)}_n(X/G_i;\cald_{\IQ [G/G_i]});
    \\
     b^{(2)}_n(X;\cald_{\IF_p G})
    & = &
     b^{(2)}_n(X/G_i;\cald_{\IF_p[G/G_i]}).
  \end{eqnarray*}
  Since 
  Lemma~\ref{lem:Betti-numbers_over_Q_and_F_p_ALI}~\ref{lem:Betti_over_Q_and_F_p_ALI:all_primes}
  implies $ b^{(2)}_n(X/G_i;\cald_{\IQ [G/G_i]}) \le b^{(2)}_n(X/G_i;\cald_{F[G/G_i]})$,
  assertion~\ref{the_comparing_L2-Betti_numbers_over_Q_and_F_p:all_primes} follows.
  \\[1mm]~\ref{the_comparing_L2-Betti_numbers_over_Q_and_F_p:almost_all_primes}
  Because of  Remark~\ref{rem:Convergence_of_L22_Betti_numbers} there exists
  $i \in \IZ^{\ge 0}$ satisfying
  \begin{eqnarray*}
    b^{(2)}_n(X;\cald_{\IQ G})
    & = &
    b^{(2)}_n(X/G_i;\cald_{\IQ [G/G_i]}).
  \end{eqnarray*}
  From Lemma~\ref{lem:Betti-numbers_over_Q_and_F_p_ALI}~%
\ref{lem:Betti_over_Q_and_F_p_ALI:almost_all_primes}
  applied to $G/G_i$ and $X/G_i$ we obtain
  a finite set of primes $\calp$ such that for every prime $p$ with $p \notin \calp$ we get
  \[
    b^{(2)}_n(X/G_i;\cald_{\IQ [G/G_i]}) = b^{(2)}_n(X/G_i;\cald_{\IF_p [G/G_i]}).
  \]
  Now we conclude from assertion~\ref{the_comparing_L2-Betti_numbers_over_Q_and_F_p:all_primes}
  and Remark~\ref{rem:Convergence_of_L22_Betti_numbers} for every prime  $p$ with $p \notin \calp$ 
  \begin{multline*}
   b^{(2)}_n(X;\cald_{\IQ G}) =  b^{(2)}_n(X/G_i;\cald_{\IQ [G/G_i]}) =  b^{(2)}_n(X/G_i;\cald_{\IF_p [G/G_i]})
   \\
    \ge  b^{(2)}_n(X;\cald_{\IF_p G}) \ge b^{(2)}_n(X;\cald_{\IQ G}),
 \end{multline*}
 which implies  $b_n^{(2)}(X;\caln(G)) = b_n(X;\cald_{\IQ G}) = b^{(2)}_n(X;D_{\IF_pG})$.
\end{proof}

\begin{conjecture}\label{con:_dimension_1}
 The equality $b_1^{(2)}(G;\caln(G)) = b_1^{(2)}(G;\cald_{FG})$ holds for
 all finitely presented groups $G$ and every field $F$.
\end{conjecture}

Note  that this is not the case in degrees $n \ge 2$, see~\cite[Theorem~1]{Avramidi-Okun-Schreve(2021)}.

%%%%%%%%%%%%%%%%%%%%%%%%%%%%%% Section 5 %%%%%%%%%%%%%%%%%%%%%%%%%%%%%%%%
%%%%%%%%%%%%%%%%%%%%%%%%%%%%%%%%%%%%%%%%%%%%%%%%%%%%%%%%%%%%%%%%%%%%%

\typeout{---------- Section 5: A conjecture due to Gromov and Wise   ---------------}

\section{A question due to Gromov and Wise}%
\label{sec:A_conjecture_due_to_Gromov_and_Wise}

% ----------------------------------------------------------------------------

\subsection{Towers}\label{subsec:towers}

The following definition is a variation of~\cite[Definition~2.1]{Wise(2022)}.

     \begin{definition}[(Connected $\calf$-)tower]\label{def:tower}
       Consider  connected finite $CW$-complexes  $X$ and $Y$.
       A  \emph{tower}  with $X$ as source and $Y$ as target is a sequence of maps
    \[X = Y_n \xrightarrow{i_n}  \overline{Y_{n}}  \xrightarrow{p_{n}} Y_{n-1} \xrightarrow{i_{n-1}}  \overline{Y_{n-2}}
      \xrightarrow{p_{n-2}}       \cdots  \xrightarrow{p_2} Y_{1}\xrightarrow{i_1}  \overline{Y_{1}}
      \xrightarrow{p_{1}}  Y_{0} = Y
   \]
   for $n \in \IZ^{\ge 1}$, where the map $i_k$ is the inclusions of a finite subcomplex
   $Y_k$ and the map $p_k$ is a regular covering with total space $\overline{Y_k}$ for
   $k = 1,2, \ldots, n$.

   The composite of the maps in the tower gives the \emph{underlying map} $f \colon X \to Y$.
   In this case we also say  that $f$ \emph{factorizes as a tower}.

   The number $n$ is the \emph{length of the tower}. An  \emph{elementary tower} $X \xrightarrow{i} \overline{Y} \to Y$
   is just a tower of length $1$, i.e., an inclusion of $CW$-complexes $i$ and a regular covering $p$.

   If $\calf$ is a family of groups and the deck transformation group $G_k$ of the
   covering $p_k$ belongs to $\calf$ for $k = 1,2, \ldots, n$, we call it a
   \emph{$\calf$-tower}.

   If we additionally assume that the spaces $\overline{Y_k}$ and $Y_k$ are connected for
   $k = 1,2, \ldots, n$, we call it a \emph{connected tower} or \emph{connected
     $\calf$-tower}.
 \end{definition}

 Recall that for a discrete group $G$ a $G$-covering is the same as a principal $G$-bundle
 and a covering $q \colon \overline{Z} \to Z$ is regular if and only if it is a
 $G$-covering for the group $G$ of deck transformations of $q$.

 \begin{definition}[Non-positive tower]\label{def_non-positive_tower}
   A tower in the sense of Definition~\ref{def:tower} is called \emph{non-positive} if
   either the Euler characteristic of $X$ satisfies $\chi(X) \le 0$ or $X$ is
   contractible.
 \end{definition}
 \begin{definition}[Tower lift]
   An \emph{$\mathcal F$-tower lift} of $f:X\rightarrow Y$ is a factorization of $f$ as a
   composition $X\rightarrow C\rightarrow Y$ such that $C\rightarrow Y$ is an
   $\mathcal F$-tower. It is a \emph{maximal $\mathcal F$-tower lift} if $X\rightarrow C$
   has no non-trivial $\mathcal F$-tower lift.
 \end{definition}
 Note that if $X$ is a finite CW complex and $f$ is a cellular map, then a maximal
 $\mathcal F$-tower lift can be constructed by taking successive lifts, which terminates
 after finitely many steps because of the finiteness assumption, see~\cite[ Lemma
 2.2]{Wise(2022)}.

  % ----------------------------------------------------------------------------

  \subsection{A conjecture of Wise in two-dimensional topology}

  We are interested in the following question.

  \begin{question}\label{que:Gromov_Wise}
    Are the following statements for a finite connected $2$-dimensional $CW$-complex $Y$ equivalent?

    \begin{enumerate}
    \item\label{que:Gromov_Wise:tower} Every connected tower  with $Y$ as target is non-positive;

    \item\label{que:Gromov_Wise:L2_Betti} The second $L^2$-Betti number
      $b_2^{(2)}(\widetilde{Y}) = b_2^{(2)}\bigl(\widetilde{Y};\caln(\pi_1(Y))\bigr)$
      of the universal covering with respect to the $\pi_1(Y)$-action vanishes.

    \end{enumerate}
    
  \end{question}

  \begin{remark}\label{rem:passing_to_connected_towers}
    Note that we assume in Definition~\ref{def:tower} that $X$ and $Y$ are
    connected.  Then it is easy to construct from an arbitrary tower with underlying map $f \colon X \to Y$
    a connected tower with underlying map $f \colon X \to Y$.  Namely
    replace successively for $k = (n-1), (n-2), \ldots ,1$ the space $\overline{Y_{k}}$
    by the component $\overline{C}_k$ of $\overline{Y_{k}}$ containing $i_{k+1}(Y_{k+1})$ and $Y_k$ by the 
    path component $C_k$ of $Y_k$ which contains the image of $\overline{C_k}$ under
    $p_k \colon \overline{Y_k} \to Y_k$.  Then $p_k$ induces a regular covering
    $p_k' \colon \overline{C}_k \to C_k$ whose deck transformation group is a subgroup of
    the deck transformation group of $p_k$.

    If $\calc$ is closed under taking subgroups, then by the construction above any
    $\calc$-tower with $X$ as source
    and $Y$ as target can be replaced by a connected $\calc$-tower with $X$ as source
    and $Y$ as target. 

    Note that the statement~\ref{que:Gromov_Wise:tower} is equivalent to the statement that
    every tower $f \colon X \to Y$ is non-positive.
\end{remark}

We call  a tower a \emph{$\IZ$-tower} if the tower is a $\calc$-tower with respect to the class of groups
which are either trivial or isomorphic to $\IZ$.

\begin{lemma}\label{lem:zaspherical}
  Let $Y$ be a connected finite $2$-dimensional $CW$-complexes.
  Suppose that  all $\IZ$-towers with  $Y$ as target are non-positive.

  Then $Y$ is  aspherical.
\end{lemma}
\begin{proof}
  It is enough to show that any map $f \colon S^2\to Y$ is null-homotopic. Take a maximal
  $\IZ$-tower lift $S^2\rightarrow C\rightarrow Y$ of such a map. Then $C$
  is a finite, connected $2$-complex with $b_1(C;\mathbb Q)=0$. Hence we get
  $\chi(C)\geq 1$.  Since all $\IZ$-towers of $Y$ are non-positive, $C$ is
  contractible. Hence $f$ is null-homotopic.
\end{proof}

\begin{lemma}\label{lem:zlocind}
  Let  $Y$ be a connected finite $2$-dimensional $CW$-complex.
  Suppose  that all $\IZ$-towers with $Y$ as target are non-positive.
  
  Then  $\pi_1(Y)$ is locally indicable.
\end{lemma}
\begin{proof}
  Let $G\subseteq \pi_1(Y) $ be a non-trivial finitely generated subgroup. To prove the
  lemma, we need to show that $G$ surjects onto $\IZ$. Note that the abelianization of $G$
  is finitely presented, so there is a finitely presented group $G'$ and a surjection
  $G' \to G$ that induces an isomorphism on abelianizations, see~\cite[Lemma 3.1]{Wise(2022)}.
  Let $P$ be a finite
  presentation $2$-complex for $G'$ and $P \to X$ the map inducing
  $G' \to G \subseteq \pi_1X$. Take a maximal $\IZ$-tower lift $P \to C \to X$. Then $C$
  is a finite, connected $2$-complex. Since $G$ is non-trivial and $\pi_1(P) \to \pi_1(X)$
  has $G$ as image, $\pi_1(C)\not= \{1\}$ and hence $C$ is not contractible. Since $X$ has
  non-positive $\IZ$-towers, we conclude that $\chi(C)\leq 0$ and hence
  $b_1(C;\IQ)\geq 1$. Since $C$ is maximal, we cannot take a further infinite cyclic lift,
  so the map $P\to C$ induces a surjection $H_1(P;\IQ)\rightarrow H_1(C;\IQ)$. Therefore,
  $b_1(G;\IQ) = b_1(G';\IQ) = b_1(P;\IQ )\geq b_1(C;\IQ)\geq 1$, implying that $G$
  surjects onto $\IZ$.
\end{proof}

So Question~\ref{que:Gromov_Wise} leads in view of Lemma~\ref{lem:zlocind} to the
following conjecture.

\begin{conjecture}[Wise]\label{con:Wise}
 The following statements for a finite connected $2$-dimen\-sio\-nal $CW$-complex $Y$ are  equivalent:

 \begin{enumerate}
 \item\label{con:Wise:tower} Every connected tower  with $Y$ as target is non-positive;

 \item\label{con:Wise:tower_locally_indicable} Every connected tower  with $Y$ as target
   \[\quad \quad \quad X = Y_n \xrightarrow{i_n}  \overline{Y_{n}}
     \xrightarrow{p_{n}} Y_{n-1} \xrightarrow{i_{n-1}}  \overline{Y_{n-2}}
      \xrightarrow{p_{n-2}}       \cdots  \xrightarrow{p_2} Y_{1}\xrightarrow{i_1}  \overline{Y_{1}}
      \xrightarrow{p_{1}}  Y_{0} = Y
   \]
  is non-positive, provided that the fundamental group of each space appearing in it is locally indicable;

    \item\label{con:Wise:L2_Betti} The fundamental  group $\pi_1(Y)$ is locally indicable
      and the second $L^2$-Betti number
      $b_2^{(2)}(\widetilde{Y}) = b_2^{(2)}\bigl(\widetilde{Y};\caln(\pi_1(Y))\bigr)$
      of the universal covering with respect to the $\pi_1(Y)$-action vanishes.

    \end{enumerate}

  \end{conjecture}

  \begin{remark}\label{rem:question_and_the_conjectire_of_Wise}
    Daniel Wise suggested a geometric characterization for the vanishing of the top
    $L^2$-Betti number of a $2$-complex in terms of non-positivity of Euler
    characteristics of immersed subcomplexes,
    see~\cite[Conjecture~2.6]{Wise(2022criteria)}. This characterization turned out to be
    incorrect because local modifications of a $2$-complex (not affecting the homotopy
    type and hence leaving the $L^2$-Betti numbers unchanged) can be done to introduce
    immersed $2$-spheres (see e.g., the introduction of~\cite{Freedman-Nguyen-Phan(2021)} for a source
    of immersed $2$-spheres in contractible $2$-complexes). 
    Chemtov and Wise suggested the  modification appearing as Conjecture~\ref{con:Wise} above.
  \end{remark}

  \begin{remark}\label{rem:homotopy_invariance_and_Whitehead}
  The vanishing of $b_2^{(2)}(\widetilde{Y})$ is a homotopy invariant property of a finite
  connected $2$-dimensional $CW$-complex $Y$. If Conjecture~\ref{con:Wise}
  holds, the property that for a connected finite $2$-dimensional $CW$-complex
  $Y$ any connected tower with $Y$ as target is non-positive, has also to be homotopy
  invariant. This is an interesting and open claim. It implies because of
  Lemma~\ref{lem:zaspherical} that connected subcomplexes of finite contractible
  $2$-complexes are aspherical, which is a special case of the Whitehead Conjecture and is
  an open and difficult problem. Recall that the Whitehead Conjecture predicts that a
  connected $CW$-subcomplex of an aspherical finite $2$-dimensional $CW$-complex is
  aspherical again.
\end{remark}

%------------------------------------------------------------------------------

\subsection{Evidence for Wise's conjecture\label{wiseconjectureexamples}} 
Here are three classes of $2$-complexes with the non-positive towers property:
\begin{itemize}
\item (One relator complexes) Let $\Gamma$ be a finite graph and $w:S^1\rightarrow\Gamma$
  an immersion representing a word in the fundamental group that is not a proper
  power. Then the $2$-complex $\Gamma_w: =\Gamma\cup_w D^2$ obtained by attaching a $2$-disk
  along $w$ has non-positive towers, see~\cite{Helfer-Wise(2016),Louder-Wilton(2017)};
  
\item Spines of aspherical $3$-manifolds with non-empty boundary have non-positive towers,
see~\cite[1.3]{Chemtov-Wise(2022)};
\item 
(Some quotients of Davis complexes)
Let $L$ be a flag graph and denote by 
\[
W_L:=\left<s_v,v\in L^{(0)}\mid s_v^2=1,  [s_v,s_w]=1\mbox{ if }v \mbox{ is adjacent to } w\right>
\] 
the corresponding right angled Coxeter group. The commutator subgroup $C_L:=[W_L,W_L]$ is
a finite index torsion-free subgroup of $W_L$. It is the fundamental group of a locally
CAT(0) 
square complex $P_L$, which can be defined as the subcomplex of the cube
$[-1,1]^{L^{(0)}}$ given by
\[
P_L:=\bigcup_{\sigma\subset L}[-1,1]^{\sigma^{(0)}}\times\{-1,1\}^{(L-\sigma)^{(0)}}.
\]
The $2$-complex $P_L$ has non-positive towers if and only if for every connected subgraph
$K\subset L$ that is not a vertex or a edge, we have $\chi(P_K)\leq 0$. 
\begin{remark}
  This follows from~\cite[4.5]{Jankiewicz-Wise(2016)} and
  in~\cite[11.12]{Wise(2020invitation)}, where one should note that the compression
  complex $\overline X$ (associated the commutator subgroup $C_L$ of $W_L$) in the first
  reference agrees with $P_L$.
\end{remark}
\end{itemize}

The $2$-complexes with the non-positive towers property appearing in the first and second
bullet point have vanishing second $L^2$-Betti number,
see~\cite[Theorem~4.2]{Dicks-Linnell} for one-relator
groups,~\cite[Theorem~0.1]{Lott-Lueck(1995)} for $3$-manifold groups. In the situation of
the third bullet point, there are partial results in two situations.

\begin{enumerate}
\item If the flag graph $L$ is \emph{planar} then $b_2^{(2)}(\widetilde P_L;\mathcal N(C_L))=0$, see~\cite[Theorem
  11.4.1]{Davis-Okun(2001)}.  Here is a sketch of their argument: Embed $L$ as a full
  subcomplex in a flag triangulation $T$ of $S^2$. Then $P_T$ is an aspherical
  $3$-manifold, hence it satisfies Singer, hence $b_2^{(2)}(\widetilde P_T;\mathcal N(C_T))=0$. A Mayer-Vietoris
  argument shows that removing the vertices of $T-L$ one by one does not affect
  $b_2^{(2)}(\widetilde P_{-};\mathcal N(C_{-}))$ vanishing.
\item If the flag graph $L$ is \emph{trivalent} then $b_2^{(2)}(\widetilde P_L;\mathcal N(C_L)) =0$ unless $L=K_{3,3}$
  is the utilities graph,\footnote{Note that $b_2^{(2)}(P_{K_{3,3}})$ does not vanish
    since $\chi(P_{K_{3,3}})=2^6(1-6/2+9/4)$ is positive.}  in~\cite[Theorem
  6.1]{Avramidi-Okun-Schreve(2024subdivision)}.
\end{enumerate}

%----------------------------------------------------------------------------

\subsection{Maximal residually $\calc$ coverings}%
\label{subsec:Maximal_residually_calc_coverings}

Let $\calc$ be a class of  groups which is closed under taking subgroups and finite direct products.
A group $G$ is called \emph{residually $\calc$}
if there exists  a \emph{normal
  $\calc$-chain}, i.e., a, descending sequence of in $G$ normal subgroups
  \[
    G = G_0 \supseteq G_1 \supseteq G_2 \supseteq G_3 \supseteq \cdots
  \]
  such that $\bigcap_{i = 0} ^{\infty} G_i = \{1\}$ holds and each quotient group $G/G_i$
  lies in $\calc$.

  \begin{lemma}\label{lem:equivalent_dcesription_of_residually_calc}
    Let $G$ be a group. Consider the following statements:

    \begin{enumerate}
    \item\label{lem:equivalent_dcesription_of_residually_calc:chain}
      The group $G$ is residually $\calc$;
      
    \item\label{lem:equivalent_dcesription_of_residually_calc:elements} For every element
      $g \in G$ different from the unit $e_G \in G$, there exists a surjective group
      homomorphism $f \colon G \to Q$ such that $Q$ belongs to $\calc$ and
      $f(g) \not = e_Q$ holds.

    \end{enumerate}
    Then~\ref{lem:equivalent_dcesription_of_residually_calc:chain}
    $\implies$~\ref{lem:equivalent_dcesription_of_residually_calc:elements} holds.
    
    If $G$ is countable, then we get~\ref{lem:equivalent_dcesription_of_residually_calc:chain}
    $\Longleftrightarrow$~\ref{lem:equivalent_dcesription_of_residually_calc:elements}.
  \end{lemma}
  \begin{proof} The implication~\ref{lem:equivalent_dcesription_of_residually_calc:chain}
    $\implies$~\ref{lem:equivalent_dcesription_of_residually_calc:elements} is obvious. It
    remains to show the implication~\ref{lem:equivalent_dcesription_of_residually_calc:elements}
    $\implies$~\ref{lem:equivalent_dcesription_of_residually_calc:chain}, provided that
    $G$ is countable.  Choose an enumeration $G = \{g_0,g_1, g_2, \ldots \}$ with
    $g_0 = e_G$. We construct inductively for every $n = 0,1,2, \ldots $ a normal subgroup
    $G_n \subseteq G$ such that $G/G_n$ belongs to $\calc$, $G_{n} \subseteq G_{n-1}$
    holds, and we have $\{g_0,g_1, \ldots g_n\} \cap G_n = \{e_G\}$. For the induction
    beginning $n = 0$ take $G_0 = G$. The induction step from $(n-1)$ to $n \ge 1$ is done
    as follows. Choose a normal subgroup $H$ such that $G/H$ belongs to $\calc$ and
    $g_n \not \in H$ holds. Put $G_n = H \cap H_{n-1}$.  Since $G/H_n$ is a subgroup of
    $G/G_{n-1} \times G/H$, we get $G/H_n \in \calc$. We have $H_{n-1} \subseteq H_n$ and
    \begin{multline*}
     \{e_G\} \subseteq \{g_0,g_1, \ldots, g_n\} \cap G_n = \{g_0,g_1, \ldots, g_n\} \cap G_{n-1} \cap H
     \\
     \subseteq  \{g_0,g_1, \ldots, g_{n-1}\} \cap G_{n-1} = \{e_G\}.
   \end{multline*}
   This implies  $\bigcap_{n = 0}^{\infty} G_n = \{e_G\}$.
 \end{proof}

 \begin{lemma}\label{lem_maximal_residually_calc_quotient}
   Let $G$ be a countable group. Then there exists a \emph{maximal residually $\calc$-quotient},
   i.e., a surjective group homomorphism $\phi \colon G \to \widehat{G}$ with the property that
   $\widehat{G}$ is residually $\calc$,
   such that for any surjective group homomorphism $\psi \colon G \to Q$ with the property that
   $Q$ is residually $\calc$,
   there exists precisely one group homomorphism $\widehat{\psi } \colon \widehat{G} \to Q$ with
   $\psi  = \widehat{\psi } \circ \phi$.
 \end{lemma}
 \begin{proof}
   Let $K$ be the normal subgroup of $G$ which is the intersection of the kernels of
   surjective group homomorphism $G \to Q$ with the property that $Q$ is residually
   $\calc$. Let $\phi \colon G \to \widehat{G} = G/K$ be the projection.  Then for any 
   surjective group homomorphism $\psi  \colon G \to Q$ with the property that $Q$ is
   residually $\calc$, there exists precisely one group homomorphism
   $\widehat{\psi } \colon \widehat{G} \to Q$ with $\psi = \widehat{\psi} \circ \phi$. It
   remains to show that $\widehat{G}$ is residually $\calc$.  This follows from
   Lemma~\ref{lem:equivalent_dcesription_of_residually_calc} since for every $\widehat{g}$
   in $\widehat{G}$ with $\widehat{g} \not= e_{\widehat{G}}$ there exists a surjective
   group homomorphism $\mu  \colon G \to Q$ with the property that $Q$ is residually $\calc$
   and $\mu(\widehat{g}) \not= e_Q$ holds.
 \end{proof}

 Let $X$ be a connected $CW$-complex with countable fundamental group $\pi =
 \pi_1(X)$. Let $\widetilde{p} \colon \widetilde{X} \to X$ be its universal covering.  Let
 $\Phi \colon \pi \to \widehat{\pi}$ be the \emph{maximal residually $\calc$-quotient}.
 Define the  \emph{maximal residually $\calc$-covering}
   $\widehat{p} \colon \widehat{X} \to X$ to be the projection for
   $\widehat{X} = X/\ker(\Phi)$. Then $\widehat{\pi}$ is residually $\calc$
   and $\widehat{p} \colon \widehat{X} \to X$ is a
   $\widehat{\pi}$-covering which has the following property: For any $H$-covering
   $q \colon \overline{X} \to X$ with the property that $\overline{X}$ is connected and
   $H$ is residually $\calc$ there is precisely one covering
   $\widehat{q} \colon \widehat{X} \to \overline{X}$ such that
   $q \circ \widehat{q} = \widehat{p}$ holds.

   If $\calc$ consists of all groups, then $\widehat{\pi} = \pi$,
   $\widehat{X} = \widetilde{X}$, and $\widehat{p} = \widetilde{p}$ hold.  If $\calc$
   consists of only of one element, namely the trivial group, then
   $\widehat{\pi} = \{e\}$, $\widehat{X} = X$, and $\widehat{p} = \id_X$ hold.

   For us the cases are interesting, where $\calc$ consists of all $\ALI$-groups, all
   finite groups, all finite $p$-groups for a prime $p$, and all torsionfree nilpotent groups.

  %------------------------------------------------------------------------

  \subsection{Some basic properties of $L^2$-Betti numbers and coverings}%
\label{subsec:Some_basic_properties_of_L2-Betti_numbers_and_coverings}

We collect some facts about $L^2$-Betti numbers which are useful for answering
Question~\ref{que:Gromov_Wise}.  Note for the sequel that $L^2$-Betti numbers
$b_n^{(2)}(Z;\caln(G))$ are defined and studied for any $G$-space $Z$ and in particular
for any $G$-$CW$-complex $Z$, see~\cite[Chapter~6]{Lueck(2002)} if one allows that they
take values in $\IR^{\ge 0 } \amalg \{\infty\}$. The analogous statement holds if $G$ is a
\RALI-groups and we consider $b_n^{(2)}(Z;\cald_{FG})$.

\begin{lemma}\label{lem:vanishing_of_the_second_L2_Betti_number_and_coverings}
  Let $p \colon \overline{X} \to X$ be a $G$-covering of a connected finite
  $2$-dimensional $CW$-complex $X$. Let $F$ be any field.
  Suppose  that $G$ is non-trivial and  that one of the following condition holds:
  \begin{enumerate}
  \item\label{lem:vanishing_of_the_second_L2_Betti_number_and_coverings:(1)}
    $b_2^{(2)}(\overline{X};\caln(G)) = 0$ holds.
  \item\label{lem:vanishing_of_the_second_L2_Betti_number_and_coverings:(2)}
    The group $G$  is a \RALI-group and $b_2^{(2)}(\widetilde{X};\cald_{F[\pi_1(Y)]}) = 0$.
  \end{enumerate}
  
    Then  $\chi(X) \le 0$. 
\end{lemma}
\begin{proof} We give the proof only for
  condition~\ref{lem:vanishing_of_the_second_L2_Betti_number_and_coverings:(1)}, the one
  for condition~\ref{lem:vanishing_of_the_second_L2_Betti_number_and_coverings:(2)} is
  completely analogous using the fact that a non-trivial \RALI-group is infinite.  Suppose
  that $G$ is infinite. Then $b_0^{(2)}(\overline{X};\caln(G)) = 0$ by
  Theorem~\ref{the:zeroth_L2-Betti_number}. We conclude from
  Theorem~\ref{the:Euler-Poincare_formula}
  \[
    \chi(X) = \chi^{(2)}(\overline{X};\caln(G)) = \sum_{n \ge 0} (-1)^n \cdot
    b_n^{(2)}(\overline{X};\caln(G)) = - b_1^{(2)}(\overline{X};\caln(G)) \le 0.
  \]

  Suppose that $|G|$ is finite.  Since
  $b_m^{(2)}(\overline{X};\caln(G)) = \frac{b_m(\overline{X})}{|G|}$ for $m \ge 0$ by
  Theorem~\ref{the:zeroth_L2-Betti_number}~\ref{the:zeroth_L2-Betti_number:N(G)},
  we get $b_2(\overline{X}) = 0$. If $b_1(\overline{X})$ vanishes,
  then for any $g \in G$ the Lefschetz number of
  $l_g \colon \overline{X} \to \overline{X}$ given by multiplication with $g$ is one and
  hence has a fixed point. Hence we have $b_1(\overline{X}) \ge 1$, since $G$ is by
  assumption non-trivial and acts freely on $\overline{X}$.  We conclude
  \[\chi(X) = \frac{\chi(\overline{X})}{|G|} = \frac{1}{|G|} \cdot \left(\sum_{n \ge 0} (-1)^n \cdot
      b_n(\overline{X}) \right) = \frac{1 -b_1(\overline{X})}{|G|} \le 0.
  \]

\end{proof}

\begin{lemma}\label{lem:vanishing_of_the_second_L2_Betti_number_and_universal_coverings}
  Let $p \colon \widetilde{X} \to X$ be the universal covering of a connected finite
  $2$-dimensional $CW$-complex $X$. Let $F$ be any field. Suppose that one of the
  following conditions holds:
  \begin{enumerate}
  \item We have $b_2^{(2)}(\widetilde{X}) = 0$;
  \item The fundamental  group $\pi_1(X)$ is a \RALI-group and $b_2^{(2)}(\widetilde{X};\cald_{F[\pi_1(X)]}) = 0$.
  \end{enumerate}

  Then $X$ is aspherical. If $X$ is not simply connected, then $\chi(X) \le 0$.
\end{lemma}
\begin{proof} We conclude from
  Lemma~\ref{lem:vanishing_of_the_second_L2_Betti_number_and_coverings} that
  $\chi(X) \le 0$ holds, provided that $X$ is not simply connected.

  We conclude from Theorem~\ref{the:On_the_top_L2-Betti_number}
  that  $b_2(\widetilde{X}) = 0$ or $b_2(\widetilde{X};F) = 0$ holds.
  Since $X$ is $2$-dimensional, we get $H_m(\widetilde{X};\IZ) = 0$ for $m \in \IZ^{\ge 2}$.
  Now the Hurewicz Theorem implies that
  $\widetilde{X}$ is contractible, or, equivalently,  that $X$ is aspherical.
\end{proof}

\begin{lemma}\label{lem:making_the_total_space_connected}
  Let $p \colon \overline{X} \to X $ be a $G$-covering for a connected (not necessarily
  finite) $CW$-complex as basis.  Let $C$ be a path component of $\overline{X}$ and $H$ be
  the subgroup of $G$ consisting of elements $g \in G$ for which $gC = C$ holds in
  $\pi_0(\overline{X})$. Let $F$ be a field.

  Then the restriction $p|_C \colon C \to X$ of $p$ to $C$ is an $H$-covering satisfying
  \[
    b_m^{(2)}(C;\caln(H)) = b_m^{(2)}(\overline{X};\caln(G)) \quad \text{for}\; m  \in \IZ^{\ge 0}
  \]
  and, provided that $G$ is a \RALI-group,
  \[
    b_m^{(2)}(C;\cald_{FH}) = b_m^{(2)}(\overline{X};\cald_{FG}) \quad \text{for}\; m  \in \IZ^{\ge 0}.
  \]
\end{lemma}
\begin{proof} One easily checks that we have a $G$-homeomorphism
  $ G \times_H C \xrightarrow{\cong} \overline{X}$ and  that
  $p|_C \colon C \to X$ is an $H$-covering. Now apply
  Theorem~\ref{the:induction}.
\end{proof}

\begin{lemma}\label{lem:restricting_to_a_subcomplex}
  Let $p \colon \overline{X} \to X $ be a $G$-covering for a (not necessarily finite or connected)
  $d$-dimensional $CW$-complex as basis such that $b_d^{(2)}(\overline{X};\caln(G)) = 0$.
  Consider a (not necessarily finite or connected) subcomplex $Z \subseteq X$ and let
  $p|_Z \colon \overline{Z} \to Z$ be the restriction of $p$ to $Z$. 
\begin{enumerate}
\item\label{lem:restricting_to_a_subcomplex:N(G)}
  
We have $b_d^{(2)}(\overline{Z};\caln(G)) =  0$ if $b_d^{(2)}(\overline{X};\caln(G)) = 0$ holds;

\item\label{lem:restricting_to_a_subcomplex:cald_FG} Suppose that $G$ is \RALI-group.
  Then we get $b_d^{(2)}(\overline{Z};\cald_{FG}) = 0$ provided that
  $b_d^{(2)}(\overline{X};\cald_{FG}) = 0$ holds.
\end{enumerate}
\end{lemma}
\begin{proof}~\ref{lem:restricting_to_a_subcomplex:N(G)}
    Let   $c'_d \colon C_*(\overline{Z}) \to C_{d-1}(\overline{Z})$ be the $d$-th differential of the cellular
    $\IZ G$-chain complex of $\overline{Z}$. The following diagram commutes
    \[
      \xymatrix{\caln(G) \otimes_{\IZ G}  C_d(\overline{Z}) \ar[r] \ar[d]_{\id_{\caln(G) \otimes_{\IZ G} c'_d}}
      &
     \caln(G) \otimes_{\IZ G}  C_d(\overline{X}) \ar[d]^{\id_{\caln(G) \otimes_{\IZ G} c_d}}
    \\
    \caln(G) \otimes_{\IZ G}  C_{d-1}(\overline{Z}) \ar[r] 
      &
      \caln(G) \otimes_{\IZ G}  C_{d-1}(\overline{X})}
  \]
  and  the horizontal  arrows induced by the inclusion $\overline{Z} \to \overline{X}$ are split injective.
  Hence the left vertical arrow is injective if the right vertical  arrow is injective. Now the claim
  follows from Lemma~\ref{lem:injectivity_top_differential}~\ref{lem:injectivity_top_differential:caln(G)}.
\\[1mm]~\ref{lem:restricting_to_a_subcomplex:cald_FG} This proof is analogous.
\end{proof}

% ----------------------------------------------------------------------

\subsection{From $L^2$-Betti numbers to towers}\label{subsec:From_L2-Betti_numbers_to_towers}

Note for the sequel that for a $G$-covering $p \colon \overline{X} \to X$
with connected $\overline{X}$ over a connected finite $CW$-complex the group $G$ is a quotient of the
finitely generated group $\pi_1(X)$ and hence is finitely generated and in particular countable. 

\begin{definition}[$L^2$-Betti numbers of maximal residually 
  $\calc$-coverings]\label{def:L2-Betti_numbers_of_maximal_calc_covering}
  Let $\calc$ be a class of  groups which is closed under taking subgroups and finite direct products.
  For a connected finite $CW$-complex $X$, let $\widehat{X}^{\calc} $ denote the
  maximal residually  $\calc$-covering of $X$, and set
  $\cald_{\widehat{X}^{\calc};\IQ} :=\cald_{\IQ[\pi_1(X)/\pi_1(\widehat X^{\calc})]}$.  Then the $L^2$-Betti
  number
  \[
    b_n^{(2)}(\widehat{X}^{\calc},\cald_{\widehat{X}^{\calc};\IQ}) =
    b_n^{(2)}(\widehat{X}^{\calc}; \caln(\pi_1(X)/\pi_1(\widehat X^{\calc})))\in \IR^{\ge 0}
  \]
  is defined. If $\calc$ is the class $\calALI$ of $\ALI$-groups, we omit $\calc$ from
  the notation and write $\widehat{X}$, $\cald_{X;\IQ}$, and
  $b_n^{(2)}(\widehat{X},\cald_{X;\IQ})$.

  If $\calc$ is contained in $\calALI$ and $F$ is any field,
  then we put $\cald_{\widehat{X}^{\calc};F} :=\cald_{F[\pi_1(X)/\pi_1(\widehat X^{\calc})]}$
  and define  the $L^2$-Betti number over $F$
  \[
    b_n^{(2)}(\widehat{X}^{\calc},\cald_{\widehat{X}^{\calc};F}) = b_n^{(2)}(\widehat{X}^{\calc}; \cald_{F[\pi_1(X)/\pi_1(\widehat{X}^{\calc})]})
    \in \IZ^{\ge 0}.
  \]
  If $\calc$ is $\calALI$, we omit $\calc$ from the notation and write $\widehat{X}$,
  $\cald_{X;F}$, and $b_n^{(2)}(\widehat{X},\cald_{\widehat{X};F})$.
\end{definition}

In the sequel $\calRC$ denotes the family of residually $\calc$ groups.
Note that with this notation $\calRALI$ is $\calRC$ for $\calc = \calALI$.
   
\begin{theorem}\label{the:vanishing_of_the_top_L2-Betti_numbers_and_RALI-towers}
  Consider $d \in \IZ^{\ge 2}$ and let $F$ be any field. Let $\calc$ be a class of groups
  which is closed under taking subgroups and finite products and satisfies
  $\calc \subseteq \calALI$. Let $Y$ be a connected finite $d$-dimensional $CW$-complex
  with $b_d^{(2)}(\widehat{Y}^{\calc};\cald^{\calc}_{\widehat{Y};F}) = 0$.  Consider any
  connected $\calRC$-tower
  with  $Y$ as target in the sense of Definition~\ref{def:tower}
  \[
    X = Y_n \xrightarrow{i_n} \overline{Y_{n-1}} \xrightarrow{p_{n-1}} Y_{n-1}
    \xrightarrow{i_{n-1}} \overline{Y_{n-2}} \xrightarrow{p_{n-2}} \cdots
    \xrightarrow{p_3} Y_{3}\xrightarrow{i_3} \overline{Y_{2}} \xrightarrow{p_{2}} Y_{2}
    \xrightarrow{i_{2}} Y_{1} = Y.
  \]
  Then $b_d^{(2)}(\widehat{X}^{\calc};\cald^{\calc}_{\widehat{X};F}) = 0$.
\end{theorem}
\begin{proof}
   We prove
  $b_d^{(2)}(\widehat{Y_k}^{\calc};\cald^{\calc}_{\widehat{Y_k};F}) = 0$ for
  $k = 1,2, \ldots, n$ by induction. The induction beginning $k = 1$ follows directly from
  the assumptions.  The induction step from $k\ge 1$ to $k+1 \le n$ is done as
  follows. Let $\widehat{p}_k \colon \widehat{Y_k} \to Y_k$ be the maximal residually
  $\calc$-covering.  As $p_k \colon \overline{Y_k} \to Y_k$ is a residually
  $\calc$-covering with
  connected total space, there exists a subgroup
  $G \subseteq \pi_1(Y_k)/\pi_1(\widehat{Y_k})$ such that $\overline{Y_k}$ can be
  identified with $\widehat{Y_{k+1}}/G$.  Consider the $G$-covering
  $q_k \colon \widehat{Y_k} \to \widehat{Y_k}/G = \overline{Y_k}$. Regard the pullback
\[
\xymatrix{i_{k+1}^* \widehat{Y_k} \ar[r] \ar[d]_{q_{k+1}} & \widehat{Y_k}
    \ar[d]^{q_k}
    \\
    Y_{k+1} \ar[r]_{i_{k+1}} & \overline{Y_k}.}
\]
As
$b_d^{(2)}(\widehat{Y_k}^{\calc};\cald^{\calc}_{\widehat{Y_k};F}) =
b_d^{(2)}(\widehat{Y_k}^{\calc};\cald_{F[\pi_1(Y_k)/\pi_1(\widehat{Y_k})]}) =0$ holds by the
induction hypothesis, we conclude
$b_d^{(2)}(\widehat{Y_k}^{\calc};\cald^{\calc}_{FG}) = 0$ from
Theorem~\ref{the:On_the_top_L2-Betti_number}.  Lemma~\ref{lem:restricting_to_a_subcomplex}
implies that $b^{(2)}_d(i_{k+1}^* \widehat{Y_k} ;\cald^{\calc}_{FG}) $ is trivial.  We
conclude from Lemma~\ref{lem:making_the_total_space_connected} that we can find a subgroup
$H \subseteq G$ and an $H$-covering $q_{k+1}' \colon C \to Y_{k+1}$ such that $C$ is
connected and $b_d^{(2)}(C;\cald_{FH}) = 0$ holds. Let
$\widehat{p_{k+1}} \colon \widehat {Y_{k+1}} \to Y_{k+1}$ be the maximal residually 
$\calc$-covering of $Y_{k+1}$. Then $H$ belongs to $\calRC$, since $H \subseteq G \subseteq
\pi_1(Y_k)/\pi_1(\widehat{Y_k})$ holds and
$\pi_1(Y_k)/\pi_1(\widehat{Y_k})$ is residually $\calc$.  Moreover, there is a subgroup $K$ of
$\pi_1(Y_{k+1})/\pi_1(\widehat{Y_{k+1}})$, an identification $C = \widehat {Y_{k+1}} /K$,
and an exact sequence $1 \to K \to \pi_1(Y_{k+1})/\pi_1(\widehat{Y_{k+1}}) \to H \to 1$.
Theorem~\ref{the:Monotonicity} implies
\[
b_d^{(2)}(\widehat{Y_k}^{\calc};\cald^{\calc}_{\widehat{Y_k};F})
= b_d^{(2)}(\widehat{Y_k}^{\calc};\cald_{F[\pi_1(Y_{k+1})/\widehat{Y_{k+1}}]}) \le
b_d^{(2)}(C;\cald_{FH}) = 0
\]
and hence $b_d^{(2)}(\widehat{Y_k}^{\calc};\cald^{\calc}_{\widehat{Y_k};F}) = 0$.
\end{proof}
  
Next we deal with the implication~\ref{que:Gromov_Wise:L2_Betti}
$\implies$~\ref{que:Gromov_Wise:tower} of Question~\ref{que:Gromov_Wise} in a special  case.

\begin{theorem}\label{the:vanishing_of_the_top_L2-Betti_numbers_and_RALI-towers_d_is_2}
  Let $F$ be any field. Let $\calc$ be a class of groups
  which is closed under taking subgroups and finite products and contained in $\calALI$.
  Let $Y$ be a connected finite $2$-dimensional $CW$-complex
  with $b_2^{(2)}(\widehat{Y}^{\calc};\cald^{\calc}_{\widehat{Y};F}) = 0$.  Consider any
  connected $\calRC$-tower
  with $Y$ as target in the sense of Definition~\ref{def:tower}
  \[
    X = Y_n \xrightarrow{i_n} \overline{Y_{n-1}} \xrightarrow{p_{n-1}} Y_{n-1}
    \xrightarrow{i_{n-1}} \overline{Y_{n-2}} \xrightarrow{p_{n-2}} \cdots
    \xrightarrow{p_3} Y_{3}\xrightarrow{i_3} \overline{Y_{2}} \xrightarrow{p_{2}} Y_{2}
    \xrightarrow{i_{2}} Y_{1} = Y.
  \]
  Then
  \begin{enumerate}
  \item\label{the:vanishing_of_the_top_L2-Betti_numbers_and_RALI-towers_d_is_2:vanishing_of:b_2}
   We have $b_2^{(2)}(\widehat{X}^{\calc};\cald^{\calc}_{\widehat{X};F}) = 0$;

 \item\label{the:vanishing_of_the_top_L2-Betti_numbers_and_RALI-towers_d_is_2:cases}
   Precisely one of  the follows statements  hold:

   \begin{enumerate}
   \item\label{the:vanishing_of_the_top_L2-Betti_numbers_and_RALI-towers_d_is_2:chi(X)_ge_0:(1)}
     $\pi_1(X) /\pi_1(\widehat{X}^{\calc})$ is trivial, $\chi(X) = 1$ and $b_1(X;F) = 0$;
     \item\label{the:vanishing_of_the_top_L2-Betti_numbers_and_RALI-towers_d_is_2:chi(X)_ge_0:(2)}
    $\pi_1(X) /\pi_1(\widehat{X}^{\calc})$ is trivial, $\chi(X) = 0 $ and $b_1(X;F) = 1$;
     \item\label{the:vanishing_of_the_top_L2-Betti_numbers_and_RALI-towers_d_is_2:chi(X)_ge_0:(3)}
    $\pi_1(X)/\pi_1(\widehat{X}^{\calc})$ is non-trivial and  $\chi(X) \le 0$.
  \end{enumerate}

\item\label{the:vanishing_of_the_top_L2-Betti_numbers_and_RALI-towers_d_is_2:alternative}
  Precisely one of  the follows statements  hold:
  \begin{enumerate}
  \item\label{the:vanishing_of_the_top_L2-Betti_numbers_and_RALI-towers_d_is_2:alternative:(1)}
    We have $\chi(X) \le 0$;
\item\label{the:vanishing_of_the_top_L2-Betti_numbers_and_RALI-towers_d_is_2:alternative:(2)}
    We have $\pi_1(X) /\pi_1(\widehat{X}^{\calc}) = \{1\}$, $\chi(X) = 1$, and $b_1(X;F) = 0$;
  \end{enumerate}

 \item\label{the:vanishing_of_the_top_L2-Betti_numbers_and_RALI-towers_d_is_2:pi_1(X)_is_RALI}
   Suppose that $\pi_1(X)$ belongs to $\calRC$. Then the tower is non-positive, i.e.,  either $\chi(X) \le 0$ or $X$
   is contractible.
 \end{enumerate}  
\end{theorem}
\begin{proof}~\ref{the:vanishing_of_the_top_L2-Betti_numbers_and_RALI-towers_d_is_2:vanishing_of:b_2}
  This follows from Theorem~\ref{the:vanishing_of_the_top_L2-Betti_numbers_and_RALI-towers}.
  \\[1mm]~\ref{the:vanishing_of_the_top_L2-Betti_numbers_and_RALI-towers_d_is_2:cases}
  We conclude from assertion~\ref{the:vanishing_of_the_top_L2-Betti_numbers_and_RALI-towers_d_is_2:vanishing_of:b_2}
  and Theorem~\ref{the:Euler-Poincare_formula}
  \[\chi(X) = b_0^{(2)}(\widehat{X}^{\calc};\cald^{\calc}_{\widehat{X};F})- b_1^{(2)}(\widehat{X}^{\calc};\cald^{\calc}_{\widehat{X};F}).
  \]
  Suppose $\pi_1(X)/\pi_1(X)/\pi_1(\widehat{X}_{\calc})$ is trivial.
  Then $b_n^{(2)}(\widehat{X}^{\calc};\cald^{\calc}_{\widehat{X};F}) =b_n(X;F)$ for $n \ge 0$, $b_0(X;F) = 1$,
  and $\chi(X) = b_0(X;F) - b_1(X;F) + b_2(X;F)$. This implies $\chi(X) =1 - b_1(X;F)$.
  Hence we have $\chi(X) = 1$ and $b_1(X;F) = 0$ or we have $\chi(X) = 0 $ and $b_1(X;F) = 1$.
  Suppose that $\pi_1(X)/\pi_1(\widehat{X}^{\calc})$ is non-trivial. Then the claim follows from
  Lemma~\ref{lem:vanishing_of_the_second_L2_Betti_number_and_coverings}.
  \\[1mm]~\ref{the:vanishing_of_the_top_L2-Betti_numbers_and_RALI-towers_d_is_2:alternative}
  This follows directly from assertion~\ref{the:vanishing_of_the_top_L2-Betti_numbers_and_RALI-towers_d_is_2:cases}.
  \\[1mm]~\ref{the:vanishing_of_the_top_L2-Betti_numbers_and_RALI-towers_d_is_2:pi_1(X)_is_RALI}
  Since $\pi_1(X)$ belongs to $\calRC$ by assumption,
  $\pi_1(\widehat{X}^{\calc})$ is trivial. Hence either $\pi_1(X)/\pi_1(\widehat{X}^{\calc})$ is non-trivial
  or $X$ is simply connected. We conclude from
  assertion~\ref{the:vanishing_of_the_top_L2-Betti_numbers_and_RALI-towers_d_is_2:cases}
  that $\chi(X) \le 0$ holds or $X$ is simply connected. If $X$ is simply connected,
  then it is contractible, since $X$ is a  simply connected $2$-dimensional $CW$-complex satisfying  $b_2(X;F) = 0$
  and hence $H_2(X;\IZ) = 0$.
\end{proof}

\begin{remark}\label{rem:conclusion_of_Theorem_ref(the:vanishing_of_the_top_L2-Betti_numbers_and_RALI-towers_d_is_2)}
  Note that Theorem~\ref{the:vanishing_of_the_top_L2-Betti_numbers_and_RALI-towers_d_is_2}~%
\ref{the:vanishing_of_the_top_L2-Betti_numbers_and_RALI-towers_d_is_2:pi_1(X)_is_RALI}
  gives a positive answer to the implication~\ref{que:Gromov_Wise:L2_Betti}
  $\implies$~\ref{que:Gromov_Wise:tower} of Question~\ref{que:Gromov_Wise} if we consider
  connected  $\RALI$-towers and assume that both $\pi_1(X)$ and $\pi_1(Y)$
  are \RALI-groups.
\end{remark}

\begin{remark}\label{rem:loc_ind_implies_weal_Lewin_yield_Wise}
  Suppose Conjecture~\ref{con:Locally_indicibale_groups_are_Lewin_groups} holds.
  Then our arguments show that the implication~\ref{con:Wise:L2_Betti}  $\implies$~\ref{con:Wise:tower_locally_indicable} 
  appearing in Conjecture~\ref{con:Wise} is true. Namely,  if Conjecture~\ref{con:Locally_indicibale_groups_are_Lewin_groups}
  holds, we can choose in Theorem~\ref{the:vanishing_of_the_top_L2-Betti_numbers_and_RALI-towers}
  the family $\calc$ to be the family of locally indicable groups.
  Hence for any tower with $Y$ as target
  \[
    X = Y_n \xrightarrow{i_n} \overline{Y_{n-1}} \xrightarrow{p_{n-1}} Y_{n-1}
    \xrightarrow{i_{n-1}} \overline{Y_{n-2}} \xrightarrow{p_{n-2}} \cdots
    \xrightarrow{p_3} Y_{3}\xrightarrow{i_3} \overline{Y_{2}} \xrightarrow{p_{2}} Y_{2}
    \xrightarrow{i_{2}} Y_{1} = Y
  \]
  such that the fundamental groups of each of the spaces and in particular  the one of $X$ are locally indicable,
  we conclude $b_2^{(2)}\bigl(\widetilde{X};\caln(\pi_1(X))\bigr) = 0$. 
 Lemma~\ref{lem:vanishing_of_the_second_L2_Betti_number_and_universal_coverings}
 extends to the family  of locally indicable groups instead of the class of $\RALI$-groups if
 Conjecture~\ref{con:Locally_indicibale_groups_are_Lewin_groups} holds.
 Hence $\chi(X) \le 0$ or $X$ is simply connected. If $X$ is simply connected, we conclude 
 $b_2(X;\IQ) = b_2(\widetilde{X};\IQ) = 0$ and hence that $H_2(X;\IZ)$ vanishes
 which implies  that $X$ is contractible.

 In view of Remark~\ref{rem:homotopy_invariance_and_Whitehead}
 this also implies that connected subcomplexes of finite contractible
  $2$-complexes are aspherical.
\end{remark}

Next we deal the class of groups  $\calRF_p$ of residually (finite $p$-groups) for a
prime $p$.  This is of course related to the class $\calf_p$ of finite $p$-groups. Since
this class is not contained in $\calRALI$, we have to use one of the adhoc definitions of
$L^2$-Betti numbers over an arbitrary field $F$ in this case, and we choose
$\beta^{\inf,p}_{n}(Y;F)$, see
Subsection~\ref{subsec:Approximation_by_subgroups_of_finite_index_and_Idenitfication_with_adhoc_definitions}.
Moreover, the Monotonicity Theorem~\ref{the:Monotonicity} does not hold anymore and we can
only use Lemma~\ref{lem:monotonicity_for_p-groups}. Note that
Lemma~\ref{lem:monotonicity_for_p-groups} is only available for subgroups of finite
$p$-power index and not just of finite index which is the reason why we will not be able
to treat the class $\calRF$ of residually finite groups and the related class of finite
groups. We often will call a (connected) $\calf_p$-tower just a \emph{(connected) $p$-tower} and a
$\calf_p$-covering just a \emph{p-covering}.

\begin{lemma}\label{lem:passing_to_connected_p-towers} Let $f \colon X \to Y$ be map of finite
  connected $CW$-complexes.  Let $p$ be a prime.  Consider any $\calRF_p$-tower with underlying map $f$.

  Then there is a connected $p$-tower with underlying map $f$.
\end{lemma}
  \begin{proof}
  Because of Remark~\ref{rem:passing_to_connected_towers} it
  suffices to consider connected $\calRF_p$-towers.  Now consider an $\calRF_p$-covering
  $q \colon \overline{Z} \to Z$ and a connected finite $CW$-subcomplex
  $A \subseteq \overline{Z}$. Then there exists a factorization of $q$ as the composite
  $q_1 \colon \overline{Z} \to \overline{\overline{Z}}$ and
  $q_2 \colon \overline{\overline{Z}} \to Z$ such that $ \overline{\overline{Z}}$ is
  connected, $q_2$ is a $p$-covering,
  and $q_0|_A \colon A \to \overline{\overline{Z}}$ is an embedding of the connected
  finite $CW$-complex $A$ into the finite connected $CW$-complex
  $\overline{\overline{Z}}$.   Now apply this construction successively to $Y_{i+1} \subseteq \overline{Y_i}$ and
  $\overline{Y_i} \to Y_i$  appearing in a connected $\calRF_p$-tower.
\end{proof}

\begin{proposition}\label{pro:vanishing_of_the_top_L2-Betti_numbers_beta_upper(inf,p_n)(Y;F)_and_calRF_p-towers}
  Let $p$ be a prime  and $d \in \IZ^{\ge 2}$. Consider connected finite $d$-dimensional $CW$-complexes $X$ and $Y$. Let $F$ be
  any field of characteristic $p$.  Suppose that there exists a  $\calRF_p$-tower with $X$ as
  source and $Y$ as target.  Assume  that $\beta^{\inf,p}_{d}(Y;F)$ vanishes.

  Then

  \begin{enumerate}
  \item\label{pro:vanishing_of_the_top_L2-Betti_numbers_beta_upper(inf,p_n)(Y;F)_and_calRF_p-towers:beta_upper_(inf,p)_d(X;F)_is_0}
    We have $\beta^{\inf,p}_{d}(X;F) = 0$;

  \item\label{pro:vanishing_of_the_top_L2-Betti_numbers_beta_upper(inf,p_n)(Y;F)_and_calRF_p-towers:chi}
    Either we have $\chi(X) \le 0$, or   we have $b_1(X;F) = b_2(X;F) = 0$ (and hence $\chi(X) = 1$).
  \end{enumerate}
\end{proposition}
\begin{proof}~\ref{pro:vanishing_of_the_top_L2-Betti_numbers_beta_upper(inf,p_n)(Y;F)_and_calRF_p-towers:beta_upper_(inf,p)_d(X;F)_is_0}
  We can assume without loss of generality $F = \IF_p$ by
  Lemma~\ref{lem:cald_(FG)_is_cald(FG_subseteq_calu(G))}~%
\ref{lem:cald_(FG)_is_cald(FG_subseteq_calu(G)):dependence_on_F}.  In view of
  Lemma~\ref{lem:passing_to_connected_p-towers} we can assume without loss of
  generality that there exists a connected $p$-tower with $X$ is a source and $Y$ as
  target. Moreover, by induction over the length of the tower we only have to consider an
  elementary connected $p$-tower, i.e., we have a $p$-covering $q \colon \overline{Y} \to Y$
  with a connected finite $CW$-complex $\overline{Y}$ as total space and an inclusion $i \colon X \to \overline{Y}$
  of connected finite   $CW$-complexes.

  Fix $\delta > 0$. The assumption $\beta^{p-\inf}_2(Y;\mathbb F_p)=0$
  implies that  there is a $p$-covering $q' \colon Y'\rightarrow Y$ with
  \[
    \frac{b_2(Y';\mathbb F_p)}{|Y'\rightarrow Y|} \le  \frac{\delta}{|\overline{Y} \rightarrow Y|}.
    \]
    Consider  the following commutative  diagram
  whose squares are pullbacks
  \[\xymatrix{X' \ar[r]^{q''} \ar[d]^{\overline{i}}
      &
      X\ar[d]^i
      \\
      \overline{Y}' \ar[r]^{\overline{q'}} \ar[d]^{\overline{q}}
      &
      \overline{Y} \ar[d]^q
      \\
      Y' \ar[r]^{q'}
      &
      Y.
    }
  \]
  Note  that all horizontal maps are $p$-coverings of the same degree. We conclude
  \begin{eqnarray*}
    \frac{b_2(\overline{Y'};\mathbb F_p) }{|\overline{X'} \rightarrow X|}
    & = & 
    \frac{b_2(\overline{Y'};\mathbb F_p) }{|\overline{Y'} \rightarrow \overline{Y}|}
    \\
    & = &
          |\overline{Y} \rightarrow Y| \cdot   \frac{b_2(\overline{Y'};\mathbb F_p) }{|\overline{Y}
          \rightarrow Y| \cdot |\overline{Y'} \rightarrow \overline{Y}|}
    \\
    & = &
    |\overline{Y} \rightarrow Y| \cdot   \frac{b_2(\overline{Y'};\mathbb F_p)}{|\overline{Y'} \rightarrow \overline{Y}|}
    \\
    & \stackrel{\textup{Lemma}~\ref{lem:monotonicity_for_p-groups}}{\le} &
    |\overline{Y} \rightarrow Y| \cdot   \frac{b_2(Y';\mathbb F_p) }{|Y' \rightarrow Y|}
    \\
    & \le  &
    |\overline{Y} \rightarrow Y| \cdot  \frac{\delta}{|\overline{Y} \rightarrow Y|}
    \\
    & = &
          \delta.
  \end{eqnarray*}

  Since $X$ is a subcomplex of $Y$ and $X$ and $Y$ have the same dimension $d$,  $X'$ is a
  subcomplex of $\overline{Y}'$ and $X'$ and $\overline{Y'}$ have the same dimension $d$.
  This implies that the kernel of the $d$-th differential of the cellular chain complex
  with $\IF_p$-coefficients of $X'$ is contained in the kernel of the $d$-th differential
  of the cellular chain complex with $\IF_p$-coefficients of $\overline{Y}'$. Hence we
  get $b_2(X';\mathbb F_p) \le b_2(\overline{Y}';\mathbb F_p)$. We conclude
  \[
    \frac{b_2(X';\mathbb F_p)}{|X' \to X|}
    \le 
    \frac{b_2(\overline{Y}';\mathbb F_p)}{|X' \to X|}
    \le \delta.
  \]
  This implies that $\beta^{\inf,p}_{n}(X;F)$ vanishes.
  \\[1mm]~\ref{pro:vanishing_of_the_top_L2-Betti_numbers_beta_upper(inf,p_n)(Y;F)_and_calRF_p-towers:chi}
  Suppose that $d = 2$. Since $\beta^{\inf,p}_{d}(X;F) = 0$ holds by assumption, there
  exists a $p$-covering $q \colon \overline{X} \to X$ with finite total space
  $\overline{Y}$ satisfying
  \[
    \frac{b_2(\overline{X};\IF_p)}{|\overline{Y} \to Y|} < \frac{1}{p}.
  \]
   Suppose that $|\overline{X} \to X| \not= 1$. Then we get 
  \begin{eqnarray*}
    \chi(X)
    & = &
    \frac{\chi(\overline{X})}{|\overline{X} \to X|}
    \\
    & = & 
    \frac{b_0(\overline{X};\IF_p) - b_1(\overline{X};\IF_p) + b_2(\overline{X};\IF_p)}{|\overline{X} \to X|}
    \\
    & = &
    \frac{1 - b_1(\overline{X};\IF_p)}{|\overline{X} \to X|}   + \frac{b_2(\overline{X};\IF_p)}{|\overline{X} \to X|}
    \\
    & < &
     \frac{1}{p}  +  \frac{1}{p}
    \\
    & \le 1
  \end{eqnarray*}
  and hence $\chi(X) = 0$.

  Suppose that $|\overline{X} \to X| =  1$.
  Then $\overline{X} = X$. We conclude from
  \[
  b_2(X;\IF_p) = \frac{b_2(\overline{X};\IF_p)}{|\overline{Y} \to Y|} < \frac{1}{p}
  \]
  that $b_2(\overline{X};\IF_p) = 0$ holds. We get
  \[
    \chi(X) = b_0(X;\IF_p) -  b_1(X;\IF_p) = 1 -   b_1(X;\IF_p).
  \]
  Hence either $\chi(X) \le  0$  or we have both $\chi(X) = 1$ and $b_1(X;\IF_p) = 0$.  
\end{proof}

\begin{remark}\label{rem;addendum}
  Consider the situation of
  Proposition~\ref{pro:vanishing_of_the_top_L2-Betti_numbers_beta_upper(inf,p_n)(Y;F)_and_calRF_p-towers}. 
  If we additionally assume that $X$ is simply connected or $\pi_1(X)$ admits a surjection
  onto $\IZ/p$, then the tower is non-positive, i.e, $\chi(X) \le 0$ or $X$ is
  contractible.
\end{remark}

  % ----------------------------------------------------------------------------

\subsection{From towers to $L^2$-Betti numbers}\label{subsec:From_towers_toL2-Betti_numbers}

In this section we deal with the implication~\ref{que:Gromov_Wise:tower} $\implies$~\ref{que:Gromov_Wise:L2_Betti}
of Question~\ref{que:Gromov_Wise}. In this context the following question is relevant.

\begin{question}\label{que:other_implication_Wise_implies_L2}
  Let $Y$ be a finite connected $2$-dimensional $CW$-complex $Y$. Does
  $b_n^{(2)}(\widetilde{Y}) = 0$ hold if, for any connected finite $CW$-subcomplex
  $Z \subseteq \widetilde{Y}$, either we have $\chi(Z) \le 0$ or $Z$ is contractible?
\end{question}

Here is a very special case where the answer is yes.

\begin{lemma}\label{the:other_implication_if_1_upper_(2)_is_zero}
    Let $Y$ be a connected finite $2$-dimensional CW-complex such that
    $b_1^{(2)}(\pi_1(Y) )= 0$ holds.

    Then the following statements are equivalent:

    \begin{enumerate}
    \item\label{the:other_implication_if_1_upper_(2)_is_zero:vanishing_of_b_2_upper_(2)}
    $b_2^{(2)}(\widetilde{Y}) = 0$;
    \item\label{the:other_implication_if_1_upper_(2)_is_zero:chi_is_0}
          $\chi(Y) =  0$ or $Y$ is contractible;
      \item\label{the:other_implication_if_1_upper_(2)_is_zero:chi_le_0}
        $\chi(Y) \le 0$ or $Y$ is contractible.        
        \end{enumerate}
      \end{lemma}
  \begin{proof} We obtain $b_1^{(2)}(\widetilde{Y}) = b_1^{(2)}(\pi_1(Y)) = 0$
    from~\cite[Theorem~1.35~(1) on page~37]{Lueck(2002)} since the classifying map
    $\widetilde{Y} \to E\pi_1(Y)$ is $2$-connected. Hence
    \[
      b^{(2)}_2(\widetilde{Y}) + b^{(2)}_0(\widetilde{Y}) = \chi(Y)
    \]
    holds by~\cite[Theorem~1.35~(2) on page~37]{Lueck(2002)}. We have
    $b_0^{(2)}(\widetilde{Y}) = 0$ if $\pi_1(Y)$ is infinite and
    $b_0^{(2)}(\widetilde{Y}) = |\pi_1(Y)|^{-1}$, if $\pi_1(Y)$ is finite, see
    Theorem~\ref{the:zeroth_L2-Betti_number}.
    
    If $b^{(2)}_2(\widetilde{Y}) = 0$, then $b^{(2)}_0(\widetilde{Y}) = \chi(Y)$ and hence
    we have $\chi(Y) = 0$ or $\pi_1(Y) = \{1\}$. As $Y$ is $2$-dimensional, $Y$ is
    contractible if and only if $\pi_1(Y) = \{1\}$ and $b^{(2)}_2(\widetilde{Y}) = 0$
    hold.  This shows~\ref{the:other_implication_if_1_upper_(2)_is_zero:vanishing_of_b_2_upper_(2)}~$\implies$~%
\ref{the:other_implication_if_1_upper_(2)_is_zero:chi_is_0}.
  Obviously we have~\ref{the:other_implication_if_1_upper_(2)_is_zero:chi_is_0}~$\implies$~%
\ref{the:other_implication_if_1_upper_(2)_is_zero:chi_le_0}. If $\chi(Y) \le 0$, then
  $b^{(2)}_2(\widetilde{Y}) + b^{(2)}_0(\widetilde{Y}) = \chi(Y)$ implies
  $b^{(2)}_2(\widetilde{Y}) = 0$, $\chi(Y) = 0$, and $|\pi_1(Y)| = \infty$.
  Hence~\ref{the:other_implication_if_1_upper_(2)_is_zero:chi_le_0}~$\implies$~%
\ref{the:other_implication_if_1_upper_(2)_is_zero:vanishing_of_b_2_upper_(2)}.
\end{proof}

The class of groups $G$ satisfying $b_1^{(2)}(G) = 0$ is studied in~\cite[Theorem~7.2 on
page~294]{Lueck(2002)}.

    \begin{example}\label{exa:one_2_cell}
      Let $Y$ be a connected finite $2$-dimensional $CW$-complex with precisely one
      $2$-cell and torsionfree fundamental group. Then either
      $b_2^{(2)}(\widetilde{Y}) = 0$ or $Y$ is homotopic to
      $S^2 \vee \bigvee_{i = 1} ^n S^1$. This follows from the fact that the Atiyah
      Conjecture is known for torsionfree one-relator groups,
      see~\cite{Jaikin-Zapirain+Lopez-Alvarez(2020)}, and hence either
      $b_2^{(2)}(\widetilde{Y}) = 0$ or $b_2^{(2)}(\widetilde{Y}) = 1$ holds, and
      $b_2^{(2)}(\widetilde{Y}) = 1$ implies the vanishing of the second differential of
      the cellular $\IZ[\pi_1(X)]$-chain complex $C_*(\widetilde{Y})$.
    \end{example}

    The results in this Subsection~\ref{subsec:From_towers_toL2-Betti_numbers} hold also
    in prime characteristic if one assumes that the groups under consideration are
    \RALI-groups. However, if Conjecture~\ref{con:_dimension_1} is true, then we get for
    any finite $2$-dimensional $CW$-complex $X$ whose fundamental group $\pi$ is a \RALI-group
    that $b_n^{(2)}(\widetilde{X} ;\caln(\pi)) = b_n^{(2)}(\widetilde{X};\cald_{F\pi})$ holds for every
    $n \in \IZ^{\ge 0}$ and any field $F$.

 %%%%%%%%%%%%%%%%%%%%%%%%%%%%%%%%%%%%%%%%%%%%%%%%%%%%%%%%%%%%%%%%%%%%%
%%%%%%%%%%%%%%%%%%%%%%%%%%%%%% Section 6 %%%%%%%%%%%%%%%%%%%%%%%%%%%%%%%%
%%%%%%%%%%%%%%%%%%%%%%%%%%%%%%%%%%%%%%%%%%%%%%%%%%%%%%%%%%%%%%%%%%%%%

    \typeout{---------- Section 6: Finitely generated grou with virtual cohomological dimension $2$
    and vanishing second $L^2$-Betti number}%

  \section{Finitely   generated   group  with   cohomological
    dimension $2$ and vanishing second $L^2$-Betti number}%
  \label{sec:Fin_gen_groups_with_cd_2_and_b_2_upper(2)_is_0}

  \begin{question}\label{que:Fin_gen_groups_with_cd_2_and_b_2_upper_(2)_is_0}
    Let $G$ be a finitely generated torsionfree group such that 
    $\operatorname{cd}(G) \le 2$  and  $b_2^{(2)}(G) = 0$ hold. Which of the following statements are true?
    \begin{enumerate}
    \item\label{sec:Fin_gen_groups_with_cd_2_and_b_2_upper_(2)_is_0:b_1(G)_is_0}
    The first Betti number of $G$ satisfies $b_1(G) \ge 1$;
  \item\label{sec:Fin_gen_groups_with_cd_2_and_b_2_upper_(2)_is_0:locally_indicable}
    The group $G$ is locally indicable;
  \item\label{sec:Fin_gen_groups_with_cd_2_and_b_2_upper_(2)_is_0:FF_2}
    The group $G$ is of type $\operatorname{FF}_2$;
  \item\label{sec:Fin_gen_groups_with_cd_2_and_b_2_upper_(2)_is_0:finitely_presented}
    The group $G$ is finitely presented and has a finite $2$-dimensional model for $BG$.
    
    \end{enumerate}
  \end{question}

  This is related to the following question.

  \begin{question}\label{que:projective_finite_dimension}
    Let $G$ be a torsionfree group and let $P$ be a projective $\IZ G$-module.
    Suppose that $\dim_{\caln(G)}(\caln(G) \otimes_{\IZ G} P) < \infty$ holds. Is then $P$
    finitely generated?
  \end{question}

  Here is a partial answer to Question~\ref{que:projective_finite_dimension}.

  \begin{theorem}\label{the:fin_gen_and_finite_dim}
    Let $G$ be a group. Assume that there is a positive integer $d$
    such that the order of any finite subgroup of $G$ is less than or equal to $d$ and
    that $G$ satisfies the Atyiah Conjecture over the integral group
    ring $\IZ G$.
  
    Let $P$ be a projective $\IZ G$-module. Then the following
    assertions are equivalent:

  \begin{enumerate}
  \item\label{the:fin_gen_and_finite_dim:fin_gen} The $\IZ G$-module $P$ is finitely generated;
    
  \item\label{the:fin_gen_and_finite_dim:fin_dim} We have $\dim_{\caln(G)}(\caln(G) \otimes_{\IZ G} P) < \infty$;

  \end{enumerate}
\end{theorem}

It is a direct consequence of the following two lemmas,
whose proof can be found
in~\cite[Lemma~4 and Lemma~5]{Kropholler-Linnell-Lueck(2009)}.

\begin{lemma}\label{lem:_caln(G)_otimes_P/M_has_zero_dimension}
Let $A$ be a ring with $\IZ \subset A \subset \IC$.
Suppose that there is a positive integer $d$ such that the order of any finite subgroup
of $G$ is less than or equal to $d$ and that the Atiyah Conjecture holds for $A$ and $G$. Let $N$
be a $AG$-module. Suppose that $\dim_{\caln(G)}(\caln(G) \otimes_{AG} N) < \infty$.

Then there exists
a finitely generated $AG$-submodule $M \subset N$ with the property
$\dim_{\caln(G)}(\caln(G) \otimes_{AG} N/M) = 0$.
\end{lemma}

\begin{lemma}\label{lem:_projectives_in_finitely_generated_free_modules}
Let $A$ be a ring with $\IZ \subset A \subset \IC$.
Let $P$ be a projective $AG$-module
such that for some finitely generated $AG$-submodule $M \subset P$ satisfying 
$\dim_{\caln(G)}(\caln(G)\otimes_{AG} P/M) = 0$.

Then $P$ is finitely generated.
\end{lemma}

 The next lemma gives some insight about  Question~\ref{que:Fin_gen_groups_with_cd_2_and_b_2_upper_(2)_is_0}.

 \begin{lemma}\label{lem:FP_2_and_b_2_upper_(2)-finite} Consider a natural number $d \ge 1$.
   Let $G$ be a group of cohomological dimension $\le d$ such that there is a model for
   $BG$ with finite $(d-1)$-skeleton.  Suppose that $G$ satisfies the strong Atyiah
   Conjecture. Then:
   \begin{enumerate}
   \item\label{lem:FP_2_and_b_2_upper_(2)-finite:type_FP_d}
   The following assertions are equivalent:

   \begin{enumerate}
   \item\label{lem:FP_2_and_b_2_upper_(2)-finite:type_FP_d:Betti}
   The $d$-th $L^2$-Betti number $b_d^{(2)}(G)$ is finite;

 \item\label{lem:FP_2_and_b_2_upper_(2)-finite:type_FP_d:FP_2:BG} The group $G$ is of type
   $\operatorname{FP}_d$, i.e., there is a finitely generated projective $d$-dimensional
   resolution of the trivial $\IZ G$-module $\IZ$;

 \end{enumerate}
 
\item\label{lem:FP_2_and_b_2_upper_(2)-finite:type_FF_d} Suppose that $G$ satisfies the Full
  Farrell-Jones Conjecture. (The  statement of the Full Farrell-Jones Conjecture is quite complicated and involves
  $L$-groups as well as $K$-theory and the precise formulation is not relevant for this paper,
  see~\cite[Conjecture13.30 on page~387]{Lueck(2025book)}.
  For a status report and a list of applications we refer to~\cite[Theorem~16.1 on page~481
  and Theorem~13.65 on page~405]{Lueck(2025book)}.)  Then one can replace in
  assertion~\ref{lem:FP_2_and_b_2_upper_(2)-finite:type_FP_d} the condition
  $\operatorname{FP}_d$ by the condition $\operatorname{FF}_d$, i.e., there is a finitely
  generated free $d$-dimensional resolution of the trivial $\IZ G$-module $\IZ$;

\item\label{lem:FP_2_and_b_2_upper_(2)-finite:BG} Suppose $d \ge 3$ and $G$ satisfies the Full
  Farrell-Jones Conjecture.
  
Then the following assertions are equivalent:

   \begin{enumerate}

   \item\label{lem:FP_2_and_b_2_upper_(2)-finite:BG:BG_finite}
    There is a finite $d$-dimensional model for $BG$;
     
   \item\label{lem:FP_2_and_b_2_upper_(2)-finite:BG:Betti}
     The $d$-th $L^2$-Betti number $b_d^{(2)}(G)$ is finite.

   \end{enumerate}
    \end{enumerate}
  \end{lemma}
  \begin{proof}~\ref{lem:FP_2_and_b_2_upper_(2)-finite:type_FP_d:Betti}
    We begin with the implication~\ref{lem:FP_2_and_b_2_upper_(2)-finite:type_FP_d:Betti}~$\implies$~%
\ref{lem:FP_2_and_b_2_upper_(2)-finite:type_FP_d:FP_2:BG}.
    Since there is a model of $BG$ with finite $(d-1)$-skeleton, we can find a projective $\IZ G$-resolution
    \[
    \cdots \to P_{d+1} \to P_d \to P_{d-1} \to \cdots \to  P_1 \to P_0 \to \IZ \to 0
  \]
  of the trivial $\IZ G$-module $\IZ$ such that $P_i$ is finitely generated free for
  $i \le (d-1)$.  Since $G$ has cohomological dimension $d$, the group $G$ is torsionfree
  and we can arrange $P_i = 0$ for $i \ge (d+1)$.  The $d$-th $L^2$-Betti number of the
  $d$-dimensional $\caln(G)$-chain complex $\caln(G) \otimes_{\IZ G} P_*$ agrees with
  $b_d^{(2)}(G) = \dim_{\caln(G)}(H_d(\caln(G) \otimes_{\IZ G} P_*))$ and hence is finite.
  Since the $\caln(G)$-chain module $\caln(G) \otimes_{\IZ G} P_i$ is finitely generated
  free and hence satisfies $\dim_{\caln(G)}(\caln(G) \otimes_{\IZ G} P_i) < \infty$, we
  conclude from the additivity of the dimension function $\dim_{\caln(G)}$ that
  $\dim_{\caln(G)}(\caln(G) \otimes_{\IZ G} P_d) < \infty$ holds.
  Theorem~\ref{the:fin_gen_and_finite_dim} implies that $P_d$ is finitely generated
  projective.  Hence $P_*$ is a finitely generated projective $d$-dimensional resolution
  of of the trivial $\IZ G$-module $\IZ$, in other words, $G$ is of type
  $\operatorname{FP}_d$.  The other
  implication~\ref{lem:FP_2_and_b_2_upper_(2)-finite:type_FP_d:FP_2:BG}~$\implies$~%
\ref{lem:FP_2_and_b_2_upper_(2)-finite:type_FP_d:Betti} is obvious.
  \\[1mm]~\ref{lem:FP_2_and_b_2_upper_(2)-finite:type_FP_d:FP_2:BG} Since $G$ satisfies
  the Full Farrell-Jones Conjecture, we can find for every finitely generated projective
  $\IZ G$-module $Q$ natural numbers $m,n$ with $Q \oplus \IZ G^m \cong_{\IZ G} \IZ
  G^n$. Hence the conditions $\operatorname{FP}_d$ and $\operatorname{FF}_d$ are
  equivalent.  \\[1mm]~\ref{lem:FP_2_and_b_2_upper_(2)-finite:BG} Suppose $d \ge 3$. Then
  there is a finite $d$-dimensional model for $BG$ if and only if there is a model for
  $BG$ with finite $(d-1)$-skeleton and $G$ is of type $\operatorname{FF}_d$.
\end{proof}

The next result is taken from~\cite[Theorem~3]{Kropholler-Linnell-Lueck(2009)}.

\begin{theorem}\label{the_finitely_generated_elementary_amenable_groups_with_cd_le_2_and_2-dim_res}
Let $G$ be an elementary amenable group of cohomological dimension $\le
2$. Then

\begin{enumerate}

\item\label{the_finitely_generated_elementary_amenable_groups_with_cd_le_2_and_2-dim_res:fin_gen}
Suppose that $G$ is finitely generated.  Then $G$  possesses a presentation of the form
\[
  \langle x,y \mid yxy^{-1} = x^n \rangle.
\]
In particular there is a finite $2$-dimensional model for $BG$ and $b_1(G) \ge 1$;

\item\label{the_finitely_generated_elementary_amenable_groups_with_cd_le_2_and_2-dim_res:not_fin_gen}
Suppose that $G$ is countable but not finitely generated. Then $G$ is a non-cyclic subgroup of
the additive group $\IQ$.

\end{enumerate}
\end{theorem}

In particular we see  that an elementary amenable group of cohomological dimension $\le
2$ is locally indicable and coherent, i.e., every finitely generated subgroup is finitely presented.

Lemma~\ref{lem_type_FF_2} has already been proved by
  Jaikin-Zapirain-Linton~\cite[Theorem~3.10]{Jaikin-Zapirain-Linton(2025)}. 

\begin{lemma}\label{lem_type_FF_2}
  Let $G$ be a group of cohomological dimension $\le 2$ with $b_2^{(2)}(G) = 0$.  Suppose
  that $G$ satisfies the strong Atiyah Conjecture.  Then $G$ is almost coherent in the
  sense that every finitely generated subgroup $H \subseteq G$ is of type
  $\operatorname{FP}_2$.
  
  If $G$ satisfies the Full Farrell-Jones Conjecture, then we can replace
  $\operatorname{FP}_2$ by $\operatorname{FF}_2$ in the statement above.
\end{lemma}
\begin{proof}
  Let $H \subseteq G$ be a finitely generated subgroup. We conclude from
  Theorem~\ref{the:On_the_top_L2-Betti_number} that $b_2^{(2)}(H) = 0$.  Since $G$ is of
  cohomological dimension $\le 2$ and satisfies the strong Atiyah Conjecture, the same is
  true for $H$.
  Lemma~\ref{lem:FP_2_and_b_2_upper_(2)-finite}~\ref{lem:FP_2_and_b_2_upper_(2)-finite:type_FP_d}
  implies that $H$ is of type $\operatorname{FP}_2$.

  If $G$ satisfies the Full Farrell-Jones Conjecture, then also $H$ satisfies the Full
  Farrell-Jones Conjecture and hence is of type $\operatorname{FF}_2$.
\end{proof}

  \begin{remark}\label{rem:Bestvina_Brady}
  In assertion~\ref{lem:FP_2_and_b_2_upper_(2)-finite:BG} of
  Lemma~\ref{lem:FP_2_and_b_2_upper_(2)-finite} the condition $d \ge 3$ is necessary,
  since there are finitely generated groups $G$ of type $\operatorname{FP}_2$ which are
  not finitely presented, see~\cite{Bestvina-Brady(1997)}.
  Actually, the infinite group $G$ is the kernel of
  an epimorphism $f \colon A \to \IZ$ for some right angled Artin group $A$. Hence $G$
  satisfies the Full Farrell-Jones Conjecture and the strong Atiyah Conjecture and is even of
  type $\operatorname{FF}_2$.  
    So the group $G$ has the following properties:
    \begin{itemize}
    \item $G$ is finitely generated but not finitely presented;
    \item $G$ satisfies the Full Farrell-Jones Conjecture and the strong Atiyah Conjecture;
    \item $G$ is of type $\operatorname{FF}_2$.
    \end{itemize}

    However, $b_n^{(2)}(G) \not= 0$ by~\cite[Theorem~C]{Fisher-Hughes-Leary(2024)}.
    So it is still possible that any group $G$, whose
    virtually cohomological dimension is $\le 2$ and which satisfies $b_n^{(2)}(G) = 0$,
    is coherent.
  \end{remark}

%%%%%%%%%%%%%%%%%%%%%%%%%%%%%%%%%%%%%%%%%%%%%%%%%%%%%%%%%%%%%%%%%%%% 
%%%%%%%%%%%%%%%%%%%%%%%%%%%% Reference  %%%%%%%%%%%%%%%%%%%%%%%%%%%%%%%%
%%%%%%%%%%%%%%%%%%%%%%%%%%%%%%%%%%%%%%%%%%%%%%%%%%%%%%%%%%%%%%%%%%%%

\typeout{----------------------------- References ------------------------------}

\addcontentsline{toc<<}{section}{References} 
% \bibliographystyle{abbrv}
% \bibliography{dbpub,dbpre}

%\version{31.03.2026 (Wolfgang)}

\end{document}